\documentclass[11pt]{article}
\usepackage{amsfonts}
\usepackage[leqno]{amsmath}
\usepackage{graphicx}
\usepackage{hyperref}
\hypersetup{
	colorlinks=true,
	linkcolor=blue,
	filecolor=magenta,      
	urlcolor=blue,
	citecolor=red
}
\usepackage{latexsym}
\usepackage{amsmath,amsfonts,amssymb,amsthm,mathrsfs,euscript,makeidx,color}
\usepackage{bm}
\usepackage{enumerate}
\allowdisplaybreaks[1]
\usepackage{multirow}

\def\sqr#1#2{{\vcenter{\vbox{\hrule height.#2pt
				\hbox{\vrule width.#2pt height#1pt \kern#1pt \vrule width.#2pt}
				\hrule height.#2pt}}}}

\def\3n{\negthinspace \negthinspace \negthinspace }
\def\2n{\negthinspace \negthinspace }
\def\1n{\negthinspace }
\def\bel{\begin{equation}\label}
	\def\eel{\end{equation}}

\def\dbE{\mathbb{E}}
\def\dbF{\mathbb{F}}

\def\dbH{\mathbb{H}}

\def\dbN{\mathbb{N}}

\def\dbP{\mathbb{P}}

\def\dbR{\mathbb{R}}
\def\dbS{\mathbb{S}}

\def\sU{\mathscr{U}}

\def\sX{\mathscr{X}}

\def\={\buildrel \triangle \over =}

\def\d{\delta}

\def\m{\mu}

\def\th{\theta}

\def\Th{\Theta}

\def\O{\Omega}

\def\cF{{\cal F}}

\def\cL{{\cal L}}
\def\cM{{\cal M}}

\def\cO{{\cal O}}
\def\cP{{\cal P}}
\def\cQ{{\cal Q}}
\def\cR{{\cal R}}
\def\cS{{\cal S}}

\def\cW{{\cal W}}

\def\BA{{\bm A}}
\def\BB{{\bm B}}
\def\BC{{\bm C}}
\def\BD{{\bm D}}

\def\BP{{\bm P}}
\def\BQ{{\bm Q}}
\def\BR{{\bm R}}
\def\BS{{\bm S}}

\def\BU{{\bm U}}

\def\BX{{\bm X}}

\def\BZ{{\bm Z}}
\def\BGamma{{\bm \Gamma}}

\def\Bb{{\bm b}}

\def\Be{{\bm e}}

\def\Bk{{\bm k}}

\def\Bp{{\bm p}}
\def\Bq{{\bm q}}
\def\Br{{\bm r}}

\def\Bu{{\bm u}}

\def\Bx{{\bm x}}

\def\Th{\Theta}

\def\O{\Omega}

\def\ms{\medskip}

\def\hb{\hbox}

\def\limsup{\mathop{\overline{\rm lim}}}

\def\lan{{\langle}}
\def\ran{{\rangle}}

\def\h{\widehat}
\def\wt{\widetilde}

\def\cd{\cdot}

\def\les{\leqslant}
\def\ges{\geqslant}

\def\({\Big (}
\def\){\Big )}
\def\[{\Big[}
\def\]{\Big]}
\def\lan{\langle}
\def\ran{\rangle}

\def\bde{\begin{definition}\label}
	\def\ede{\end{definition}}

	\def\bel{\begin{equation}\label}
		\def\ee{\end{equation}}
	\def\bt{\begin{theorem}\label}
		\def\et{\end{theorem}}
	\def\bc{\begin{corollary}\label}
		\def\ec{\end{corollary}}
	\def\bl{\begin{lemma}\label}
		\def\el{\end{lemma}}
	\def\bp{\begin{proposition}\label}
		\def\ep{\end{proposition}}
	\def\bex{\begin{example}\label}
		\def\ex{\end{example}}
	\def\bas{\begin{assumption}}
		\def\eas{\end{assumption}}
	\def\br{\begin{remark}\label}
		\def\er{\end{remark}}
	\def\ba{\begin{array}}
		\def\ea{\end{array}}
	\def\ed{\end{document}}

\def\square#1{\vbox{\hrule\hbox{\vrule height#1%
			\kern#1\vrule}\hrule}}
\def\rectangle#1#2{\vbox{\hrule\hbox{\vrule height#1%
			\kern#2\vrule}\hrule}}

\font\tenbb=msbm10 \font\sevenbb=msbm7 \font\fivebb=msbm5

\newfam\bbfam
\scriptscriptfont\bbfam=\fivebb \textfont\bbfam=\tenbb
\scriptfont\bbfam=\sevenbb

\newtheorem{theorem}{Theorem}[section]
\newtheorem{corollary}[theorem]{Corollary}

\newtheorem{lemma}[theorem]{Lemma}
\newtheorem{proposition}[theorem]{Proposition}
\newtheorem{assumption}[theorem]{Assumption}

\theoremstyle{definition}
\newtheorem{definition}[theorem]{Definition}

\newtheorem{remark}[theorem]{Remark}

\newtheorem{example}{Example}[section]

\makeatletter

\@addtoreset{equation}{section}
\makeatother

\usepackage{xcolor}
\makeatletter

\input pdf-trans
\newbox\qbox
\def\usecolor#1{\csname\string\color@#1\endcsname\space}
\newcommand\bordercolor[1]{\colsplit{1}{#1}}
\newcommand\fillcolor[1]{\colsplit{0}{#1}}
\newcommand\outline[1]{\leavevmode%
	\def\maltext{#1}%
	\setbox\qbox=\hbox{\maltext}%
	\boxgs{Q q 2 Tr \thickness\space w \fillcol\space \bordercol\space}{}%
	\copy\qbox%
}
\makeatother
\newcommand\colsplit[2]{\colorlet{tmpcolor}{#2}\edef\tmp{\usecolor{tmpcolor}}%
	\def\tmpB{}\expandafter\colsplithelp\tmp\relax%
	\ifnum0=#1\relax\edef\fillcol{\tmpB}\else\edef\bordercol{\tmpC}\fi}
\def\colsplithelp#1#2 #3\relax{%
	\edef\tmpB{\tmpB#1#2 }%
	\ifnum `#1>`9\relax\def\tmpC{#3}\else\colsplithelp#3\relax\fi
}
\bordercolor{black}
\fillcolor{white}
\def\thickness{.3}

\title{Uniform-in-time convergence and turnpike properties of linear-quadratic mean field control problems with common noise}

\author{Erhan Bayraktar
\thanks{Department of Mathematics, University of Michigan, Ann Arbor, MI 48109, USA. {\tt erhan@umich.edu.} This author is funded in part by the National Science Foundation under grant DMS-2507940 and in part by the Susan M. Smith Professorship.}
~~~~Jiamin Jian
\thanks{Department of Mathematics, University of Michigan, Ann Arbor, MI 48109, USA. {\tt jiaminj@umich.edu.}}
}

\date{}

\begin{document}

\maketitle


\begin{abstract}
We investigate uniform-in-time convergence and turnpike properties for linear-quadratic mean field control problems with common noise. Within a unified framework, we analyze a finite-horizon social optimization problem, its mean field control limit, and the corresponding ergodic mean field control problem. The finite-horizon problems are characterized by coupled Riccati differential equations, whereas the ergodic problem is addressed via a Bellman equation on the Wasserstein space, which reduces to a system of stabilizing algebraic Riccati equations. By deriving estimates for these Riccati systems, we establish a turnpike property for the finite-horizon mean field control problem and obtain quantitative convergence results from the social optimization problem to its mean field limit and the associated ergodic control problem.
\end{abstract}


\section{Introduction}

Large-population stochastic systems arise in a wide range of applications where a large number of interacting agents are exposed to both idiosyncratic sources of uncertainty and common external shocks. When the agents act cooperatively, a natural modeling paradigm is social optimization: a central planner coordinates the controls of all agents so as to minimize a global performance criterion.

In this paper, we study linear-quadratic (LQ) social optimization problems with common noise, together with their mean field limits and long-time behavior. The presence of common noise, representing external shocks shared by all agents, induces conditional dependence across the whole population in the large-population limit and leads to control problems driven by conditional McKean--Vlasov dynamics. Even in the LQ setting, this feature produces coupled Riccati systems whose solvability, stability, and asymptotic properties require careful analysis.

We investigate three closely related control problems within a unified framework: (i) a finite-horizon social optimization problem for an $N$-agent system; (ii) the associated finite-horizon mean field control problem, obtained as the large-population limit; and (iii) an ergodic mean field control problem, which captures the long-time behavior of socially optimal policies. Under a common set of assumptions, we establish well-posedness and provide explicit characterizations of the optimal controls for all three control problems. 

The central objectives of this work are to establish turnpike properties for the mean field control problem and derive uniform-in-time convergence estimates from the social optimization problem to the mean field control problem. Our analysis separates the two asymptotic mechanisms involved in the problems. First, we establish an exponential turnpike estimate comparing the finite-horizon mean field optimal pair with its ergodic counterpart. Second, we prove that the finite-population social optimization problem converges to the finite-horizon mean field control problem uniformly with respect to the time horizon. Combining these two estimates yields a quantitative comparison between the finite-population optimal pair and its ergodic mean field counterpart, in which the error is decomposed into a population-size term and a temporal boundary-layer term.

Although each of these problems has been studied individually in the literature, a unified analysis linking social optimization, mean field control, and ergodic control---particularly in the presence of multiplicative common noise---has not yet been fully developed. Moreover, most existing large-population convergence results are established for a fixed time horizon, with constants that may depend on the horizon. Such estimates become ineffective when the time horizon is large and therefore do not justify the mean field approximation in a regime where the population size and the time horizon tend to infinity simultaneously. The present work addresses these issues by developing a unified solvability theory and establishing quantitative convergence estimates that remain valid over arbitrarily long horizons.

\subsection{Related work}

\subsubsection{Social optimization problems}

Social optimization problems, in which a population of interacting agents cooperates to optimize a collective objective, arise in diverse applications such as communication networks \cite{HTK-2017} and social welfare systems \cite{Moulin-2004}. In large-population settings with weak interactions, such problems naturally lead to mean field approximations.

In the linear-quadratic framework, early work by Huang, Caines, and Malham\'e \cite{Huang-Caines-Malhame-2012} established asymptotic team-optimal solutions for mean field social optimization problems under additive noise. Subsequent extensions addressed non-uniform agents \cite{Chen-Huang-2019}, indefinite state weights and cost functionals \cite{Huang-Yang-2021-Linear, Wang-Zhang-2020}, and additional sources of randomness \cite{Wang-Zhang-2017}. In particular, Huang and Yang \cite{Huang-Yang-2021-Linear} analyzed asymptotic solvability of finite-horizon LQ mean field social optimization problems, while Wang and Zhang \cite{Wang-Zhang-2020} studied infinite-horizon variants. Despite these developments, most existing results focus either on finite-horizon or infinite-horizon settings separately and typically treat common noise in a restricted form. A unified analysis connecting finite-population social optimization, its mean field limit, and the long-time behavior under general common noise remains relatively limited.

\subsubsection{Mean field optimal control problems}

Mean field control problems arise as infinite-population limits of social optimization problems and provide a tractable alternative to high-dimensional centralized control. Foundational contributions include the works of Nourian, Caines, Malham\'e, and Huang \cite{Nourian-Caines-Malhame-Huang-2012}, as well as the monograph by Bensoussan, Frehse, and Yam \cite{Bensoussan-Frehse-Yam-2013}, which established the basic formulation and analytical framework of mean field control. 

Subsequent theoretical developments have employed both the stochastic maximum principle \cite{Andersson-Djehiche-2011, Buckdahn-Li-Ma-2016} and dynamic programming approaches \cite{Pham-Wei-2017}. Infinite-horizon mean field control has also been investigated, for instance in \cite{Bayraktar-Zhang-2023}, via the solvability of the associated infinite-horizon McKean--Vlasov forward-backward stochastic differential equations (SDEs). In the linear-quadratic setting, Yong \cite{Yong-2013} derived explicit solutions through forward-backward SDEs and Riccati equations, with further extensions addressing non-homogeneous terms \cite{Sun-2017, Li-Sun-Yong-2016}, random coefficients \cite{Xiong-Xu-2025}, and infinite-horizon control problems \cite{Huang-Li-Yong-2015}.

In the presence of common noise, the limiting state dynamics become conditional McKean--Vlasov processes. Pham and Wei \cite{Pham-Wei-2017} established a general dynamic programming principle for McKean--Vlasov control problems and illustrated its application to the LQ setting. Moreover, \cite{Li-Mou-Wu-Zhou-2025} proved global well-posedness of the associated Hamilton--Jacobi equation for an LQ mean field control problem with additive common noise, where the dependence on the mean field terms can be non-convex. For the more general framework of optimal control of McKean--Vlasov dynamics, we refer to the works \cite{Lacker-2017, Djete-Possamai-Tan-2022, Bayraktar-Cosso-Pham-2018, Pham-Wei-2018} and the references therein.

\subsubsection{Turnpike property}

The turnpike property refers to the phenomenon that, for optimal control problems over long time horizons, optimal controls and state trajectories remain close to those of a corresponding stationary or ergodic problem for most of the time interval, except near the temporal boundaries. This concept originates in economic theory, dating back to von Neumann \cite{Neumann-1945} and later developments by Dorfman, Samuelson, and Solow \cite{Dorfman-Samuelson-Solow-2012}, and has become a fundamental principle in optimal control.

In deterministic control, turnpike properties have been extensively studied for both finite- and infinite-dimensional systems, encompassing both discrete-time and continuous-time formulations; see, for instance \cite{Porretta-Zuazua-2013, DGSW-2014, Trelat-Zuazua-2015, Gugat-Herty-Segala-2024} and the references therein. These works establish turnpike behavior under various dissipativity and stability conditions. A comprehensive and up-to-date survey on turnpike theory and its various applications is presented in \cite{Trelat-Zuazua-2025}.

In contrast, turnpike properties in stochastic optimal control remained largely unexplored until the recent breakthrough by Sun, Wang, and Yong \cite{Sun-Wang-Yong-2022}, who established the first rigorous turnpike results for linear-quadratic stochastic control problems. This work initiated a systematic study of stochastic turnpike phenomena, leading to further developments in settings with periodic coefficients \cite{Sun-Yong-2024-Periodic} and more general stochastic dynamics \cite{Conforti-2023, Mei-Wang-Yong-2025, Mei-Wang-Yong-2-2025}.
In addition, \cite{Jian-Jin-Song-Yong-2024} introduces the so-called probability cell problem, establishes a connection between the cell problem and the ergodic control problem, and reveals turnpike properties of LQ stochastic optimal control problems from multiple perspectives.

More recently, attention has turned to turnpike behavior in mean field control problems. In particular, Sun and Yong \cite{Sun-Yong-2024} extended their earlier work \cite{Sun-Wang-Yong-2022} to mean field linear stochastic systems, establishing a turnpike property under the stabilizability assumption on the homogeneous state equation. In a related development, \cite{Bayraktar-Jian-2025} derived the turnpike property for finite-horizon linear-quadratic mean field control problems through analysis of the associated ergodic control problem. Meanwhile, driven by growing interest in long-term equilibrium behavior in large interacting particle systems, the study of turnpike properties in multi-player differential games (see \cite{Cohen-Jian-2025}) and mean field games \cite{CLLP-2013, Cardaliaguet-Porretta-2019, Cirant-Porretta-2021} has emerged as a rapidly advancing field. In particular, the long-time behavior of mean field games with additive common noise has been studied in \cite{CMY25}.

Despite these advances, turnpike results for mean field control problems with common noise, particularly when the common noise enters the dynamics multiplicatively, remain scarce. The conditional nature of the limiting dynamics and the resulting coupled Riccati systems pose additional analytical challenges.

\subsubsection{Convergence of mean field control problems}

A fundamental question in mean field control theory concerns the quantitative convergence of finite-population social optimization problems toward their mean field limits. Such results provide a rigorous justification for using mean field controls as approximations of centralized policies in large-population systems. Convergence of controlled large-population systems to their mean field limits has been studied under various regularity and convexity assumptions, using probabilistic, PDE-based, and variational techniques; see, for example \cite{Lacker-2017, Carmona-Delarue-2018, Germain-Pham-Warin-2022, Bayraktar-Cecchin-Chakraborty-2023, Cardaliaguet-Souganidis-2023, Li-Mou-Wu-Zhou-2025}. Recently, Daudin, Delarue, and Jackson \cite{Daudin-Delarue-Jackson-2024} derived two optimal convergence rates for value functions from $N$-particle systems to the corresponding mean field control problems, identifying regimes corresponding to central limit scaling and dimension-dependent rates. Moreover, Bayraktar, Ekren, and Zhang \cite{Bayraktar-Ekren-Zhang-2025-convergence} established a convergence rate for particle approximations of a class of second-order PDEs on the Wasserstein space, which arise in mean field control when the control of the common noise is state-independent and in stochastic control with partial observation.

Most existing quantitative convergence results are formulated for a fixed finite horizon, and the constants appearing in the estimates are generally allowed to depend on the horizon length. Consequently, such estimates may deteriorate as the horizon grows and do not directly address regimes in which the population size and the time horizon become large simultaneously. The present paper contributes to this direction by proving quantitative convergence from the finite-population social optimization problem to its mean field limit with constants independent of the time horizon.

\subsection{Contributions and organization of the paper}

The main contributions of this paper are summarized as follows:
\begin{itemize}
\item \textit{Unified framework and solvability results}. We study linear-quadratic social optimization problems for large-population stochastic systems with common noise and analyze the finite-horizon social optimization problem, its mean field control limit, and the associated ergodic mean field control problem within a single framework. Our model accommodates multiplicative common noise and a general coupling structure that is not covered by existing work. Under a common set of assumptions, we establish solvability for all three control problems; see Proposition \ref{p:solvability_SO_N}, Proposition \ref{p:solvability_MFC_finite_horizon}, and Theorem \ref{t:ergodic_control_problem}. In particular, the finite-horizon social optimization and mean field control problems are characterized via coupled Riccati systems, while the ergodic mean field control problem is solved through a Bellman equation on the Wasserstein space and a system of stabilizing algebraic Riccati equations.

\item \textit{Ergodic analysis under common noise}.
We introduce appropriate stabilizability conditions for the ergodic mean field control problem with multiplicative common noise and prove a verification theorem via the associated Bellman equation on the Wasserstein space; see Theorem \ref{t:ergodic_control_problem}. The analysis relies on moment estimates in Lemma \ref{l:moment_boundedness} and the identification of an invariant distribution in Lemma \ref{l:invariant-distribution} for the optimal state process.

\item \textit{Turnpike properties}. We derive quantitative estimates comparing finite-horizon and ergodic mean field control problems and establish a turnpike property in Theorem \ref{t:turnpike_property} showing that, away from the temporal boundaries, the optimal pair of the finite-horizon mean field control problem remains close to its ergodic counterpart. To the best of our knowledge, this is the first turnpike result for linear-quadratic mean field control with multiplicative common noise. We also establish a turnpike property for the value functions; see Corollary \ref{c:convergence_of_value_function_finite_ergodic}. 

\item \textit{Large-population convergence}. We prove quantitative convergence results from the social optimization problem to the mean field control problem at the level of both optimal pairs and value functions; see Theorem \ref{t:convergence_path_social_optimal_MFC} and Proposition \ref{p:convergence_value_social_optimal_MFC}, respectively. A principal feature of the result in Theorem \ref{t:convergence_path_social_optimal_MFC} is that the convergence estimate is uniform with respect to the time horizon. Moreover, we obtain convergence estimates that explicitly capture the joint dependence on the population size and the time horizon in Corollary \ref{c:convergence_both_N_and_T_optimal_pair} and Proposition \ref{p:convergence_value_social_optimal_EMFC}, thereby linking finite-horizon and
ergodic regimes.
\end{itemize}

The remainder of the paper is organized as follows: Section \ref{s:problem_setup} presents the problem setup for three control problems. Sections \ref{s:social_optimization_problem}, \ref{s:mean_field_control_problem}, and \ref{s:ergodic_MF_control_problem} establish solvability results for the social optimization problem, the finite-horizon mean field control problem, and the ergodic mean field control problem, respectively. Section \ref{s:turnpike_property} is devoted to turnpike properties for the finite-horizon mean field control problem. In Section \ref{s:convergence_of_SO_N}, we establish the convergence results for the social optimization problem and its mean field limit. Proofs of some results are provided in Appendix \ref{s:appendx}. We close this section by introducing some frequently used notation.

\subsection{Notation}

Let $\dbN^{+}$ be the set of positive integers. For any $k, \ell \in \dbN^{+}$, we denote by $\dbR^k$ the standard $k$-dimensional real Euclidean space and by $\dbR^{k \times \ell}$ the Euclidean space of all $k \times \ell$ real matrices. Let $\dbS^k$ (respectively, $\dbS^k_{+}$ and $\dbS^k_{++}$) denote the set of $k \times k$ real symmetric (respectively, positive semidefinite and positive definite) matrices. For matrices $P, Q \in \dbS^k$, we write $P \ges Q$ (respectively, $P > Q$) if $P - Q$ is positive semidefinite (respectively, positive definite). 

We denote by $\bm{1}_{k \times \ell}$ the $k \times \ell$ matrix with all entries equal to $1$, and by $\otimes$ the Kronecker product. The column vectors $\{e_1^k, \dots, e_k^k\}$ form the canonical basis of $\dbR^k$, and $I_k$ denotes the $k \times k$ identity matrix. For a collection of vectors $\{x^i \in \mathbb R^k: i =1, \dots, N\}$, we use $\bm{x} = (x^1, \dots, x^N)$ to denote the concatenated column vector $((x^1)^\top, \dots, (x^N)^\top)^\top \in \mathbb R^{Nk}$ whenever no confusion arises.

The inner product of two vectors is denoted by $\lan \cd, \cd \ran$. We use the superscript $\top$ to denote the transpose of a matrix, $|\cdot|$ for the Euclidean norm, and define the Frobenius norm of a matrix $A$ by $\|A\| := \sqrt{\text{trace}(A A^\top)}$, where $\text{trace}(\cd)$ denotes the trace operator.

We write $\cP_\ell(\dbR^k)$ for the Wasserstein space of probability measures $\m$ on $\dbR^k$ satisfying $\int_{\dbR^k}|x|^\ell d\m(x)< \infty$, endowed with the $\ell$-Wasserstein metric $\cW_\ell(\cd\,, \cd)$ defined by
$$\cW_\ell(\m_1,\m_2) = \inf_{\pi\in\Pi(\m_1,\m_2)}\Big(\int_{\dbR^k}\int_{\dbR^k}|x-y|^\ell
d\pi(x, y)\Big)^{\frac{1}{\ell}},$$
where $\Pi(\m_1,\m_2)$ is the collection of all probability measures on $\dbR^k\times\dbR^k$ whose marginals agree with $\m_1$ and $\m_2$, respectively.

\section{Problem setup}
\label{s:problem_setup}

In this section, we present the three control problems central to our analysis: the finite-horizon control problem for large-population systems, its finite-horizon mean field limit, and the associated ergodic mean field control problem describing the long-time regime.

\subsection{Finite-horizon control problem for large population systems}

Let $(\wt{\O}, \wt{\cF}, \wt{\dbP})$ be a complete filtered probability space satisfying the usual conditions, and let $\{W^i(\cd): i = 0, 1, \dots, N\}$ be $N+1$ independent one-dimensional standard Brownian motions defined on it. For each $i \in \{0, 1, \dots, N\}$, let $\{\cF^i_t\}_{t \ges 0}$ denote the augmented filtration generated by $W^i(\cdot)$. We define $\wt{\dbF}$ as the complete, right-continuous augmentation of the filtration $\{\vee_{i=1}^N \cF^i_t \vee \cF^0_t \vee \wt{\cF}_0\}_{t \ges 0}$, where $\wt{\cF}_0$ is the $\sigma$-algebra containing all the Borel sets on $\dbR^n$. Assume that $\wt{\dbP}$ is atomless on $\wt{\cF}_0$. The expectation $\dbE^{\wt{\dbP}}$ on the space $(\wt{\O}, \wt{\cF}, \wt{\dbP})$ is denoted by $\wt{\dbE}$ for simplicity.

For each $i \in \{1, \dots, N\}$, let $\xi^i \in L^4_{\wt{\cF}_0}$ be an $\dbR^n$-valued random variable, independent of $\{W^i(t): i = 0, 1, \dots, N\}$, where $L^4_{\wt{\cF}_0}$ is the collection of $\wt{\cF}_0$-measurable random variables $\xi: \wt{\Omega} \to \dbR^n$ such that $\wt{\dbE}[|\xi|^4] < \infty$. Consider the following linear stochastic large-population system with $N$ agents, in which the state of the $i$-th agent evolves according to
\begin{equation}
\label{eq:state_N_agent}
\begin{cases}
\vspace{4pt}
\displaystyle
d X^{i,N}(t) = \big\{AX^{i,N}(t) + \bar A X^{(N)}(t) + Bu^{i, N}(t) + b \big\} dt \\
\vspace{4pt}
\displaystyle
\hspace{0.8in} + \big\{CX^{i, N}(t) + \bar C X^{(N)}(t) + Du^{i, N}(t) + \sigma \big\} d W^i(t) \\
\vspace{4pt}
\displaystyle
\hspace{0.8in} + \big\{\Gamma X^{i, N}(t) + \bar{\Gamma} X^{(N)}(t) + \gamma \big\} d W^0(t), \quad t \ges 0, \\
X^{i,N}(0) = \xi^i,
\end{cases}
\end{equation}
for $i = 1, \dots, N$, where $A, \bar{A}, C, \bar{C}, \Gamma, \bar{\Gamma} \in \dbR^{n \times n}$ and $B, D \in \dbR^{n \times m}$ are the system coefficients, $b, \sigma, \gamma \in \dbR^n$ are the non-homogeneous terms, and $X^{(N)}(t) = \frac{1}{N} \sum_{i = 1}^{N} X^{i, N}(t)$ is the weakly coupled state average. Here, $X^{i, N}(\cd)$ is the $\dbR^n$-valued state process of agent $i$, and $u^{i, N}(\cdot)$ is the $\dbR^m$-valued control process of the $i$-th agent. We refer to the SDE \eqref{eq:state_N_agent} as the state equation.
In the $N$-agent system \eqref{eq:state_N_agent}, $W^i(\cd)$ represents the idiosyncratic noise for agent $i$, and $W^0(\cd)$ represents common noise shared by all agents. 


To proceed, we introduce the following space for the admissible control processes over a finite horizon. Let $T > 0$ denote a large time horizon and
\begin{equation*}
\begin{aligned}
&\sU_c[0,T] := \Big\{u(\cd) : [0,T] \times \wt{\O} \to \dbR^m \Bigm| u(\cd) \hb{ is $\wt{\dbF}$-progressively measurable, and} \\
&\hspace{2.5in} \wt{\dbE} \Big[\int_0^T|u(t)|^2dt \Big]< \infty\Big\}.
\end{aligned}
\end{equation*}
We refer to $\sU_c[0, T]$ as the centralized admissible control set for the social optimization problem.

Let $Q, \bar Q \in \dbS^{n}$, $S \in \dbR^{m \times n}$, and $R \in \dbS^{m}$ be constant matrices, and let $q \in \dbR^n$ and $r \in \dbR^m$ be constant vectors. Given the state dynamics \eqref{eq:state_N_agent}, we introduce the following quadratic cost functional for the $i$-th agent:
\begin{equation}
\label{eq:cost_finite_N_agent}
J_T^{i}(\bm{\xi}; \Bu(\cd)) := \wt{\dbE} \Big[\int_0^T f \big(X^{i, N}(t), X^{(N)}(t), u^{i, N}(t) \big) dt \Big],
\end{equation}
where $\bm{\xi} = (\xi^1, \dots, \xi^N)$ collects the initial states, $\Bu(\cd) = (u^{1, N}(\cd), \dots, u^{N, N}(\cd))$ is the collection of controls for all agents, and $f: \dbR^n \times \dbR^n \times \dbR^m \to \dbR$ is the running cost function defined as follows:
\begin{equation}
\label{eq:function_f}
\begin{aligned}
f(x, \bar x, u) &:= \left\lan \begin{bmatrix}
Q & S^\top \\ S & R
\end{bmatrix} \begin{bmatrix}
x \\ u
\end{bmatrix}, \begin{bmatrix}
x \\ u
\end{bmatrix} \right\ran + 2 \left\lan \begin{bmatrix}
q \\ r
\end{bmatrix}, \begin{bmatrix}
x \\ u
\end{bmatrix}\right\ran + \left\lan 
\bar Q \bar x, \bar x \right\ran. 
\end{aligned}
\end{equation}

The objective of the $N$-agent system is to minimize the social cost subject to the dynamics \eqref{eq:state_N_agent} and the cost functional \eqref{eq:cost_finite_N_agent} for each agent. Define the social cost by
$$J^{\text{soc}}_T (\bm{\xi}; \Bu(\cd)) := \frac{1}{N} \sum_{i = 1}^{N} J^i_T (\bm{\xi}; \Bu(\cd)).$$
The finite-horizon stochastic optimal control problem for the $N$-agent system is then formulated as follows.

\ms

{\bf Problem (SO-N)$^T$.} For a given collection of initial states $\bm{\xi} = (\xi^1, \dots, \xi^N)$ with $\xi^i \in L^4_{\wt{\cF}_0}$ for all $i \in \{1, \dots, N\}$, find an open-loop optimal control profile $\Bu_T(\cd) = (u_T^{1, N}(\cd), \dots, u_T^{N, N}(\cd))$ with $u_T^{i, N}(\cd) \in \sU_c[0,T]$ for all $i \in \{1, \dots, N\}$ such that
\begin{equation}
\label{eq:SO-N_finite}
V^{\text{soc}}_T (\bm{\xi}) := J_T^{\text{soc}} (\bm{\xi}; \Bu_T(\cd)) = \inf_{\bm{u}(\cd) \in (\sU_c[0, T])^N} \frac{1}{N} \sum_{i = 1}^{N} J_T^i (\bm{\xi}; \Bu(\cd)),
\end{equation}
where $(\sU_c[0, T])^N := \sU_c[0,T] \times \dots \times \sU_c[0,T]$.

If there exists a control $\Bu_T(\cd) = (u_T^{1,N}(\cd), \dots, u_T^{N,N}(\cd)) \in (\sU_c[0, T])^N$ attaining the infimum in \eqref{eq:SO-N_finite}, then \textbf{Problem (SO-N)$^T$} is said to be open-loop solvable, and $\Bu_T(\cd)$ is referred to as a socially optimal open-loop control. The corresponding open-loop optimal state process is denoted by $X_T^{i, N}(\cd) := X^{i, N}(\cd\,; \bm{\xi}, \Bu_T(\cd))$ for each $i \in \{1, \dots, N\}$.

\subsection{Finite-horizon mean field control problem}

In this subsection, we introduce the corresponding finite-horizon mean field control problem. 
Let $(\O, \cF, \dbP)$ be a complete filtered probability space supporting two independent one-dimensional standard Brownian motions $W(\cd)$ and $W^0(\cd)$. Here, $W(\cd)$ represents the idiosyncratic noise and $W^0(\cd)$ represents the common noise. Let $\{\mathcal F^W_t\}_{t \ges 0}$ and $\{\mathcal F_t^0\}_{t \ges 0}$ be the augmented filtrations generated by $W(\cd)$ and $W^0(\cd)$, respectively. We define $\dbF$ as the complete, right-continuous augmentation of the filtration $\{\cF^W_t \vee \cF^0_t \vee \cF_0\}_{t \ges 0}$, where $\cF_0$ is the $\sigma$-algebra containing all the Borel sets on $\dbR^n$, and we assume that $\mathbb P$ is atomless on $\mathcal F_{0}$. The expectation on the space $(\O, \cF, \dbP)$ is denoted by $\dbE$.

It is well known that as the population size $N$ tends to infinity, the dynamics \eqref{eq:state_N_agent} are formally approximated by the following linear SDE with a conditional mean field term:
\begin{equation}
\label{eq:state_N_infinity}
\begin{cases}
\vspace{4pt}
\displaystyle
d X(t) = \big\{AX(t) + \bar A \dbE \big[X(t)|\cF_t^0 \big] + Bu(t) + b \big\} dt \\
\vspace{4pt}
\displaystyle
\hspace{0.8in} + \big\{CX(t) + \bar C \dbE \big[X(t)|\cF_t^0 \big] + Du(t) + \sigma \big\} d W(t) \\
\vspace{4pt}
\displaystyle
\hspace{0.8in} + \big\{\Gamma X(t) + \bar{\Gamma} \dbE \big[X(t)|\cF_t^0 \big] + \gamma \big\} d W^0(t), \quad t \ges 0, \\
X(0) = \xi \sim \mu_0,
\end{cases}
\end{equation}
where $\xi \in L^4_{\cF_0}$ has law $\mu_0 \in \cP_4(\dbR^n)$ and is independent of $W(\cd)$ and $W^0(\cd)$. The corresponding cost functional in the mean field limit is 
\begin{equation}
\label{eq:cost_finite_N_infinity}
J_{T}(\mu_0; u(\cd)) := \dbE \Big[\int_0^T f \big(X(t), \dbE \big[X(t)|\cF_t^0 \big], u(t) \big) dt \Big].
\end{equation}

Define the admissible control set on $[0,T]$ by
\begin{equation}
\label{eq:admissible_control_MFC_finite}
\begin{aligned}
&\sU[0,T] :=  \Big\{u(\cdot) : [0,T] \times \O \to \dbR^m \Bigm| u(\cd) \hb{ is $\dbF$-progressively measurable, and} \\
&\hspace{2.5in} \dbE \Big[\int_0^T |u(t)|^2 dt \Big]< \infty\Big\}.
\end{aligned}
\end{equation}
We can now formulate the corresponding finite-horizon mean field control problem as follows:

\ms

{\bf Problem (MFC)$^T$.} For a given random variable $\xi \in L^4_{\cF_0}$ with $\xi \sim \mu_0 \in \cP_4(\dbR^n)$, find an open-loop optimal control $u_T(\cd) \in\sU[0,T]$ for the state dynamics \eqref{eq:state_N_infinity} and cost functional \eqref{eq:cost_finite_N_infinity} such that
\begin{equation}
\label{eq:value_N_infinity}
U_T(\mu_0) := J_T \big(\mu_0; u_T(\cd) \big) = \inf_{u(\cd) \in \sU[0, T]} J_T(\mu_0; u(\cd)).
\end{equation}

Similarly, if there exists a control $u_T(\cd) \in \sU[0,T]$ that attains the infimum in \eqref{eq:value_N_infinity}, we say that \textbf{Problem (MFC)$^T$} is open-loop solvable, and $u_T(\cd)$ is called an open-loop optimal control. The corresponding open-loop optimal state process is denoted by $X_T(\cd) := X(\cd\,; \xi, u_T(\cd))$. 
The mapping $U_T(\cd): \cP_4(\dbR^n) \to \dbR$ is defined as the value function of \textbf{Problem (MFC)$^T$}, and $(X_T(\cd), u_T(\cd))$ is referred to as an open-loop optimal pair.

\subsection{Ergodic mean field control problem}

When the time horizon $T$ is large, it is natural to compare the finite-horizon mean field control problem with its ergodic counterpart, which is commonly used as a proxy for the long-time regime. We prove below that the optimal pair for the finite-horizon control problem satisfies a turnpike property, showing that solutions to the finite-horizon control problem remain exponentially close to those of the ergodic control problem except near the temporal
boundaries. 

We now introduce the ergodic mean field control problem:

\ms

{\bf Problem (EMFC).} For a given random variable $\xi \in L^4_{\cF_0}$ with $\xi \sim \mu_0 \in \cP_4(\dbR^n)$ and the state dynamics \eqref{eq:state_N_infinity}, determine a pair $(U(\cd), c)$ and a control $\bar u(\cd) \in \sU$ such that
\begin{equation}
\label{eq:ergodic_control_problem}
\begin{aligned}
& c := \inf_{u(\cd) \in \sU} \limsup_{T \to \infty} \frac{1}{T} \int_0^T \dbE \big[f \big(X(t), \dbE \big[X(t)|\cF_t^0 \big], u(t) \big) \big] dt, \\
& U(\mu_0) := J_{\infty} (\mu_0; \bar{u}(\cd)) =   \inf_{u(\cd) \in \sU} J_{\infty} (\mu_0; u(\cd)),
\end{aligned}
\end{equation}
where 
$$J_{\infty} (\mu_0; u(\cd)) := \limsup_{T \to \infty} \int_0^T \dbE \big[f \big(X(t), \dbE \big[X(t)|\cF_t^0 \big], u(t) \big) - c \big] dt,$$
and
$\sU$ is a suitable class of admissible controls to be defined later (see Definition \ref{def:admissible_control_ergodic}).

\ms

If $\bar u(\cd) \in \sU$ exists, it is called an optimal control, and the associated state process $\bar X(\cd) := X(\cd \,; \xi, \bar u(\cd))$ is called an optimal state. The mapping $U(\cd)$ is referred to as the relative value function of \textbf{Problem (EMFC)}, the constant $c$ is called the optimal ergodic cost, and $(\bar X(\cd), \bar u(\cd))$ is termed an optimal pair. A solution to \textbf{Problem (EMFC)} is defined as the quadruple $\{U(\cd), c, \bar X(\cd), \bar u(\cd) \}$.


\section{Finite-horizon control problem for large population systems}
\label{s:social_optimization_problem}

In this section, we establish the solvability of \textbf{Problem (SO-N)$^T$} in \eqref{eq:SO-N_finite} by analyzing the corresponding system of Riccati equations. 

For simplicity of notation, we define
\begin{equation*}
\begin{aligned}
& \BX(t) = \begin{bmatrix}
X^{1,N}(t) \\ \vdots \\ X^{N,N}(t)
\end{bmatrix} \in \dbR^{Nn}, \quad \Bu(t) = \begin{bmatrix}
u^{1,N}(t) \\ \vdots \\ u^{N,N}(t)
\end{bmatrix} \in \dbR^{Nm}, \\
& \BA = \text{diag}[A, \dots, A] \in \dbR^{Nn \times Nn}, \quad \bar{\BA} = \bm{1}_{N \times N} \otimes \frac{\bar A}{N} \in \dbR^{Nn \times Nn} \\
& \BB^i = e_i^N \otimes B \in \dbR^{Nn \times m}, \quad \Bb = \bm{1}_{N \times 1} \otimes b \in \dbR^{Nn}, \\
& \BC^i = e_i^N \otimes (e_i^N)^\top \otimes C \in \dbR^{Nn \times Nn}, \quad \bar{\BC}^i = e_i^N \otimes \frac{\bm{1}_{N \times 1}^\top}{N} \otimes \bar{C} \in \dbR^{Nn \times Nn}, \\
& \BD^i = e_i^N \otimes D \in \dbR^{Nn \times m}, \quad \bm{\sigma}^i = e_i^N \otimes \sigma \in \dbR^{Nn}, \quad \bm{\gamma} = \bm{1}_{N \times 1} \otimes \gamma \in \dbR^{Nn}, \\
& \BGamma = \text{diag}[\Gamma, \dots, \Gamma] \in \dbR^{Nn \times Nn}, \quad \bar{\BGamma} = \bm{1}_{N \times N} \otimes \frac{\bar \Gamma}{N} \in \dbR^{Nn \times Nn}.
\end{aligned}
\end{equation*}
Here, $\text{diag}[A, \dots, A]$ is the block-diagonal matrix whose diagonal blocks are $A, \dots, A$.

With the above notation, the system of dynamics \eqref{eq:state_N_agent} for $N$ agents can be written in compact form as
\begin{equation}
\label{eq:dynamic_N_agent_vector_form}
\begin{cases}
\displaystyle d \BX(t) = \Big\{(\BA + \bar{\BA}) \BX(t) + \sum_{i=1}^{N} \BB^i u^i(t) + \Bb \Big\} dt \\
\displaystyle \hspace{0.6in} + \sum_{i=1}^{N} \big((\BC^i + \bar{\BC}^i) \BX(t) + \BD^i u^i(t) + \bm{\sigma}^i \big) dW^i(t) \\
\hspace{0.6in} + \big\{(\BGamma + \bar{\BGamma}) \BX(t) + \bm{\gamma} \big\} d W^0(t), \quad t \ges 0, \\
\BX(0) = \bm{\xi}.
\end{cases}
\end{equation}
Similarly, we can rewrite the social cost as follows:
\begin{equation*}
\begin{aligned}
J_T^{\text{soc}}(\bm{\xi}; \Bu(\cd)) &= \frac{1}{N} \sum_{i = 1}^{N} J^i_T(\bm{\xi}; \Bu(\cd)) = \wt{\dbE} \Big[\int_0^T F(\BX(t), \Bu(t)) dt \Big],
\end{aligned}
\end{equation*}
where $F: \dbR^{Nn} \times \dbR^{Nm} \to \dbR$ is defined by
$$F(\Bx, \Bu) := \Bx^\top (\BQ + \bar{\BQ}) \Bx + 2 \Bu^\top \BS \Bx + \Bu^\top \BR \Bu + 2 \Bq^\top \Bx + 2 \Br^\top \Bu,$$
and
\begin{equation*}
\begin{aligned}
& \BQ = \frac{1}{N}\text{diag}[Q, \dots, Q] \in \dbR^{Nn \times Nn}, \quad \bar{\BQ} = \bm{1}_{N \times N} \otimes \frac{\bar{Q}}{N^2} \in \dbR^{Nn \times Nn}, \\
& \BS = \frac{1}{N}\text{diag}[S, \dots, S] \in \dbR^{Nm \times Nn}, \quad \BR = \frac{1}{N}\text{diag}[R, \dots, R] \in \dbR^{Nm \times Nm}, \\
& \Bq = \frac{\bm{1}_{N \times 1}}{N} \otimes q \in \dbR^{Nn}, \quad \Br = \frac{\bm{1}_{N \times 1}}{N} \otimes r \in \dbR^{Nm}.
\end{aligned}
\end{equation*}

\subsection{Derivation and solvability of the system of Riccati equations}

Let $V(t, \bm{x})$ denote the value function of \textbf{Problem (SO-N)$^T$} with the initial condition $\bm{X}(t) = \bm{x} := (x^1, \dots, x^N) \in \dbR^{Nn}$ at time $t$. Let $\Bu = (u^1, \dots, u^N) \in \dbR^{Nm}$. The Hamilton-Jacobi-Bellman (HJB) equation for $V(t, \bm{x})$ is
\begin{equation*}
\begin{cases}
\displaystyle \partial_t V + \min_{\Bu \in \dbR^{Nm}} \Big\{(D_{\bm{x}} V)^\top \Big((\BA + \bar{\BA}) \Bx + \sum_{i=1}^{N} \BB^i u^i + \Bb \Big) \\
\displaystyle \hspace{0.4in} + \frac{1}{2} \sum_{i=1}^{N} \big((\BC^i + \bar{\BC}^i) \Bx + \BD^i u^i + \bm{\sigma}^i \big)^\top D^2_{\bm{x} \bm{x}} V \big((\BC^i + \bar{\BC}^i) \Bx + \BD^i u^i + \bm{\sigma}^i \big) \\
\displaystyle \hspace{0.4in} + \frac{1}{2} \big((\BGamma + \bar{\BGamma}) \Bx + \bm{\gamma} \big)^\top D^2_{\bm{x} \bm{x}} V \big((\BGamma + \bar{\BGamma}) \Bx + \bm{\gamma} \big) + F(\Bx, \Bu) \Big\} = 0, \\
V(T, \Bx) = 0.
\end{cases}
\end{equation*}
To obtain the desired system of Riccati equations, we introduce additional notation. Set
\begin{equation*}
\BB = (\BB^1, \dots, \BB^N) \in \dbR^{Nn \times Nm}, \quad \text{and} \quad \Be^i = (e_i^N \otimes I_{m})^\top = (0, \dots, I_m, \dots, 0) \in \dbR^{m \times Nm}.
\end{equation*}
Moreover, for $\BZ \in \dbR^{Nn \times Nn}$, we define the mappings
\begin{small}
\begin{equation*}
\begin{aligned}
& \cM_1(\BZ) = \frac{1}{2} \sum_{i=1}^{N} (\Be^i)^\top (\BD^i)^\top \BZ \BD^i \Be^i, \quad \cM_2(\BZ) = \sum_{i=1}^{N} (\BC^i + \bar{\BC}^i)^\top \BZ \BD^i \Be^i, \\
& \cM_3(\BZ) = \sum_{i=1}^{N} (\bm{\sigma}^i)^\top \BZ \BD^i \Be^i, \quad \cM_4(\BZ) = \frac{1}{2} \sum_{i=1}^{N} (\BC^i + \bar{\BC}^i)^\top \BZ (\BC^i + \bar{\BC}^i) + \frac{1}{2} (\BGamma + \bar{\BGamma})^\top \BZ (\BGamma + \bar{\BGamma}), \\
& \cM_5(\BZ) = \sum_{i=1}^{N} (\bm{\sigma}^i)^\top \BZ(\BC^i + \bar{\BC}^i) + \bm{\gamma}^\top \BZ (\BGamma + \bar{\BGamma}), \quad \cM_6(\BZ) = \frac{1}{2} \sum_{i=1}^{N} (\bm{\sigma}^i)^\top \BZ \bm{\sigma}^i + \frac{1}{2} \bm{\gamma}^\top \BZ \bm{\gamma}.
\end{aligned}
\end{equation*}
\end{small}
Then, the above HJB equation can be reduced to
\begin{equation}
\label{eq:HJB_N_agent}
\begin{cases}
\vspace{4pt}
\displaystyle
\partial_t V + \min_{\Bu \in \dbR^{Nm}} \big\{\Bu^\top \big(\BR + \cM_1 (D^2_{\bm{x} \bm{x}} V ) \big) \Bu + \big( (D_{\Bx} V)^\top \BB + \bm{x}^\top \cM_2 (D^2_{\bm{x} \bm{x}} V) + \cM_3 (D^2_{\bm{x} \bm{x}} V ) \\
\vspace{4pt}
\displaystyle \hspace{0.4in} + 2 \Bx^\top \BS^\top + 2 \Br^\top \big) \Bu \big\} + (D_{\Bx} V)^\top \big((\BA + \bar{\BA}) \Bx + \Bb \big) + \bm{x}^\top \cM_4 (D^2_{\bm{x} \bm{x}} V ) \bm{x} \\
\vspace{4pt}
\displaystyle \hspace{0.4in} + \cM_5 (D^2_{\bm{x} \bm{x}} V ) \bm{x} + \Bx^\top (\BQ + \bar{\BQ}) \Bx + 2 \Bq^\top \Bx + \cM_6 (D^2_{\bm{x} \bm{x}} V ) = 0, \\
V(T, \Bx) = 0.
\end{cases}
\end{equation}
If $\BR + \cM_1 (D^2_{\bm{x} \bm{x}} V)$ is positive definite, then the minimizer of \eqref{eq:HJB_N_agent} is given by
\begin{equation}
\label{eq:optimal_control_function_N_agent}
\begin{aligned}
\Bu(t, \Bx) &= - \frac{1}{2} \big(\BR + \cM_1 (D^2_{\bm{x} \bm{x}} V )\big)^{-1} \big( (D_{\Bx} V)^\top \BB + \bm{x}^\top \cM_2 (D^2_{\bm{x} \bm{x}} V ) \\
& \hspace{0.7in} + \cM_3 (D^2_{\bm{x} \bm{x}} V ) + 2 \Bx^\top \BS^\top + 2 \Br^\top \big)^\top.
\end{aligned}
\end{equation}
Substituting the minimizer \eqref{eq:optimal_control_function_N_agent} into \eqref{eq:HJB_N_agent} yields the reduced HJB equation
\begin{equation*}
\begin{aligned}
& \partial_t V - \frac{1}{4} \big[ (D_{\Bx} V)^\top \BB + \bm{x}^\top \cM_2 (D^2_{\bm{x} \bm{x}} V ) + \cM_3 (D^2_{\bm{x} \bm{x}} V ) + 2 \Bx^\top \BS^\top + 2 \Br^\top \big] \big(\BR + \cM_1 (D^2_{\bm{x} \bm{x}} V )\big)^{-1} \\
& \hspace{0.6in} \big[ (D_{\Bx} V)^\top \BB + \bm{x}^\top \cM_2 (D^2_{\bm{x} \bm{x}} V ) + \cM_3 (D^2_{\bm{x} \bm{x}} V ) + 2 \Bx^\top \BS^\top + 2 \Br^\top \big]^\top \\
& \hspace{0.4in} + (D_{\Bx} V)^\top \big((\BA + \bar{\BA}) \Bx + \Bb \big) + \bm{x}^\top \cM_4 (D^2_{\bm{x} \bm{x}} V ) \bm{x} + \cM_5 (D^2_{\bm{x} \bm{x}} V ) \bm{x} + \Bx^\top (\BQ + \bar{\BQ}) \Bx \\
& \hspace{0.4in} + 2 \Bq^\top \Bx + \cM_6 (D^2_{\bm{x} \bm{x}} V ) = 0.
\end{aligned}
\end{equation*}
Motivated by the linear-quadratic structure, we seek a solution to the HJB equation \eqref{eq:HJB_N_agent} of the following form:
\begin{equation}
\label{eq:value_function_N_agent}
V(t, \Bx) = \Bx^\top \BP_T(t) \Bx + 2 \Bx^\top \Bp_T(t) + \Bk_T(t),
\end{equation}
where $\BP_T(t) \in \dbR^{Nn \times Nn}$ is symmetric, $\Bp_T(t) \in \dbR^{Nn}$, and $\Bk_T(t) \in \dbR$ for all $t \in [0, T]$. Substituting the ansatz \eqref{eq:value_function_N_agent} into the reduced HJB equation and separately equating the polynomial terms of orders 0, 1, and 2, we derive the following system of ODEs for the unknowns $\{\BP_T(t), \Bp_T(t), \Bk_T(t): t \in [0, T]\}$:
\begin{footnotesize}
\begin{equation}
\label{eq:Riccati_N_agent}
\begin{cases}
\vspace{4pt}
\displaystyle
\dot{\BP}_T(t) - \big(\BP_T(t) \BB + \cM_2(\BP_T(t)) + \BS^\top \big) \big(\BR + 2 \cM_1(\BP_T(t)) \big)^{-1} \big(\BP_T(t) \BB + \cM_2(\BP_T(t)) + \BS^\top \big)^\top \\
\vspace{4pt}
\displaystyle
\hspace{0.5in} + \BP_T(t) (\BA + \bar{\BA}) + (\BA + \bar{\BA})^\top \BP_T(t) + 2 \cM_4(\BP_T(t)) + \BQ + \bar{\BQ} = 0, \\
\vspace{4pt}
\displaystyle
\dot{\Bp}_T(t) - \big(\BP_T(t) \BB + \cM_2(\BP_T(t)) + \BS^\top \big) \big(\BR + 2 \cM_1(\BP_T(t)) \big)^{-1} \big((\Bp_T(t))^\top \BB + \cM_3(\BP_T(t)) + \Br^\top \big)^\top \\
\vspace{4pt}
\displaystyle
\hspace{0.5in} + \BP_T(t) \Bb + (\BA + \bar{\BA})^\top \Bp_T(t) + \big(\cM_5(\BP_T(t)) \big)^\top + \Bq = 0, \\
\vspace{4pt}
\displaystyle
\dot{\Bk}_T(t) - \big((\Bp_T(t))^\top \BB + \cM_3(\BP_T(t)) + \Br^\top \big) \big(\BR + 2 \cM_1(\BP_T(t)) \big)^{-1} \big((\Bp_T(t))^\top \BB + \cM_3(\BP_T(t)) + \Br^\top \big)^\top \\
\hspace{0.5in} + 2 (\Bp_T(t))^\top \Bb + 2 \cM_6(\BP_T(t)) = 0
\end{cases}
\end{equation}
\end{footnotesize}
with the terminal conditions
\begin{equation}
\label{eq:terminal_condition_N_agent}
\BP_T(T) = 0 \in \dbS^{Nn}, \quad \Bp_T(T) = 0 \in \dbR^{Nn}, \quad \Bk_T(T) = 0 \in \dbR.
\end{equation}

Next, to ensure the well-posedness of the system \eqref{eq:Riccati_N_agent} and the solvability of \textbf{Problem (SO-N)$^T$}, we impose the following assumption.

\ms

{\bf(H1)} \textit{The matrices $Q, \bar Q \in \dbS^n$ and $R \in \dbS^m_{++}$ satisfy
$$Q - S^\top R^{-1} S \in \dbS^{n}_{++}, \quad \text{and} \quad Q + \bar{Q} - S^\top R^{-1} S \in \dbS^{n}_{++}.$$}

Note that in Assumption \textbf{(H1)}, we do not require $\bar{Q}$ to be positive semidefinite. In particular, it may be an indefinite matrix.

\begin{lemma}
\label{l:unique_solvability_Riccati_N_agent}
Let \textnormal{\textbf{(H1)}} hold. Then the system of differential equations \eqref{eq:Riccati_N_agent} admits a unique solution $\{\BP_T(\cd), \Bp_T(\cd), \Bk_T(\cd)\}$ such that $\BP_T(t) \ges 0$ for all $t \in [0, T]$. Moreover, there exists a constant $\alpha>0$ such that
\begin{equation*}
\BR + 2 \cM_1(\BP_T(t)) \ges \alpha I_{Nm}, \quad \forall t \in [0, T].
\end{equation*}
\end{lemma}
\begin{proof}
Under Assumption \textnormal{\textbf{(H1)}}, $R \in \dbS^{m}_{++}$ and $Q - S^\top R^{-1} S \in \dbS^{n}_{++}$. Hence, $\BR \in \dbS^{Nm}_{++}$ and $Q > S^\top R^{-1} S \ges 0$, so $Q \in \dbS^{n}_{++}$. By the definitions of $\BQ$, $\BS$, and $\BR$, we have
$$\BQ - \BS^\top \BR^{-1} \BS = \frac{1}{N} \text{diag} \big[Q - S^\top R^{-1} S, \dots, Q - S^\top R^{-1} S \big] \in \dbS^{Nn}_{++}.$$
Let $\wt{\BQ} := \bar{\BQ} + \BQ - \BS^\top \BR^{-1} \BS$. We define $\bm{U} \in \dbR^{N \times N}$ to be an orthogonal matrix whose first column is $\frac{1}{\sqrt{N}} \bm{1}_{N \times 1}$ and whose remaining columns form an orthonormal basis of the subspace $\{(y_1, \dots, y_N) \in \dbR^N: \sum_{i=1}^N y_i = 0\}$.
Consider the congruence transform
$$\wt{\BQ}_2 = (\BU \otimes I_n)^\top \wt{\BQ} (\BU \otimes I_n).$$
A direct calculation gives
\begin{equation*}
\begin{aligned}
\wt{\BQ}_2 &= \frac{1}{N} (I_N \otimes (Q - S^\top R^{-1} S)) + \frac{1}{N^2} (\BU^\top \bm{1}_{N \times N} \BU \otimes \bar{Q}) \\
&= \frac{1}{N} (I_N \otimes (Q - S^\top R^{-1} S)) + \frac{1}{N^2} (\text{diag}[N, 0, \dots, 0] \otimes \bar{Q}).
\end{aligned}
\end{equation*}
Since $Q + \bar{Q} - S^\top R^{-1} S \in \dbS^{n}_{++}$ and $Q - S^\top R^{-1} S \in \dbS^{n}_{++}$, we obtain $\wt{\BQ}_2 \in \dbS^{Nn}_{++}$, which implies that $\wt{\BQ} \in \dbS^{Nn}_{++}$. Thus, Conditions (L1), (L2), and (4.23) in Chapter 6 of \cite{Yong-Zhou-1999-stochastic} are satisfied. Therefore, by Theorem 7.2 in Chapter 6 of \cite{Yong-Zhou-1999-stochastic}, 
there exists a unique solution $\BP_T(\cdot) \in C([0, T];\dbS^{Nn}_{+})$ to the first equation in the system \eqref{eq:Riccati_N_agent} and a constant $\alpha > 0$ such that $\BR + 2 \cM_1(\BP_T(t)) \ges \alpha I_{Nm}$ for all $t \in [0, T]$.
Upon substituting $\BP_T(\cd)$ into the second equation of \eqref{eq:Riccati_N_agent}, the resulting equation for $\Bp_T(\cd)$ becomes a linear ordinary differential equation, which admits a unique solution on $[0, T]$. Subsequently, once $\BP_T(\cd)$ and $\Bp_T(\cd)$ are determined, the equation for $\Bk_T(\cd)$ in \eqref{eq:Riccati_N_agent} also reduces to a linear ODE, which likewise admits a unique solution on $[0, T]$.
\end{proof}

The following lemma provides an explicit characterization of the optimal control and the value function for \textbf{Problem (SO-N)$^T$}.

\begin{lemma}
\label{l:solvability_social_optimal_prep}
Suppose Assumption \textnormal{\textbf{(H1)}} holds. Let $\{\BP_T(\cd), \Bp_T(\cd), \Bk_T(\cd)\}$ be the unique solution to the system of equations \eqref{eq:Riccati_N_agent}. Then, \textnormal{\textbf{Problem (SO-N)$^T$}} is open-loop solvable, and the socially optimal control is given in feedback form by
\begin{equation*}
\begin{aligned}
\Bu(t, \BX(t)) &= - \big(\BR + \cM_1 (2 \BP_T(t))\big)^{-1} \big[ \big( \BB^\top \BP_T(t) + (\cM_2 (\BP_T(t)))^\top + \BS \big) \BX(t)  \\
& \hspace{0.7in} + \BB^\top \Bp_T(t) + (\cM_3(\BP_T(t)))^\top + \Br \big],
\end{aligned}
\end{equation*}
where $\BX(\cd)$ is the solution to the corresponding closed-loop system \eqref{eq:dynamic_N_agent_vector_form}. Moreover, the value function defined in \eqref{eq:SO-N_finite} with initial state $\Bx$ is given by $V(0, \bm{x})$ in \eqref{eq:value_function_N_agent}.
\end{lemma}
\begin{proof}
The conclusion is an immediate consequence of Lemma \ref{l:unique_solvability_Riccati_N_agent}, equation \eqref{eq:optimal_control_function_N_agent}, and the classical verification theorem for stochastic control problems (see, e.g., Theorem 5.1 in Chapter 5 of \cite{Yong-Zhou-1999-stochastic}).
\end{proof}

\subsection{Solvability of Problem (SO-N)$^T$}

Note that the optimal control provided in Lemma \ref{l:solvability_social_optimal_prep} depends on the solution $\BP_T(\cd)$ and $\Bp_T(\cd)$ to the system \eqref{eq:Riccati_N_agent}, which appears intractable because of its high-dimensional structure. In this subsection, we characterize the optimal control and the corresponding optimal path for each agent in \textbf{Problem (SO-N)$^T$} by exploiting structural reduction results for the solution $\{\BP_T(t), \Bp_T(t), \Bk_T(t): t \in [0, T]\}$ to the system \eqref{eq:Riccati_N_agent}. This reduction is crucial for establishing the convergence of \textbf{Problem (SO-N)$^T$} to \textbf{Problem (MFC)$^T$}. Our starting point is the Riccati equation satisfied by $\BP_T(\cd)$ in \eqref{eq:Riccati_N_agent}, namely,
\begin{small}
\begin{equation}
\label{eq:BP_1}
\begin{cases}
\vspace{4pt}
\displaystyle
\dot{\BP}_T(t) - \big(\BP_T(t) \BB + \cM_2(\BP_T(t)) + \BS^\top \big) \big(\BR + 2 \cM_1(\BP_T(t)) \big)^{-1} \big(\BP_T(t) \BB + \cM_2(\BP_T(t)) + \BS^\top \big)^\top \\
\vspace{4pt}
\displaystyle
\hspace{0.5in} + \BP_T(t) (\BA + \bar{\BA}) + (\BA + \bar{\BA})^\top \BP_T(t) + 2 \cM_4(\BP_T(t)) + \BQ + \bar{\BQ} = 0, \\
\BP_T(T) = 0 \in \dbS^{Nn}.
\end{cases}
\end{equation}
\end{small}

The following lemma shows that the solution $\BP_T(\cd)$ to \eqref{eq:BP_1} can be represented in terms of two matrix-valued functions $P_T^{1,N}(\cd)$ and $P_T^{2,N}(\cd)$ taking values in $\dbS^n_{+}$. This $N$-invariant algebraic structure yields a useful dimensional reduction and thereby makes the problem more tractable. The proof parallels the method in Lemma 1 of \cite{Huang-Yang-2021-Linear} and is included in Appendix \ref{s:proof_form_BP_1}.

\begin{lemma}
\label{l:form_BP_1}
Let \textnormal{\textbf{(H1)}} hold. The solution $\BP_T(\cd) \in C([0, T]; \dbS^{Nn}_{+})$ to the equation \eqref{eq:BP_1} has the following representation
\begin{equation*}
\BP_T(t) = \begin{bmatrix}
P_T^{1,N}(t) & P_T^{2,N}(t) & \cdots & P_T^{2,N}(t) \\
P_T^{2,N}(t) & P_T^{1,N}(t) & \cdots & P_T^{2,N}(t) \\
\vdots & \vdots & \cdots & \vdots \\
P_T^{2,N}(t) & P_T^{2,N}(t) & \cdots & P_T^{1,N}(t)
\end{bmatrix}
\end{equation*}
for all $t \in [0, T]$, where both $P_T^{1,N}(\cd)$ and $P_T^{2,N}(\cd)$ are $n \times n$ symmetric matrix-valued functions on $[0, T]$.
\end{lemma}

Similarly, for the equation satisfied by $\Bp_T(\cd)$ in \eqref{eq:Riccati_N_agent}, i.e.,
\begin{small}
\begin{equation}
\label{eq:BP_2}
\begin{cases}
\vspace{4pt}
\displaystyle
\dot{\Bp}_T(t) - \big(\BP_T(t) \BB + \cM_2(\BP_T(t)) + \BS^\top \big) \big(\BR + 2 \cM_1(\BP_T(t)) \big)^{-1} \big((\Bp_T(t))^\top \BB + \cM_3(\BP_T(t)) + \Br^\top \big)^\top \\
\vspace{4pt}
\displaystyle
\hspace{0.5in} + \BP_T(t) \Bb + (\BA + \bar{\BA})^\top \Bp_T(t) + \big(\cM_5(\BP_T(t)) \big)^\top + \Bq = 0, \\
\Bp_T(T) = 0 \in \dbR^{Nn},
\end{cases}
\end{equation}
\end{small}
we provide an analogous structural characterization of its solution in the following lemma. The proof is provided in Appendix \ref{s:proof_form_BP_2}.

\begin{lemma}
\label{l:form_BP_2}
Suppose \textnormal{\textbf{(H1)}} holds, and let $\BP_T(\cd)$ be the solution to the Riccati equation \eqref{eq:BP_1}. Then, the solution $\Bp_T(\cd)$ to the differential equation \eqref{eq:BP_2} has the following representation:
\begin{equation*}
\Bp_T(t) = \begin{bmatrix}
(p_T^{1,N}(t))^\top, (p_T^{1,N}(t))^\top, \dots, (p_T^{1,N}(t))^\top 
\end{bmatrix}^\top,
\end{equation*}
where $p_T^{1,N}(\cd)$ is a vector-valued function on $[0, T]$ with $p_T^{1,N}(t) \in \dbR^n$ for all $t \in [0, T]$.
\end{lemma}

Intuitively, the solution $V(t, \Bx)$ to the HJB equation \eqref{eq:HJB_N_agent} is expected to be of order $O(1)$. If we fix $x_i = x$ for all $i \in \{1, \dots, N\}$, then the results of Lemmas \ref{l:form_BP_1} and \ref{l:form_BP_2} give
\begin{equation*}
\begin{aligned}
V(t, \Bx) &= \Bx^\top \BP_T(t) \Bx + 2 \Bx^\top \Bp_T(t) + \Bk_T(t) \\
&= N x^\top P_T^{1,N}(t) x + (N^2 - N) x^\top P_T^{2,N}(t) x + 2N x^\top p_T^{1,N}(t) + \Bk_T(t),
\end{aligned}
\end{equation*}
which suggests the scalings
$$\|P_T^{1,N}(t)\| = O(N^{-1}), \quad \|P_T^{2,N}(t)\| = O(N^{-2}), \quad |p_T^{1,N}(t)| = O(N^{-1}), \quad \text{and} \quad |\Bk_T(t)| = O(1)$$
for all $t \in [0, T]$. Motivated by this heuristic scaling, we follow the rescaling method of \cite{Huang-Zhou-2019-Linear, Ma-Huang-2020-Linear, Huang-Yang-2021-Linear} and define
\begin{equation}
\label{eq:rescaling}
P_T^N(t) = N P_T^{1,N}(t), \quad \Pi_T^N(t) = N^2 P_T^{2,N}(t), \quad p_T^N(t) = N p_T^{1,N}(t), \quad \text{and} \quad \kappa_T^N(t) = \Bk_T(t)
\end{equation}
for all $t \in [0, T]$. 

Using Lemma \ref{l:form_BP_1} together with the definition of $\cM_1(\cd)$, we find that the term $\BR + 2 \cM_1(\BP_T(t))$ in \eqref{eq:Riccati_N_agent} has the following representation:
\begin{equation*}
\begin{aligned}
\BR + 2 \cM_1(\BP_T(t)) &= \BR + \sum_{i=1}^{N} (\Be^i)^\top (\BD^i)^\top \BP_T(t) \BD^i \Be^i \\
&= \frac{1}{N} \text{diag} \big[R+D^\top P_T^N(t) D, \dots, R+D^\top P_T^N(t) D \big].
\end{aligned}
\end{equation*}
By Lemma \ref{l:unique_solvability_Riccati_N_agent}, $\BR + 2 \cM_1(\BP_T(t))$ is positive definite for all $t \in [0, T]$. Thus, $R+D^\top P_T^N(t) D$ is also positive definite, i.e., $R+D^\top P_T^N(t) D \in \dbS^{m}_{++}$ for all $t \in [0, T]$. Consequently, the matrix $(\BR + 2 \cM_1(\BP_T(t)))^{-1}$ takes the following diagonal form
$$\big(\BR + 2 \cM_1(\BP_T(t)) \big)^{-1} = N \text{diag} \big[(R+D^\top P_T^N(t) D)^{-1}, \dots, (R+D^\top P_T^N(t) D)^{-1} \big]$$
for all $t \in [0, T]$.

We next derive the differential equations satisfied by the rescaled functions $P_T^N(\cd)$, $\Pi_T^N(\cd)$, $p_T^N(\cd)$, and $\kappa_T^N(\cd)$ defined in \eqref{eq:rescaling}. For simplicity of notation, we set
\begin{equation*}
\begin{aligned}
& \h A = A + \bar A, \quad \h C = C + \bar C, \quad \h \Gamma = \Gamma + \bar{\Gamma}, \quad \h Q = Q + \bar Q.
\end{aligned}
\end{equation*}
Moreover, for $P, \Pi \in \dbS^n$, we let $\bar{\Pi} := P + \Pi$ and define the following terms
\begin{equation*}
\begin{aligned}
& \mathcal{Q}(P) = P A+A^\top P + C^\top P C + \Gamma^\top P \Gamma + Q, \quad \h{\mathcal Q}(P, \bar{\Pi}) = \bar{\Pi} \h A + \h A^\top \bar{\Pi} + \h C^\top P \h C + \h \Gamma^\top \bar{\Pi} \h \Gamma + \h Q, \\
& \mathcal S(P) = B^\top P + D^\top P C+ S, \quad \h{\mathcal S}(P, \bar{\Pi}) = B^\top \bar{\Pi} + D^\top P \h C + S, \\
& \mathcal R(P) = R + D^\top P D.
\end{aligned}
\end{equation*}
For $P, \bar{\Pi} \in \dbS^n$ and $p \in \dbR^n$, we further set
\begin{equation}
\label{eq:function_theta}
\begin{aligned}
& \Th(P) = - \cR(P)^{-1} \cS(P), \quad \bar{\Th}(P, \bar{\Pi}) = - \cR(P)^{-1} \h \cS(P, \bar{\Pi}), \\
& \theta(p, P) = - \cR(P)^{-1} (B^\top p + D^\top P \sigma + r).
\end{aligned}
\end{equation}
Next, we define the mappings $\Psi_1: \dbS^n \to \dbS^n$, $\Psi_2: \dbS^n \times \dbS^n \to \dbS^n$, $\Psi_3: \dbS^n \times \dbS^n \times \dbR^n \to \dbR^n$, and $\Psi_4: \dbS^n \times \dbS^n \times \dbR^n \to \dbR$ as follows:
\begin{equation}
\label{eq:functions_Psi_1}
\Psi_1(P) := \cQ(P) - \cS(P)^\top \cR(P)^{-1} \cS(P),
\end{equation}
\begin{equation}
\label{eq:functions_Psi_2}
\Psi_2(P, \Pi) := \h \cQ(P, \bar{\Pi}) - \h \cS(P, \bar{\Pi})^\top \cR(P)^{-1} \h \cS(P, \bar{\Pi}) - \Psi_1(P)
\end{equation}
\begin{equation}
\label{eq:functions_Psi_3}
\begin{aligned}
\Psi_3(P, \Pi, p) & := (A + \bar{A} + B \bar{\Th}(P, \bar{\Pi}))^\top p + (C + \bar{C} + D \bar{\Th}(P, \bar{\Pi}))^\top P \sigma \\
& \hspace{0.3in} + \bar{\Th}(P, \bar{\Pi})^\top r + \bar{\Pi}b + \h{\Gamma}^\top \bar{\Pi} \gamma + q,
\end{aligned}
\end{equation}
and
\begin{equation}
\label{eq:functions_Psi_4}
\begin{aligned}
\Psi_4(P, \Pi, p) & := - \theta(p, P)^\top \cR(P) \theta(p, P) + 2 b^\top p + \sigma^\top P \sigma + \gamma^\top \bar{\Pi} \gamma.
\end{aligned}
\end{equation}
In addition, for $N \in \dbN^{+}$, $P, \Pi \in \dbS^n$ and $p \in \dbR^n$, we also define the functions $g_1: \dbN^{+} \times \dbS^n \times \dbS^n \to \dbS^n$, $g_2: \dbN^{+} \times \dbS^n \times \dbS^n \to \dbS^n$, and $g_3: \dbN^{+} \times \dbS^n \times \dbS^n \times \dbR^n \to \dbR^n$ as follows:
\begin{equation}
\label{eq:functions_g_1}
\begin{aligned}
g_1(N, P, \Pi) &:= \frac{1}{N} \bar{C}^\top P D \Th(P) + \frac{1}{N} \Th(P)^\top D^\top P \bar{C}  - \frac{1}{N^2} \bar{C}^\top P D \cR(P)^{-1} D^\top P \bar{C} \\
& \hspace{0.3in} + \frac{N-1}{N^2} (B^\top \Pi + D^\top P \bar{C})^\top \cR(P)^{-1} (B^\top \Pi + D^\top P \bar{C}) \\
& \hspace{0.3in} + \frac{1}{N} \big(\h\cQ(P, \bar{\Pi}) - \cQ(P) - \Pi A  -A^\top \Pi - \Gamma^\top \Pi \Gamma \big) \\
& \hspace{0.3in} - \frac{1}{N^2} \big(\Pi \bar{A} + \bar{A}^\top \Pi + \bar{\Gamma}^\top \Pi \bar{\Gamma} + \bar{\Gamma}^\top \Pi \Gamma + \Gamma^\top \Pi \bar{\Gamma} \big),
\end{aligned}
\end{equation}
\begin{equation}
\label{eq:functions_g_2}
\begin{aligned}
g_2(N, P, \Pi) &:= \frac{1}{N} \big(D^\top P \bar{C} + B^\top \Pi \big)^\top \cR(P)^{-1} B^\top \Pi + \frac{1}{N} \Pi B \cR(P)^{-1} \big(D^\top P \bar{C} + B^\top \Pi \big) \\
& \hspace{0.3in} - \frac{1}{N} \big(\Pi \bar{A} + \bar{A}^\top \Pi + \bar{\Gamma}^\top \Pi \bar{\Gamma} + \bar{\Gamma}^\top \Pi \Gamma + \Gamma^\top \Pi \bar{\Gamma} \big),
\end{aligned}
\end{equation}
and
\begin{equation}
\label{eq:functions_g_3}
g_3(N, P, \Pi, p) := - \frac{1}{N} \big(\Pi B \th(p, P) + \Pi b + \h \Gamma^\top \Pi \gamma \big).
\end{equation}
The following lemma uses the functions defined above to derive the differential equations satisfied by $P_T^N(\cd)$, $\Pi_T^N(\cd)$, $p_T^N(\cd)$, and $\kappa_T^N(\cd)$.

\begin{lemma}
\label{l:ODE_rescaling_terms}
Suppose \textnormal{\textbf{(H1)}} holds, and let $P_T^N(\cd)$, $\Pi_T^N(\cd)$, $p_T^N(\cd)$, and $\kappa_T^N(\cd)$ be defined in \eqref{eq:rescaling}. Then, $(P_T^N(\cd), \Pi_T^N(\cd), p_T^N(\cd), \kappa_T^N(\cd))$ is the solution to the following system of ODEs on $[0, T]$:
\begin{equation}
\label{eq:Riccati_N_agent_rescaling}
\begin{cases}
\vspace{4pt}
\displaystyle
\dot{P}_T^N(t) + \Psi_1 \big(P_T^N(t) \big) + g_1 \big(N, P_T^N(t), \Pi_T^N(t) \big) = 0, \\
\vspace{4pt}
\displaystyle
\dot{\Pi}_T^N(t) + \Psi_2 \big(P_T^N(t), \Pi_T^N(t) \big) + g_2 \big(N, P_T^N(t), \Pi_T^N(t) \big) = 0, \\
\dot{p}_T^N(t) + \Psi_3 \big(P_T^N(t), \Pi_T^N(t), p_T^N(t) \big) + g_3 \big(N, P_T^N(t), \Pi_T^N(t), p_T^N(t) \big) = 0, \\
\vspace{4pt}
\displaystyle
\dot{\kappa}_T^N(t) + \Psi_4 \big(P_T^N(t), \Pi_T^N(t), p_T^N(t) \big) - \frac{1}{N} \gamma^\top \Pi_T^N(t) \gamma = 0, \\
P_T^N(T) = \Pi_T^N(T) = 0 \in \dbS^n, \, p_T^N(T) = 0 \in \dbR^n, \, \kappa_T^N(T) = 0 \in \dbR,
\end{cases}
\end{equation}
where the mappings $\Psi_1, \Psi_2, \Psi_3, \Psi_4$ are defined in \eqref{eq:functions_Psi_1}-\eqref{eq:functions_Psi_4}, and the functions $g_1, g_2, g_3$ are given in \eqref{eq:functions_g_1}-\eqref{eq:functions_g_3}.
\end{lemma}

\begin{proof}
Recall from Lemma \ref{l:unique_solvability_Riccati_N_agent} and the representation of $\BR + 2 \cM_1(\BP_T(t))$ that $R+ D^\top P_T^N(t) D$ is positive definite for all $t \in [0, T]$. Thus, all the inverses appearing in the system of differential equations \eqref{eq:Riccati_N_agent_rescaling} are well-defined. Substituting the representation of $\BP_T(\cd)$ from Lemma \ref{l:form_BP_1} into \eqref{eq:BP_1} and using \eqref{eq:rescaling} yields the first two equations for $P_T^{N}(\cd)$ and $\Pi_T^{N}(\cd)$ in \eqref{eq:Riccati_N_agent_rescaling}. Using Lemma \ref{l:form_BP_2} together with the block structure of $\BP_T(\cd)$ in Lemma \ref{l:form_BP_1}, we then rewrite \eqref{eq:BP_2} in terms of $P_T^{1,N}(\cd)$, $P_T^{2,N}(\cd)$, and $p_T^{1,N}(\cd)$ and apply the rescaling \eqref{eq:rescaling} to obtain the ODE for $p_T^{N}(\cd)$ in \eqref{eq:Riccati_N_agent_rescaling}. Finally, the equation for $\kappa_T^{N}(\cdot)$ follows analogously from the scalar equation for $\Bk_T(\cdot)$ in \eqref{eq:Riccati_N_agent} after substituting the same structural forms and applying \eqref{eq:rescaling}.
\end{proof}

Next, we state the main result for \textbf{Problem (SO-N)$^T$} in \eqref{eq:SO-N_finite}. The following proposition provides an explicit characterization of the optimal control and the associated optimal state dynamics for each agent in terms of the solution to \eqref{eq:Riccati_N_agent_rescaling}.
Moreover, the value function $V_T^{\text{soc}}(\cd)$ is given in semi-explicit form. For all $t \in [0, T]$, we denote
\begin{equation}
\label{eq:theta_T_N}
\begin{aligned}
& \Theta_T^N (t) = \Th(P_T^N(t)), \quad \bar{\Theta}_T^N (t) = \bar{\Th}(P_T^N(t), P_T^N(t) + \Pi_T^N(t)), \\
& \theta_T^N(t) = \theta(p_T^N(t), P_T^N(t)), \quad \bar{\theta}_T^N(t) = \cR(P_T^N(t))^{-1} B^\top \Pi_T^N(t).
\end{aligned}
\end{equation}

\begin{proposition}
\label{p:solvability_SO_N}
Suppose \textnormal{\textbf{(H1)}} holds, and let $P_T^N(\cd)$, $\Pi_T^N(\cd)$, $p_T^N(\cd)$, and $\kappa_T^N(\cd)$ be defined in \eqref{eq:rescaling}. Then, 

{\rm(i)} For a given collection of initial states $\bm{\xi} = (\xi^1, \dots, \xi^N)$, \textnormal{\textbf{Problem (SO-N)}$^T$} admits a unique open-loop optimal control $\Bu_T(\cd) = (u_T^{1,N}(\cd), \dots, u_T^{N, N}(\cd))$ with the following closed-loop representation:
\begin{equation*}
\begin{aligned}
u_T^{i, N}(t) = \Big(\Theta_T^N(t) + \frac{1}{N} \bar{\theta}_T^N(t)\Big) X_T^{i, N}(t) + \big(\bar{\Theta}_T^N(t) - \Theta_T^N(t)\big) X_T^{(N)}(t) + \theta_T^N(t), \quad \forall t \in [0, T],
\end{aligned}
\end{equation*}
where $X_T^{i, N}(\cd) := X^{i, N}(\cd\,; \bm{\xi}, \Bu_T(\cd))$ is the solution to the corresponding closed-loop system obtained by plugging $u_T^{i, N}(\cd)$ into \eqref{eq:state_N_agent}, $X_T^{(N)}(t) := \frac{1}{N} \sum_{i = 1}^{N} X_T^{i, N}(t)$ is the weakly coupled state average, and the tuple $(\Theta_T^N(\cd), \bar{\Theta}_T^N(\cd), \theta_T^N(\cd), \bar{\theta}_T^N(\cd))$ is defined in \eqref{eq:theta_T_N}.

{\rm(ii)} For any given collection of initial states $\Bx = (x^1, \dots, x^N)$, the value function of \textnormal{\textbf{Problem (SO-N)}$^T$} is given in the following semi-explicit form:
\begin{equation}
\label{eq:value_function_SO-N}
\begin{aligned}
V^{\textnormal{soc}}_T(\Bx) &= V(0, \Bx) = \Bx^\top \BP_T(0) \Bx + 2 \Bx^\top \Bp_T(0) + \Bk_T(0) \\
&= \frac{1}{N} \sum_{j=1}^N (x^j)^\top P_T^N(0) x^j + \frac{1}{N^2} \sum_{k=1}^N \sum_{j \neq k}^{N} (x^j)^\top \Pi_T^N(0) x^k + 2 \bar{x}^\top p_T^N(0) + \kappa_T^N(0),
\end{aligned} 
\end{equation}
where $V(\cd, \cd)$ is the solution to the HJB equation \eqref{eq:HJB_N_agent}, and $\{\BP_T(\cd), \Bp_T(\cd), \Bk_T(\cd)\}$ solves the system of ODEs \eqref{eq:Riccati_N_agent}.
\end{proposition}
\begin{proof}
This proposition follows directly from the results in Lemma \ref{l:solvability_social_optimal_prep}, Lemma \ref{l:form_BP_1}, Lemma \ref{l:form_BP_2}, and Lemma \ref{l:ODE_rescaling_terms}.
\end{proof}


\section{Mean field control problem}
\label{s:mean_field_control_problem}

We now consider the finite-horizon mean field control problem in \eqref{eq:value_N_infinity}, which arises as the population size $N$ tends to infinity and extends the mean field control framework of \cite{Yong-2013, Huang-Li-Yong-2015, Yong-2017} to the setting with common noise. To establish the solvability of \textbf{Problem (MFC)$^T$}, we introduce an infinite-dimensional Hamilton-Jacobi (HJ) equation. Recent advances in viscosity solutions for second-order PDEs on the Wasserstein space can be found in \cite{Bayraktar-Ekren-Zhang-2025, BCEQTZ-2025, Bayraktar-Ekren-He-Zhang-2025}. For simplicity of notation, we write
\begin{equation}
\label{eq:hat_coefficient_functions}
\begin{aligned}
& \h b(x, \mu, u) := Ax + \bar A \bar \mu + B u + b, \quad \h \sigma(x, \mu, u) := Cx + \bar C \bar \mu + D u + \sigma, \\
& \h \gamma(x, \mu) := \Gamma x + \bar{\Gamma} \bar{\mu} + \gamma.
\end{aligned}
\end{equation}
To proceed, we define the Hamiltonian $\dbH: \dbR^{n} \times \cP_2(\dbR^{n}) \times \dbR^{m} \times \dbR^{n} \times \dbS^{n} \to \dbR$ by
\begin{equation*}
\dbH(x, \mu, u, \Bp, \BP) := \lan \h b(x, \mu, u), \Bp \ran + \frac{1}{2} \lan \BP \h \sigma(x, \mu, u), \h \sigma(x, \mu, u) \ran + f(x, \bar{\mu}, u).
\end{equation*}
We further define $H: \dbR^{n} \times \cP_2(\dbR^{n}) \times \dbR^{n} \times \dbS^{n} \to \dbR$ to be the infimum of $\dbH$ with respect to $u$ over $\dbR^m$, i.e.,
\begin{equation}
\label{eq:Hamiltonian}
H(x, \mu, \Bp, \BP) := \inf_{u \in \dbR^m} \dbH(x, \mu, u, \Bp, \BP).
\end{equation}
A direct expansion yields
\begin{equation*}
\begin{aligned}
& H(x, \mu, \Bp, \BP) \\
= \ & \lan Ax + \bar A \bar \mu + b, \Bp \ran + \frac{1}{2} \lan \BP(Cx + \bar C \bar \mu + \sigma), Cx + \bar C \bar \mu + \sigma \ran + \lan Qx, x\ran + 2\lan q, x\ran + \lan \bar Q \bar \mu, \bar \mu \ran \\
& \hspace{0.3in} + \inf_{u \in \dbR^m} \Big\{\frac{1}{2} \big\lan \big(2R + D^\top \BP D \big)u, u \big\ran + \big\lan u, B^\top \Bp + D^\top \BP(Cx + \bar C \bar \mu + \sigma) + 2Sx + 2r \big\ran \Big\}.
\end{aligned}
\end{equation*}
If $2R + D^\top \BP D$ is invertible, the infimum defining $H$ is attained at the following unique minimizer $\h u: \dbR^{n} \times \cP_2(\dbR^{n}) \times \dbR^{n} \times \dbS^{n} \to \dbR^{m}$:
\begin{equation}
\label{eq:u_hat_minimum_Hamiltonian}
\h u(x, \mu, \Bp, \BP) = - \big(2R + D^\top \BP D \big)^{-1} \big[ \big(D^\top \BP C + 2S \big) x + D^\top \BP \bar C \bar \mu + B^\top \Bp + D^\top \BP \sigma + 2r \big].
\end{equation}
Using the above notation and the definition of $H$ in \eqref{eq:Hamiltonian}, we introduce the following infinite-dimensional HJ equation for $U_T: [0, T] \times \cP_2(\dbR^n) \to \dbR$
\begin{equation}
\label{eq:master_equation_MC}
\begin{cases}
\vspace{4pt}
\displaystyle \partial_t U_T(t, \mu) + \int_{\dbR^n} H \big(x, \mu, D_{\mu} U_T(t, \mu, x), D_{x \mu} U_T(t, \mu, x) \big) \mu(dx) \\
\vspace{4pt}
\displaystyle \hspace{0.5in} + \frac{1}{2} \int_{\dbR^n} \text{trace} \big(\h \gamma(x, \mu) (\h \gamma(x, \mu))^\top D_{x \mu} U_T(t, \mu, x) \big) \mu(dx) \\
\vspace{4pt}
\displaystyle \hspace{0.5in} + \frac{1}{2} \int_{\dbR^n} \int_{\dbR^n} \text{trace} \big(\h \gamma(x, \mu)(\h \gamma(y, \mu))^\top D_{\mu \mu}^2 U_T(t, \mu, x, y) \big) \mu(dx) \mu(dy) = 0, \\
U_T(T, \mu) = 0.
\end{cases}
\end{equation}

Following the technique and ansatz developed in Section 2 of \cite{Bayraktar-Jian-2025}, and noting that the cost functional of \textbf{Problem (MFC)$^T$} depends only on the conditional first and second moments of the state process, we consider the following candidate solution to the HJ equation \eqref{eq:master_equation_MC}:
\begin{equation}
\label{eq:value_ansatz_MFC}
U_T(t, \mu) = \lan \Pi_T(t) \bar \mu, \bar \mu \ran + 2 \lan p_T(t), \bar \mu \ran + \kappa_T(t) + \int_{\dbR^n} \lan P_T(t) x, x \ran \mu(dx),
\end{equation}
where $\Pi_T(t) \in \dbS^n$, $P_T(t) \in \dbS^n$, $p_T(t) \in \dbR^n$, and $\kappa_T(t) \in \dbR$ are time-dependent coefficients to be determined for $t \in [0, T]$. 
Set
$$\bar{\Pi}_T(t) = P_T(t) + \Pi_T(t), \quad \forall t \in [0, T].$$
Substituting the candidate $U_T(\cd, \cd)$ from \eqref{eq:value_ansatz_MFC} into the HJ equation \eqref{eq:master_equation_MC}, we obtain the following system of Riccati equations for $\{P_T(\cd), \bar{\Pi}_T(\cd) \}$ and two backward ODEs for $\{p_T(\cd), \kappa_T(\cd)\}$:
\begin{equation}
\label{eq:Riccati_MC}
\begin{cases}
\vspace{4pt}
\displaystyle
\dot{P}_T(t) + \Psi_1(P_T(t)) = 0, \\
\vspace{4pt}
\displaystyle
\dot{\bar{\Pi}}_T(t) + \bar{\Psi}_2(P_T(t), \bar{\Pi}_T(t)) = 0, \\
\vspace{4pt}
\displaystyle
\dot{p}_T(t) + \Psi_3(P_T(t), \Pi_T(t), p_T(t)) = 0, \\
\vspace{4pt}
\displaystyle
\dot{\kappa}_T(t) + \Psi_4(P_T(t), \Pi_T(t), p_T(t)) = 0, \\
P_T(T) = \bar{\Pi}_T(T) = 0 \in \dbS^n, \, p_T(T) = 0 \in \dbR^n, \, \kappa_T(T) = 0 \in \dbR,
\end{cases}
\end{equation}
where the mappings $\Psi_1, \Psi_2, \Psi_3$ and $\Psi_4$ are defined in \eqref{eq:functions_Psi_1}-\eqref{eq:functions_Psi_4}, and
\begin{equation}
\label{eq:functions_bar_Psi_2}
\bar{\Psi}_2(P, \bar{\Pi}) := \Psi_1(P) + \Psi_2(P, \Pi) = \h \cQ(P, \bar{\Pi}) - \h \cS(P, \bar{\Pi})^\top \cR(P)^{-1} \h \cS(P, \bar{\Pi})
\end{equation}
for all $P, \Pi \in \dbS^n$ with $\bar{\Pi} = P + \Pi$.

The existence and uniqueness of solutions to the system \eqref{eq:Riccati_MC} under Assumption \textnormal{\textbf{(H1)}} can be established by the same approach as in \cite{Yong-Zhou-1999-stochastic} and \cite{Yong-2013}. The following proposition also presents the optimal control and optimal path for \textbf{Problem (MFC)$^T$} in \eqref{eq:value_N_infinity}. To proceed, we define 
\begin{equation}
\label{eq:theta_T_star}
\Th_T^{*}(t) = \Theta(P_T(t)), \quad \bar{\Th}_{T}^{*}(t) = \bar{\Th}(P_T(t), \bar{\Pi}_T(t)), \quad \th_T^{*}(t) = \th(p_T(t), P_T(t)),
\end{equation}
where the mappings $\Theta(\cd), \bar{\Theta}(\cd, \cd)$ and $\theta(\cd, \cd)$ are given by \eqref{eq:function_theta}.

\begin{proposition}
\label{p:solvability_MFC_finite_horizon}
Let \textnormal{\textbf{(H1)}} hold. Then 

{\rm(i)} The system of equations \eqref{eq:Riccati_MC} admits a unique solution $\{P_T(\cd), \bar{\Pi}_T(\cd), p_T(\cd), \kappa_T(\cd)\}$ such that $P_T(t), \bar{\Pi}_T(t) \in \dbS^n_{+}$ for all $t \in [0, T]$. Moreover, there exists a constant $\alpha > 0$ such that
\begin{equation*}
\cR(P_T(t)) = R + D^\top P_T(t) D \ges \alpha I_{m}, \quad \forall t \in [0, T].
\end{equation*}

{\rm(ii)} For any initial state $\xi \in L^4_{\cF_0}$, \textnormal{\textbf{Problem (MFC)$^T$}} admits a unique open-loop optimal control $u_T(\cd)$ with the following closed-loop representation:
\begin{equation}
\label{eq:optimal_control_MFC}
u_T(t) = \Th_T^*(t) \big(X_T(t) -\dbE[X_T(t)|\cF_t^0] \big) + \bar \Th_T^*(t) \dbE[X_T(t)|\cF_t^0] + \th_T^*(t), \quad t \in [0,T],
\end{equation}
where $X_T(\cd) := X(\cd\,; \xi, u_T(\cd))$ solves the corresponding closed-loop state equation
\begin{small}
\begin{equation}
\label{eq:optimal_path_MFC}
\begin{cases}
\vspace{4pt}
\displaystyle
d X_T(t) = \big\{(A+B\Theta_T^*(t)) X_T(t) + \big[\bar A + B(\bar \Theta_T^*(t) - \Theta_T^*(t)) \big] \dbE \big[X_T(t)|\cF_t^0 \big] + B \theta_T^*(t) + b \big\} dt \\
\vspace{4pt}
\displaystyle
\hspace{0.7in} + \big\{(C+D\Theta_T^*(t)) X_T(t) + \big[\bar C + D(\bar \Theta_T^*(t) - \Theta_T^*(t))\big] \dbE \big[X_T(t)|\cF_t^0 \big] + D \theta_T^*(t) + \sigma \big\} dW(t) \\
\vspace{4pt}
\displaystyle
\hspace{0.7in} + \big\{\Gamma X_T(t) + \bar{\Gamma} \dbE \big[X_T(t)|\cF_t^0 \big] + \gamma \big\} d W^0(t), \quad t \ges 0, \\
X_T(0) = \xi,
\end{cases}
\end{equation}
\end{small}
and the tuple $(\Th_T^{*}(\cd), \bar{\Th}_{T}^{*} (\cd), \th_T^{*}(\cd))$ is defined in \eqref{eq:theta_T_star}.

{\rm(iii)} For any initial state $\xi \in L^4_{\cF_0}$ with law $\mu_0 \in \cP_4(\dbR^n)$, the value function of \textnormal{\textbf{Problem (MFC)$^T$}} is given by
\begin{equation*}
\begin{aligned}
U_T(\mu_0) &= \lan \Pi_T(0) \bar \mu_0, \bar \mu_0 \ran + 2 \lan p_T(0), \bar \mu_0 \ran + \kappa_T(0) + \int_{\dbR^n} \lan P_T(0) x, x \ran \mu_0(dx) \\
&= \lan \Pi_T(0) \bar \mu_0, \bar \mu_0 \ran + 2 \lan p_T(0), \bar \mu_0 \ran + \int_{\dbR^n} \lan P_T(0) x, x \ran \mu_0(dx) \\
& \hspace{0.5in} + \int_0^T \big[\lan P_T(t) \sigma, \sigma \ran + \lan \bar{\Pi}_T(t) \gamma, \gamma \ran + 2 \lan p_T(t), b \ran  - \lan \cR (P_T(t)) \theta_T^*(t), \theta_T^*(t) \ran \big] dt,
\end{aligned} 
\end{equation*}
where $\bar{\mu}_0 = \int_{\dbR^n} x \mu_0(dx)$, $\Pi_T(t) = \bar{\Pi}_T(t) - P_T(t)$, $\{P_T(t), \bar{\Pi}_T(t), p_T(t), \kappa_T(t): t \in [0, T]\}$ is the solution to the system \eqref{eq:Riccati_MC}, and $\theta_T^*(\cd)$ is defined in \eqref{eq:theta_T_star}.
\end{proposition}

The proof of Proposition \ref{p:solvability_MFC_finite_horizon} is provided in Appendix \ref{s:proof_solvability_MFC_finite_horizon}. 


\section{Ergodic mean field control problem}
\label{s:ergodic_MF_control_problem}

In the previous section, we introduced the finite-horizon mean field control problem and established its solvability in Proposition \ref{p:solvability_MFC_finite_horizon}. It is then natural to consider the infinite-horizon setting on $[0,\infty)$. Specifically, we retain the state dynamics \eqref{eq:state_N_infinity} and consider the cost functional
\begin{equation*}
\wt{J}_{\infty}(\mu_0; u(\cd)) = \dbE \Big[\int_0^{\infty} f \big(X(t), \dbE \big[X(t)|\cF_t^0 \big], u(t) \big) dt \Big].
\end{equation*}
However, although the non-homogeneous dynamics \eqref{eq:state_N_infinity} on $[0, \infty)$ may be solvable when at least one of $b$, $\sigma$, or $\gamma$ is non-zero, the cost functional above may fail to be well defined, even in cases where the homogeneous system is stabilizable. Consequently, the infinite-horizon version of the mean field control problem is not well-posed in general. 

To address this issue, we turn our attention to the Bellman equation in Subsection \ref{s:Bellman_equation_algebraic_Riccati}, following the approach presented in \cite{Jian-Jin-Song-Yong-2024} (where it is referred to as the cell problem) and \cite{Bayraktar-Jian-2025}. We then show that \textbf{Problem (EMFC)} in \eqref{eq:ergodic_control_problem} provides a probabilistic interpretation of the Bellman equation and prove a verification theorem in Subsection \ref{s:verification_theorem}, establishing that the solution to \textbf{Problem (EMFC)} can be derived from the solution to the Bellman equation.

\subsection{The Bellman equation and algebraic Riccati equations}
\label{s:Bellman_equation_algebraic_Riccati}

Before introducing the Bellman equation, we consider the homogeneous state dynamics on the infinite horizon $[0, \infty)$ and formulate a stabilizability condition for this system. This condition is essential for ensuring the solvability of the algebraic Riccati equations associated with the Bellman equation.

\subsubsection{Homogeneous state equation and stabilizability condition}

To proceed, we set $\sU_{loc}[0, \infty) = \bigcap_{T > 0} \sU[0, T]$ and
\begin{equation*}
\sU[0, \infty) := L^2_{\dbF}(0, \infty; \mathbb R^m) := \Big\{u(\cd) \in \sU_{loc}[0, \infty) \Bigm| \dbE \Big[\int_0^\infty |u(t)|^2dt \Big]< \infty \Big\},
\end{equation*}
where $\sU[0, T]$ is defined in \eqref{eq:admissible_control_MFC_finite}. 

The homogeneous state process, denoted compactly by $[A, \bar A, C, \bar C, \Gamma, \bar{\Gamma}; B, D]$, is defined as follows (that is, $b = \sigma = \gamma = 0$; compare with \eqref{eq:state_N_infinity}):
\begin{equation}
\label{eq:state_MFC_homo}
\begin{cases}
\vspace{4pt}
\displaystyle
d X^\text{Hom}(t) = \big\{AX^\text{Hom}(t) + \bar A \dbE \big[X^\text{Hom}(t)|\cF_t^0 \big] + Bu(t) \big\} dt \\
\vspace{4pt}
\displaystyle
\hspace{0.8in} + \big\{CX^\text{Hom}(t) + \bar C \dbE \big[X^\text{Hom}(t)|\cF_t^0 \big] + Du(t) \big\} d W(t) \\
\vspace{4pt}
\displaystyle
\hspace{0.8in} + \big\{\Gamma X^\text{Hom}(t) + \bar{\Gamma} \dbE \big[X^\text{Hom}(t)|\cF_t^0 \big] \big\} d W^0(t), \quad t \ges 0, \\
X^\text{Hom}(0) = \xi \in L^4_{\cF_0}.
\end{cases}
\end{equation}
For $u(\cd) \in \sU[0, \infty)$, a standard argument using the contraction mapping theorem shows that the homogeneous system $[A, \bar A, C, \bar C, \Gamma, \bar{\Gamma}; B, D]$ admits a unique solution $X^\text{Hom}(\cd) := X^\text{Hom}(\cd; \xi, u(\cd)) \in \sX_{\text{loc}}[0, \infty)$, where
$\sX_{\text{loc}}[0, \infty) = \bigcap_{T > 0} \sX[0, T]$ with
\begin{equation*}
\begin{aligned}
&\sX[0,T] := \Big\{X:[0,T] \times \O \to \dbR^n \Bigm| X(\cd) \hb{ is $\dbF$-adapted, $t \mapsto X(t, \omega)$ is continuous, and} \\
&\hspace{2.5in} \dbE\Big[\sup_{t \in [0, T]} |X(t)|^2 \Big ] < \infty \Big\}.
\end{aligned}
\end{equation*}

\begin{definition}
\label{d:homo_system_stabilizable}
\textit{The homogeneous system $[A, \bar A, C, \bar C, \Gamma, \bar{\Gamma}; B, D]$ in \eqref{eq:state_MFC_homo} is MF-$L^2$-stabilizable if there exists a pair $(\Theta, \bar \Theta) \in \mathbb R^{m \times n} \times \mathbb R^{m \times n}$, called the MF-$L^2$-stabilizer of $[A, \bar A, C, \bar C, \Gamma, \bar{\Gamma}; B, D]$, such that if $X^{\Theta, \bar{\Theta}}(\cdot) := X(\cd \,; \xi, \Theta, \bar{\Theta})$ is the solution to the following homogeneous closed-loop system 
\begin{equation}
\label{eq:state_home_Theta}
\begin{cases}
\vspace{4pt}
\displaystyle
d X^{\Theta, \bar{\Theta}}(t) = \big\{(A+B\Theta) X^{\Theta, \bar{\Theta}}(t) + \left[\bar A + B(\bar \Theta - \Theta)\right] \dbE \big[X^{\Theta, \bar{\Theta}}(t)|\cF_t^0 \big] \big\} dt \\
\vspace{4pt}
\displaystyle
\hspace{0.8in} + \big\{(C+D\Theta) X^{\Theta, \bar{\Theta}}(t) + \left[\bar C + D(\bar \Theta - \Theta)\right] \dbE \big[X^{\Theta, \bar{\Theta}}(t)|\cF_t^0 \big] \big\} dW(t), \\
\vspace{4pt}
\displaystyle
\hspace{0.8in} + \big\{\Gamma X^{\Theta, \bar{\Theta}}(t) + \bar{\Gamma} \dbE \big[X^{\Theta, \bar{\Theta}}(t)|\cF_t^0 \big] \big\} d W^0(t), \quad t \ges 0, \\
X^{\Theta, \bar{\Theta}}(0) = \xi \in L^4_{\cF_0},
\end{cases}
\end{equation}
and 
\begin{equation*}
u^{\Theta, \bar{\Theta}}(t) = \Theta \big(X^{\Theta, \bar{\Theta}}(t) - \dbE \big[X^{\Theta, \bar{\Theta}}(t)|\cF_t^0 \big] \big) + \bar \Theta \dbE \big[X^{\Theta, \bar{\Theta}}(t)|\cF_t^0 \big], \quad t \ges 0,
\end{equation*}
then
$$\dbE \Big[\int_0^{\infty} \big(|X^{\Theta, \bar{\Theta}}(t)|^2 + |u^{\Theta, \bar{\Theta}}(t)|^2 \big) dt \Big] < \infty.$$}
\end{definition}

\ms

To ensure that the homogeneous system $[A, \bar A, C, \bar C, \Gamma, \bar{\Gamma}; B, D]$ is MF-$L^2$-stabilizable, we introduce the following assumption.

\ms 

{\bf(H2)} 
\textit{The controlled SDE, denoted by $[A + \bar{A}, \Gamma + \bar{\Gamma}; B]$,
\begin{equation}
\label{eq:homo_sde_1}
\begin{aligned}
d \dbE \big[X(t)|\cF_t^0 \big] &= \big\{(A + \bar A) \dbE \big[X(t)|\cF_t^0 \big] + B \dbE \big[u(t)|\cF_t^0 \big] \big\} dt \\
& \hspace{0.8in} + (\Gamma + \bar{\Gamma}) \dbE \big[X(t)|\cF_t^0 \big] d W^0(t), \quad t \ges 0
\end{aligned}
\end{equation}
is $L^2$-stabilizable, i.e., there exists a matrix $\bar{\Th} \in \dbR^{m \times n}$ such that for every initial state $\xi \in L^4_{\cF_0}$, the solution to the system $[A + \bar{A} + B \bar{\Th}, \Gamma + \bar{\Gamma}]$, i.e.,
\begin{equation*}
\begin{cases}
\vspace{4pt}
\displaystyle
d \dbE \big[X(t)|\cF_t^0 \big] = (A + \bar{A} + B \bar{\Th}) \dbE \big[X(t)|\cF_t^0 \big] dt + (\Gamma + \bar{\Gamma}) \dbE \big[X(t)|\cF_t^0 \big] d W^0(t), \quad t \ges 0 \\
\dbE \big[X(0)|\cF_0^0 \big] = \dbE[\xi],
\end{cases}
\end{equation*}
satisfies 
$$\dbE \Big[\int_0^\infty \big|\dbE \big[X(t)|\cF_t^0 \big] \big|^2 dt \Big] < \infty.$$
In this case, we call $\bar \Th$ a stabilizer of \eqref{eq:homo_sde_1}. Moreover, the controlled SDE, denoted by $[A, C, \Gamma; B, D]$,
\begin{equation}
\label{eq:homo_sde_2}
d X(t) = (AX(t) + Bu(t)) dt + (CX(t) + Du(t)) dW(t) + \Gamma X(t) dW^0(t), \quad t \ges 0
\end{equation}
is also $L^2$-stabilizable, i.e., there exists a matrix $\Th \in \dbR^{m \times n}$ such that for any initial state $\xi \in L^4_{\cF_0}$, the solution to the system $[A + B\Th, C + D\Th, \Gamma]$, i.e.,
\begin{equation*}
\begin{cases}
\vspace{4pt}
\displaystyle
d X(t) = (A + B\Th) X(t) dt + (C+ D \Th)X(t) dW(t) + \Gamma X(t) dW^0(t), \quad t \ges 0 \\
X(0) = \xi,
\end{cases}
\end{equation*}
satisfies 
$$\dbE \Big[\int_0^\infty |X(t)|^2 dt \Big] < \infty.$$
In this case, we call $\Theta$ a stabilizer of \eqref{eq:homo_sde_2}.
}

\begin{remark}
\textit{Assumption \textnormal{\textbf{(H2)}} admits a standard characterization in terms of stochastic Lyapunov inequalities and can therefore be verified directly from the matrix coefficients. Indeed, for fixed feedback matrices $\Theta, \bar{\Theta} \in \dbR^{m \times n}$, the two closed-loop systems in \eqref{eq:homo_sde_1} and \eqref{eq:homo_sde_2} are $L^2$-stable if there exist matrices $P, \bar{P} \in \dbS^n_{++}$ such that
$$P(A+B\Theta) + (A+B\Theta)^\top P + (C+D\Theta)^\top P(C+D\Theta) + \Gamma^\top P \Gamma < 0,$$
and
$$\bar{P}(A + \bar{A} + B \bar{\Th}) + (A + \bar{A} + B \bar{\Th})^\top \bar{P} + (\Gamma + \bar{\Gamma})^\top \bar{P} (\Gamma + \bar{\Gamma}) < 0,$$
respectively. For illustration, in the scalar case $n = m = 1$, Assumption \textnormal{\textbf{(H2)}} is equivalent to the existence of $\Theta, \bar{\Theta} \in \dbR$ such that
$$2(A+B\Theta) + (C+D\Theta)^2 + \Gamma^2 < 0, \quad \text{and} \quad 2(A + \bar{A} + B \bar{\Th}) + (\Gamma+\bar{\Gamma})^2 < 0.$$
In this case, the meaning of Assumption \textnormal{\textbf{(H2)}} is particularly transparent: the negative closed-loop drift must dominate the total quadratic contribution of the multiplicative noise terms.
}
\end{remark}

Next, we show that under Assumption \textnormal{\textbf{(H2)}}, the homogeneous system $[A, \bar A, C, \bar C, \Gamma, \bar{\Gamma}; B, D]$ is MF-$L^2$-stabilizable. The proof follows an argument similar to Proposition 3.9 and Theorem 4.3 in \cite{Huang-Li-Yong-2015}, and is included in Appendix \ref{s:proof_stability_homo_system}. 

\begin{lemma}
\label{l:stability_homo_system}
Suppose Assumption \textnormal{\textbf{(H2)}} holds. Then, the system $[A, \bar A, C, \bar C, \Gamma, \bar{\Gamma}; B, D]$ is MF-$L^2$-stabilizable. Consequently, the following set is nonempty:
\begin{equation}
\label{eq:u_ad_infinity}
\sU_{ad}[0,\infty) := \big\{u(\cd) \in \sU[0,\infty) \bigm| X^\textnormal{Hom}(\cd\,; \xi, u(\cd)) \in L^2_{\dbF}(0,\infty; \dbR^n) \big\}.
\end{equation}
\end{lemma}

\subsubsection{The Bellman equation and its solvability}

For the ergodic mean field control problem, we consider the following Bellman equation (also called the cell problem in the literature):
\begin{equation}
\label{eq:Bellman_equation_ergodic}
\begin{aligned}
& \int_{\dbR^n} H(x, \mu, D_{\mu} U(\mu, x), D_{x \mu} U(\mu, x)) \mu(dx)  \\
& \hspace{0.3in} + \frac{1}{2} \int_{\dbR^n} \text{trace} \big(\h\gamma(x, \mu) (\h\gamma(x, \mu))^\top D_{x \mu} U(\mu, x) \big) \mu(dx) \\
& \hspace{0.3in} + \frac{1}{2} \int_{\dbR^n} \int_{\dbR^n} \text{trace} \big(\h\gamma(x, \mu) (\h \gamma(y, \mu))^\top D_{\mu \mu}^2 U(\mu, x, y) \big) \mu(dx) \mu(dy) = c_0,
\end{aligned}
\end{equation}
where the Hamiltonian $H$ is defined in \eqref{eq:Hamiltonian}. A solution to the above equation consists of a pair $(U(\cd), c_0)$ with $U: \cP_2(\dbR^n) \to \dbR$ and $c_0 \in \dbR$.

Next, we establish the solvability of the Bellman equation \eqref{eq:Bellman_equation_ergodic} by exploiting the associated system of algebraic Riccati equations. To proceed, we denote
\begin{equation}
\label{eq:theta_star}
\Th^{*} := \Th(P), \quad \bar \Th^{*} := \bar{\Th}(P, \bar{\Pi}), \quad \theta^{*}
:= \theta(p, P),
\end{equation}
where the functions $\Theta(\cd)$, $\bar{\Theta}(\cd, \cd)$, and $\theta(\cd, \cd)$ are defined in \eqref{eq:function_theta}. We then consider the following algebraic system for the unknowns $\{P, \bar{\Pi}, p, c_0\}$:
\begin{equation}
\label{eq:Riccati_ergodic}
\Psi_1(P) = 0, \quad \bar{\Psi}_2(P, \bar{\Pi}) = 0, \quad \Psi_3(P, \Pi, p) =0, \quad \Psi_4(P, \Pi, p) = c_0.
\end{equation}
where $P, \bar{\Pi} \in \dbS^n$, $p \in \dbR^n$, $c_0 \in \dbR$, $\Pi = \bar{\Pi} - P$, and the mappings $\Psi_1, \Psi_3, \Psi_4$ and $\bar{\Psi}_2$ are defined in \eqref{eq:functions_Psi_1}, \eqref{eq:functions_Psi_3}, \eqref{eq:functions_Psi_4}, and \eqref{eq:functions_bar_Psi_2}, respectively.

Under Assumptions \textnormal{\textbf{(H1)}}-\textnormal{\textbf{(H2)}}, the following proposition establishes the unique solvability of the system \eqref{eq:Riccati_ergodic} and, consequently, the solvability of the Bellman equation \eqref{eq:Bellman_equation_ergodic}.

\begin{proposition}
\label{p:solvability_Bellman_equation_ergodic} 
Let \textnormal{\textbf{(H1)}}-\textnormal{\textbf{(H2)}} hold. Then,

\ms

{\rm(i)} The algebraic system \eqref{eq:Riccati_ergodic} admits a unique solution $\{P, \bar{\Pi}, p, c_0\}$ such that $P \in \dbS^n_{++}$ and $\bar{\Pi} \in \dbS^n_{++}$. Moreover, the matrix $\bar \Th^{*} = - \cR(P)^{-1} \h \cS(P, \bar{\Pi})$
is a stabilizer of \eqref{eq:homo_sde_1}, and the matrix 
$\Th^{*} = - \cR(P)^{-1} \cS(P)$
is a stabilizer of \eqref{eq:homo_sde_2}.

\ms

{\rm(ii)} The Bellman equation \eqref{eq:Bellman_equation_ergodic} admits a solution $(U(\cd), c_0)$ with $U(\mu)$ of the following form:
\begin{equation}
\label{eq:value_function_ergodic_MFC}
U(\mu) = \lan \Pi \bar \mu, \bar \mu \ran + 2 \lan p, \bar \mu \ran + \int_{\dbR^n} \lan P x, x \ran \mu(dx),
\end{equation}
and
\begin{equation}
\label{eq:c_0}
c_0 =  2 b^\top p + \sigma^\top P \sigma + \gamma^\top \bar{\Pi} \gamma - \lan \cR(P) \theta^*, \theta^* \ran,
\end{equation}
where $\bar{\mu} = \int_{\dbR^n} x \mu(dx)$, $\Pi = \bar{\Pi} - P$, $\{P, \bar{\Pi}, p\}$ is the unique solution to \eqref{eq:Riccati_ergodic}, and $\theta^*$ is defined in \eqref{eq:theta_star}.
\end{proposition}

\begin{proof}
{\rm(i)} We first consider the homogeneous finite-horizon mean field control problem, i.e., \textbf{Problem (MFC)$^T$} in \eqref{eq:value_N_infinity} with $b = \sigma = \gamma = q = 0 \in \dbR^n$ and $r = 0 \in \dbR^m$, which we denote by \textbf{Problem (MFC)$^T_0$}. Under Assumption \textnormal{\textbf{(H1)}}, Proposition \ref{p:solvability_MFC_finite_horizon} shows that, for any initial state $\xi \in L^4_{\cF_0}$ with law $\mu_0 \in \cP_4(\dbR^n)$, there exists a unique open-loop optimal control $u_T^{0}(\cd)$ with the following feedback form 
\begin{equation*}
u_T^0(t) = \Th_T^*(t) \big(X_T^0(t) -\dbE \big[X_T^0(t)|\cF_t^0 \big] \big) + \bar \Th_T^*(t) \dbE \big[X_T^0(t)|\cF_t^0 \big], \quad t \in [0,T],
\end{equation*}
to \textbf{Problem (MFC)$^T_0$}, where $X_T^0(\cd) := X(\cd\,; \xi, u_T^0(\cd))$ is the solution to the corresponding closed-loop system
\begin{equation*}
\begin{cases}
\vspace{4pt}
\displaystyle
d X_T^0(t) = \big\{(A+B\Theta_T^*(t)) X_T^0(t) + \big[\bar A + B(\bar \Theta_T^*(t) - \Theta_T^*(t)) \big] \dbE \big[X_T^0(t)|\cF_t^0 \big] \big\} dt \\
\vspace{4pt}
\displaystyle
\hspace{0.7in} + \big\{(C+D\Theta_T^*(t)) X_T^0(t) + \big[\bar C + D(\bar \Theta_T^*(t) - \Theta_T^*(t))\big] \dbE \big[X_T^0(t)|\cF_t^0 \big] \big\} dW(t) \\
\vspace{4pt}
\displaystyle
\hspace{0.7in} + \big\{\Gamma X_T^0(t) + \bar{\Gamma} \dbE \big[X_T^0(t)|\cF_t^0 \big] \big\} d W^0(t), \quad t \ges 0, \\
X_T^0(0) = \xi.
\end{cases}
\end{equation*}
Here, $\Th_T^{*}(\cd)$ and $\bar \Th_{T}^{*} (\cd)$ are defined in \eqref{eq:theta_T_star}. Moreover, the value function of \textbf{Problem (MFC)$^T_0$} is given by
$$U_T^0(\mu_0) = \lan \Pi_T(0) \bar \mu_0, \bar \mu_0 \ran + \int_{\dbR^n} \lan P_T(0) x, x \ran \mu_0(dx),$$
where $\Pi_T(t) = \bar{\Pi}_T(t) - P_T(t)$, and $\{P_T(t), \bar{\Pi}_T(t): t \in [0, T]\}$ is the solution to the first two Riccati differential equations in the system \eqref{eq:Riccati_MC}. Under Assumptions \textnormal{\textbf{(H1)}}-\textnormal{\textbf{(H2)}}, the homogeneous system $[A, \bar A, C, \bar C, \Gamma, \bar{\Gamma}; B, D]$ is MF-$L^2$-stabilizable. Following the argument of Theorem 5.2 in \cite{Huang-Li-Yong-2015}, we obtain, for each fixed $t$,
$$\lim_{T \to \infty} P_T(t) = P, \quad \text{and} \quad \lim_{T \to \infty} \bar{\Pi}_T(t) = \bar{\Pi},$$
which yields the unique solvability of the algebraic Riccati equations for $P$ and $\bar{\Pi}$ in \eqref{eq:Riccati_ergodic}. It follows that, for all $t \ges 0$,
$$\lim_{T \to \infty} \Theta_T^*(t) = \Theta^*, \quad \lim_{T \to \infty} \bar{\Theta}_T^*(t) = \bar{\Theta}^*,$$
and
$$\lim_{T \to \infty} X^0_T(t) = X^{\Theta^*, \bar{\Theta}^*}(t)$$
where $X^{\Theta^*, \bar{\Theta}^*}(\cd)$ is defined in \eqref{eq:state_home_Theta}. This implies that $(\Theta^*, \bar{\Theta}^*)$ is an MF-$L^2$-stabilizer of the homogeneous system $[A, \bar A, C, \bar C, \Gamma, \bar{\Gamma}; B, D]$. Therefore, we conclude that $\bar \Th^{*}$
is a stabilizer of \eqref{eq:homo_sde_1}, and $\Th^{*}$ is a stabilizer of \eqref{eq:homo_sde_2}. Next, $P, \bar{\Pi} \in \dbS^n_{++}$, and the closed-loop system $[A + \bar{A} + B \bar{\Th}^*, \Gamma + \bar{\Gamma}]$ is $L^2$-exponentially stable. Hence, the matrix $A + \bar{A} + B \bar{\Th}^*$ is Hurwitz, with all eigenvalues having strictly negative real parts, and is therefore invertible. From the third equation for $p$ in \eqref{eq:Riccati_ergodic} and the definition of $\Psi_3$ in \eqref{eq:functions_Psi_3}, we conclude that this equation admits a unique solution $p$. Once $P, \bar{\Pi}$, and $p$ are determined, the last equation in \eqref{eq:Riccati_ergodic} uniquely determines $c_0$. This proves the unique solvability of the algebraic system \eqref{eq:Riccati_ergodic}.
\ms

{\rm(ii)} From the explicit expression of $U(\cd)$ in \eqref{eq:value_function_ergodic_MFC}, we obtain the following derivatives:
\begin{equation*}
D_{\mu} U(\mu, x) = 2 \Pi \bar{\mu} + 2P x + 2 p, \quad D_{x \mu} U(\mu, x) = 2P, \quad D^2_{\mu \mu} U(\mu, x, y) = 2 \Pi.
\end{equation*}
Since $R + D^\top P D \in \dbS^m_{++}$, by the definition of the Hamiltonian $H$ in \eqref{eq:Hamiltonian} and the minimizer of $H$ in \eqref{eq:u_hat_minimum_Hamiltonian}, we have
\begin{equation}
\label{eq:feedback_form_ergodic}
\bar u(x, \mu) := \h u \big(x, \mu, D_{\mu} U(\mu, x), D_{x \mu} U(\mu, x)\big) 
= \Th^{*} (x - \bar{\mu}) + \bar \Th^{*} \bar{\mu} + \theta^{*},
\end{equation}
where the triple $(\Th^{*}, \bar{\Theta}^*, \theta^*)$ is defined in \eqref{eq:theta_star}. Then, substituting the ansatz \eqref{eq:value_function_ergodic_MFC} and the minimizer \eqref{eq:feedback_form_ergodic} into the Bellman equation \eqref{eq:Bellman_equation_ergodic}, we obtain 
\begin{equation*}
\begin{aligned}
c_0 &= \lan \Psi_2(P, \Pi) \bar{\mu}, \bar{\mu} \ran + 2 \lan \Psi_3(P, \Pi, p), \bar{\mu} \ran + \Psi_4(P, \Pi, p) + \int_{\dbR^n} \lan \Psi_1(P) x, x \ran \mu(dx),
\end{aligned}
\end{equation*}
where the mappings $\Psi_1, \Psi_2, \Psi_3$ and $\Psi_4$ are defined in \eqref{eq:functions_Psi_1}-\eqref{eq:functions_Psi_4}. Since the above equation holds for all $\mu \in \cP_2(\dbR^n)$, it follows that
\begin{equation*}
\Psi_1(P) = 0, \quad
\Psi_2(P, \Pi) = 0, \quad
\Psi_3(P, \Pi, p) = 0, \quad
\Psi_4(P, \Pi, p) = c_0.
\end{equation*}
Equivalently, using $\Psi_1(P) + \Psi_2(P, \Pi) = \bar{\Psi}_2(P, \bar{\Pi})$ with $\bar{\Pi} = P + \Pi$, this is precisely the algebraic system \eqref{eq:Riccati_ergodic}. Therefore, by the result in part (i), we conclude that $(U(\cd), c_0)$ given in \eqref{eq:value_function_ergodic_MFC}-\eqref{eq:c_0} is a solution to the Bellman equation \eqref{eq:Bellman_equation_ergodic}.
\end{proof}


\subsection{Verification theorem for the ergodic mean field control problem}
\label{s:verification_theorem}

Recall that \textbf{Problem (MFC)$^T$} in \eqref{eq:value_N_infinity} can be regarded as a probabilistic interpretation of the infinite-dimensional Hamilton-Jacobi equation \eqref{eq:master_equation_MC}. In this subsection, we establish an analogous link for the ergodic setting: we show that \textbf{Problem (EMFC)} in \eqref{eq:ergodic_control_problem} provides a probabilistic interpretation of the Bellman equation \eqref{eq:Bellman_equation_ergodic}. More precisely, we relate the Bellman equation \eqref{eq:Bellman_equation_ergodic} to \textbf{Problem (EMFC)} through a verification argument.

In a standard stochastic control problem, one seeks an optimal pair $(\bar X(\cd),\bar u(\cd))$ that minimizes the cost functional and thereby determines the value function $U(\cd)$. In contrast, \textbf{Problem (EMFC)} in \eqref{eq:ergodic_control_problem} involves two intertwined objectives: finding a pair $(\bar X(\cd),\bar u(\cd))$ that simultaneously minimizes the long-term average cost and the associated relative (residual) cost.

A further subtlety in \textbf{Problem (EMFC)} is that the choice of $u(\cd)$ does not necessarily ensure the existence of $\lim_{T\to{\infty}} J_{T}(\mu_0; u(\cd))$. In particular, the control space $\sU$ is not necessarily a subset of $\sU_{ad}[0,\infty)$ (as defined in \eqref{eq:u_ad_infinity}). To address this, we introduce a refined control space $\sU$.

\begin{definition}
\label{def:admissible_control_ergodic}
\textit{Fix an initial state $\xi \in L^4_{\cF_0}$ with $\mathcal{L}(\xi) = \mu_0$. A control process $u(\cd)$ is called an admissible control in the ergodic mean field control problem \eqref{eq:ergodic_control_problem} if:}
\begin{itemize}
\item \textit{It is $\dbF$-progressively measurable;}
\item \textit{For all $T > 0$, the state process $X(\cd) := X(\cd \,; \xi, u(\cd))$, defined as the solution to \eqref{eq:state_N_infinity}, satisfies}
\begin{equation*}
\dbE\Big[\sup_{t \in [0, T]} |X(t)|^4 \Big] < \infty;
\end{equation*}
\item \textit{The joint law of $(X(t), \dbE[X(t)|\cF_t^0], u(t))$ converges to some distribution $\nu_\infty \in \cP_2(\dbR^n \times \dbR^n \times \dbR^m)$ in the $2$-Wasserstein metric $\cW_2 (\cd, \cd)$, i.e.,
$$\lim_{t\to \infty} \cW_2 \big(\cL(X(t), \dbE[X(t)|\cF_t^0], u(t)), \nu_{\infty} \big) = 0.$$
}
\end{itemize}
\end{definition}

In what follows, we provide a complete characterization of \textbf{Problem (EMFC)} in the linear-quadratic setting, building on the solvability of the Bellman equation \eqref{eq:Bellman_equation_ergodic} established in Proposition \ref{p:solvability_Bellman_equation_ergodic}. We first introduce the following SDE, whose solution, denoted by $\bar X(\cd)$, will subsequently be identified as the optimal state process for \textbf{Problem (EMFC)}:
\begin{equation}
\label{eq:optimal_path_ergodic}
\begin{cases}
\vspace{4pt}
\displaystyle
d \bar X(t) = \big\{(A + B \Th^*)\bar X(t) + \big[\bar A + B (\bar \Th^* - \Th^*) \big] \dbE \big[\bar X(t)|\cF_t^0 \big] + B \theta^* + b \big\} dt \\
\vspace{4pt}
\displaystyle
\hspace{0.6in} + \big\{(C + D \Th^*) \bar X(t) + \big[\bar C + D (\bar \Th^* - \Th^*) \big] \dbE \big[\bar X(t)|\cF_t^0 \big] + D \theta^* + \sigma \big\} dW(t) \\
\vspace{4pt}
\displaystyle
\hspace{0.6in} + \big\{\Gamma \bar X(t) + \bar{\Gamma} \dbE \big[\bar X(t)|\cF_t^0 \big] + \gamma \big\} dW^0(t), \\
\bar X(0) = \xi.
\end{cases}
\end{equation}
For simplicity of notation, we denote
\begin{equation*}
\begin{aligned}
& \wt{A}_1 = A + B \Th^*, \quad \wt{A}_2 = \bar A + B \bar \Th^* - B \Th^*, \quad \wt{b} = B\theta^* + b, \\
& \wt{C}_1 = C + D \Th^*, \quad \wt{C}_2 = \bar C + D \bar \Th^* - D \Th^*, \quad \wt{\sigma} = D\theta^* + \sigma.
\end{aligned}
\end{equation*}
Then, with these notations, the SDE \eqref{eq:optimal_path_ergodic} can be equivalently expressed as
\begin{equation*}
\begin{cases}
\vspace{4pt}
\displaystyle
d \bar X(t) = \big\{ \wt{A}_1 \bar X(t) + \wt{A}_2 \dbE \big[\bar X(t)|\cF_t^0 \big] + \wt{b} \big\} dt + \big\{ \wt{C}_1 \bar X(t) + \wt{C}_2 \dbE \big[\bar X(t)|\cF_t^0 \big] + \wt{\sigma} \big\} dW(t) \\
\vspace{4pt}
\displaystyle
\hspace{0.6in} + \big\{\Gamma \bar X(t) + \bar{\Gamma} \dbE \big[\bar X(t)|\cF_t^0 \big] + \gamma \big\} dW^0(t), \\
\bar X(0) = \xi.
\end{cases}
\end{equation*}

Next, we present two lemmas that play a key role in establishing the nonemptiness of $\sU$ and proving the verification theorem. Throughout, $K$ and $\lambda$ denote generic positive constants whose values may change from line to line. The first lemma establishes moment bounds for $\bar{X}(\cd)$, and its proof is given in Appendix \ref{s:proof_moment_boundedness}. 

\begin{lemma}
\label{l:moment_boundedness}
Let \textnormal{\textbf{(H1)}}-\textnormal{\textbf{(H2)}} hold. Then, there exists a constant $K > 0$, independent of $t$, such that the solution $\bar X(\cd)$ to the SDE \eqref{eq:optimal_path_ergodic} satisfies, for all $t \ges 0$,
$$\dbE \Big[\big|\dbE \big[\bar X(t)|\cF_t^0 \big]\big|^2 \Big] \les K, \quad \dbE \big[|\bar X(t)|^2 \big] \les K,$$
and for all $T > 0$,
$$\dbE \Big[\sup_{t\in [0, T]} |\bar X(t)|^4 \Big] < \infty.$$
\end{lemma}

The following lemma characterizes the long-time convergence behavior of the optimal state process $\bar{X}(\cd)$  in \eqref{eq:optimal_path_ergodic} for \textbf{Problem (EMFC)}. Its proof is included in Appendix \ref{s:proof_invariant-distribution}. 

\begin{lemma}
\label{l:invariant-distribution}
Let \textnormal{\textbf{(H1)}}-\textnormal{\textbf{(H2)}} hold, and let $\bar{X}(\cd)$ be the solution to \eqref{eq:optimal_path_ergodic}. Then, the joint law of $(\bar{X}(t), \dbE[\bar{X}(t)|\cF_t^0])$ converges in $\cP_2(\dbR^n \times \dbR^n)$ as $t \to \infty$. Moreover, there exist two measures $\mu^*_{\infty} \in \cP_2(\dbR^n)$ and $\wt{\mu}_{\infty}^{0,*} \in \cP_2(\dbR^n)$ such that
$$\cW_2 \big(\cL(\bar X(t)), \mu^*_{\infty} \big) \to 0, \quad \hbox{as } t \to \infty,$$
and
$$\cW_2 \big(\cL(\dbE[\bar X(t)|\cF_t^0]), \wt{\mu}^{0,*}_{\infty} \big) \to 0, \quad \hbox{as } t \to \infty.$$
\end{lemma}
\begin{remark}
\textit{
In the presence of common noise, the conditional law $\mu_t^0:= \mathcal{L}(\bar{X}(t)|\mathcal{F}_t^0)$ is an $\mathcal{F}_t^0$-measurable random variable taking values in $\mathcal{P}_2(\dbR^n)$. We emphasize that Lemma \ref{l:invariant-distribution} does not assert that $\mu_t^0$ converges to a deterministic probability measure. In a general mean field model with common noise, an invariant long-time object is naturally a random probability measure, or equivalently an invariant probability measure on a suitable space of random conditional distributions, as discussed, for example, in \cite{CMY25}. In the present linear-quadratic setting, it is sufficient to work with the deterministic laws in Lemma \ref{l:invariant-distribution}. Indeed, the state and the running cost of \textnormal{\textbf{Problem (EMFC)}} depend on the conditional law only through its conditional mean, and the solution to the Bellman equation \eqref{eq:Bellman_equation_ergodic} has the quadratic form in \eqref{eq:value_function_ergodic_MFC}.
}
\end{remark}

We now present the main result of this subsection. The following theorem characterizes \textbf{Problem (EMFC)} in the LQ framework using the solvability of the Bellman equation \eqref{eq:Bellman_equation_ergodic} established in Proposition \ref{p:solvability_Bellman_equation_ergodic}.

\begin{theorem}
\label{t:ergodic_control_problem}
Suppose \textnormal{\textbf{(H1)}}-\textnormal{\textbf{(H2)}} hold. Let $(U(\cd), c_0)$ be the solution to the Bellman equation \eqref{eq:Bellman_equation_ergodic}, where $U(\cd)$ is given by \eqref{eq:value_function_ergodic_MFC} and $c_0$ is determined by \eqref{eq:c_0}.
Let $\bar X(\cd)$ be the solution to \eqref{eq:optimal_path_ergodic},
and define the feedback control $\bar u(\cd)$ by
\begin{equation}
\begin{aligned}
\label{eq:optimal_control_EC}
\bar u(t) = \Th^{*} \big(\bar X(t) - \dbE \big[\bar X(t)|\cF_t^0 \big] \big) + \bar \Th^{*} \dbE \big[\bar X(t)|\cF_t^0 \big] + \theta^{*},
\end{aligned}
\end{equation}
where the coefficients $(\Th^*, \bar\Th^*, \theta^*)$ are given in \eqref{eq:theta_star}. Then, the following results hold:

{\rm(i)} The joint law of $(\bar X(t), \dbE[\bar{X}(t)|\cF_t^0], \bar u(t))$ converges in $\cP_2(\dbR^n \times \dbR^n \times \dbR^m)$ as $t \to \infty$, i.e., there exists $\nu^*_{\infty}\in\cP_2(\dbR^n \times \dbR^n \times \dbR^m)$ such that
$$\lim_{t\to \infty} \cW_2 \big(\cL(\bar{X}(t), \dbE[\bar{X}(t)|\cF_t^0], \bar{u}(t)), \nu^{*}_{\infty} \big) = 0.$$

{\rm(ii)} The constant $c_0$, referred to as the ergodic cost, is uniquely determined by the following limit
\begin{equation}
\label{eq:ergodic_constant}
c_0=\lim_{T \to \infty} \frac{1}{T} \int_0^T \dbE \big[f \big(\bar X(t), \dbE \big[\bar X(t)|\cF_t^0 \big], \bar u(t) \big)\big]dt.
\end{equation}

{\rm(iii)} The 4-tuple $\{\bar{U}(\cd), c_0, \bar X(\cd), \bar u(\cd)\}$ solves \textnormal{\textbf{Problem (EMFC)}}, where 
$\bar{U}(\mu_0) := U(\mu_0) - \zeta^*$ and $\zeta^*$ is the stationary constant defined by
\begin{equation}
\label{eq:zeta_star}
\zeta^* := \int_{\dbR^n} \lan \Pi x, x \ran \wt{\mu}^{0,*}_{\infty}(dx) + 2 \lan p, \bar{\mu}^{*}_{\infty} \ran + \int_{\dbR^n} \lan P x, x \ran \mu^{*}_{\infty}(dx).
\end{equation}
Here, $\mu^*_{\infty} \in \cP_2(\dbR^n)$ and $\wt{\mu}^{0, *}_{\infty} \in \cP_2(\dbR^n)$ are the limiting measures identified in Lemma \ref{l:invariant-distribution} for $\bar{X}(t)$ and $\dbE[\bar{X}(t)|\cF_t^0]$, respectively, and $\bar{\mu}^*_{\infty} := \int_{\dbR^n} x \mu^{*}_{\infty} (dx)$.
\end{theorem}

\begin{proof}
From \eqref{eq:optimal_control_EC}, $\bar{u}(t)$ is an affine function of $\bar{X}(t)$ and $\dbE[\bar{X}(t)|\cF_t^0]$. Lemma \ref{l:invariant-distribution} yields convergence in $\cP_2(\dbR^n)$ of both marginals $\cL(\bar{X}(t))$ and $\cL(\dbE[\bar{X}(t)|\cF_t^0])$, and also convergence of the joint law of 
$(\bar{X}(t), \dbE[\bar{X}(t)|\cF_t^0])$ in $\cP_2(\dbR^n \times \dbR^n)$. Therefore, the joint law of $(\bar{X}(t), \dbE[\bar{X}(t) | \cF_t^0],\bar{u}(t))$ converges in 
$\cP_2(\dbR^n \times \dbR^n \times \dbR^m)$, which concludes (i).

For any admissible control $u(\cd) \in \sU$ and its associated state process $X(\cd) := X(\cd; \xi, u(\cd))$, we denote
$$\mu_t^0:= \cL \big(X(t)|\cF_t^0 \big), \quad \forall t \ges 0,$$
which is a random measure. Since the solution $U(\cd): \cP_2(\dbR^n) \to \dbR$ in \eqref{eq:value_function_ergodic_MFC} to the Bellman equation is sufficiently smooth in the sense of Lions derivatives, It\^o's formula along a flow of conditional measures (see Theorem 4.14 in \cite{Carmona-Delarue-2018}) gives
\begin{equation*}
\begin{aligned}
U(\mu_t^0) &= U(\mu_0) + \int_0^t \int_{\dbR^n} \lan D_{\mu}U(\mu_s^0, x), \h{b}(x, \mu_s^0, u(s)) \ran \mu_s^0(dx) ds  \\
& \hspace{0.5in} + \frac{1}{2} \int_0^t \int_{\dbR^n} \lan D_{x \mu}U(\mu_s^0, x) \h{\sigma}(x, \mu_s^0, u(s)), \h{\sigma}(x, \mu_s^0, u(s)) \ran \mu_s^0(dx) ds  \\
& \hspace{0.5in} + \frac{1}{2} \int_0^t \int_{\dbR^n} \lan D_{x \mu}U(\mu_s^0, x) \h{\gamma}(x, \mu_s^0), \h{\gamma}(x, \mu_s^0) \ran \mu_s^0(dx) ds \\
& \hspace{0.5in} + \frac{1}{2} \int_0^t \int_{\dbR^n \times \dbR^n} \lan D^2_{\mu \mu}U(\mu_s^0, x, y) \h{\gamma}(x, \mu_s^0), \h{\gamma}(y, \mu_s^0) \ran \mu_s^0(dx) \mu_s^0(dy) ds \\
& \hspace{0.5in} + \int_0^t \int_{\dbR^n} \lan D_{\mu}U(\mu_s^0, x), \h{\gamma}(x, \mu_s^0) \ran \mu_s^0(dx) dW^0(s),
\end{aligned}
\end{equation*}
where $\h{b}(\cd, \cd, \cd), \h{\sigma}(\cd, \cd, \cd)$ and $\h{\gamma}(\cd, \cd)$ are defined in \eqref{eq:hat_coefficient_functions}. Note that the quadratic variation of the last term in the above equation is finite, i.e.,
\begin{equation*}
\begin{aligned}
& \dbE \Big[\int_0^t \Big(\int_{\dbR^n} \lan D_{\mu}U(\mu_s^0, x), \h{\gamma}(x, \mu_s^0) \ran \mu_s^0(dx) \Big)^2 ds \Big] \\
= \ &  \dbE \Big[\int_0^t \Big(\int_{\dbR^n} \big\lan 2 \Pi \bar{\mu}_s^0 + 2Px + 2p, \Gamma x + \bar{\Gamma} \bar{\mu}_s^0 + \gamma \big\ran \mu_s^0(dx) \Big)^2 ds \Big] \\
\les \ &  K \dbE \Big[\int_0^t \Big(1 + |\bar{\mu}_s^0|^4 + \int_{\dbR^n} |x|^4 \mu_s^0(dx) \Big) ds \Big] \\
\les \ & K \int_0^t \big(1 + \dbE \big[|X(s)|^4 \big] \big) ds < \infty
\end{aligned}
\end{equation*}
by the definition of $\sU$. Fixing $t > 0$ and taking expectations on both sides of the equation above, 
from the Bellman equation \eqref{eq:Bellman_equation_ergodic} and the definition of the Hamiltonian $H$ in \eqref{eq:Hamiltonian}, we derive the following inequality:
\begin{equation*}
\begin{aligned}
\dbE \big[U(\mu_t^0) \big] &\ges U(\mu_0) + \dbE \Big[ \int_0^t \int_{\dbR^n} \big\{H \big(x, \mu_s^0, D_{\mu}U(\mu_s^0, x), D_{x \mu} U(\mu_s^0, x) \big) - f \big(x, \bar{\mu}_s^0, u(s) \big) \big\} \mu_s^0(dx) ds  \\
& \hspace{0.5in} + \frac{1}{2} \int_0^t \int_{\dbR^n} \lan D_{x \mu}U(\mu_s^0, x) \h{\gamma}(x, \mu_s^0), \h{\gamma}(x, \mu_s^0) \ran \mu_s^0(dx) ds \\
& \hspace{0.5in} + \frac{1}{2} \int_0^t \int_{\dbR^n \times \dbR^n} \lan D^2_{\mu \mu}U(\mu_s^0, x, y) \h{\gamma}(x, \mu_s^0), \h{\gamma}(y, \mu_s^0) \ran \mu_s^0(dx) \mu_s^0(dy) ds \Big] \\
&= U(\mu_0) - \dbE \Big[ \int_0^t \int_{\dbR^n} \big( f(x, \bar{\mu}_s^0, u(s)) - c_0 \big) \mu_s^0(dx) ds \Big]\\
&= U(\mu_0) - \dbE \Big[ \int_0^t \big( f(X(s), \dbE[X(s)|\cF_s^0], u(s)) - c_0 \big) ds \Big],
\end{aligned}
\end{equation*}
for all $u(\cd) \in \sU$, which implies that
\begin{equation}
\label{eq:value_inequality_verification}
U(\mu_0) \les \dbE \big[U(\mu_t^0)\big] + \dbE \Big[ \int_0^t \big( f(X(s), \dbE[X(s)|\cF_s^0], u(s)) - c_0 \big) ds \Big].
\end{equation}
The feedback form of $\bar{u}(\cd)$ in \eqref{eq:optimal_control_EC}, together with Lemmas \ref{l:moment_boundedness} and \ref{l:invariant-distribution}, shows that $\bar{u}(\cd) \in \sU$. This confirms that $\sU$ is nonempty. Moreover, setting $u(\cd) = \bar{u}(\cd)$ turns the inequality \eqref{eq:value_inequality_verification} into an equality. Letting $t \to \infty$ and using the explicit form of $U(\cd)$ in \eqref{eq:value_function_ergodic_MFC} and the definition of $\sU$, we obtain
\begin{equation*}
\begin{aligned}
c_0 &\les \lim_{t \to \infty} \frac{1}{t} \big(\dbE \big[U(\mu_t^0)\big] - U(\mu_0) \big) + \lim_{t \to \infty} \frac{1}{t} \dbE \Big[ \int_0^t f \big(X(s), \dbE[X(s)|\cF_s^0], u(s) \big) ds \Big] \\
&= \lim_{t \to \infty} \frac{1}{t} \dbE \Big[ \int_0^t f \big(X(s), \dbE[X(s)|\cF_s^0], u(s) \big) ds \Big]
\end{aligned}
\end{equation*}
for all $u(\cd) \in \sU$. Equality holds when we set $u(\cd) = \bar{u}(\cd)$. Thus, we establish the desired result in \eqref{eq:ergodic_constant}.

Next, using the explicit form of $U(\cd)$ in \eqref{eq:value_function_ergodic_MFC}, we have
$$\dbE \big[U(\mu_t^0)\big] = \dbE \big[\lan \Pi \dbE[X(t)|\cF_t^0], \dbE[X(t)|\cF_t^0] \ran \big] + 2 \lan p, \dbE[X(t)] \ran + \dbE[\lan P X(t), X(t) \ran].$$
By the definition of $\sU$, there exist $\mu_{\infty} \in \cP_2(\dbR^n)$ and $\wt{\mu}_{\infty}^0 \in \cP_2(\dbR^n)$ such that
$$\lim_{t \to \infty} \cW_2(\cL(X(t)), \mu_{\infty}) = 0, \quad \lim_{t \to \infty} \cW_2(\cL(\dbE[X(t)|\cF_t^0]), \wt{\mu}^0_{\infty}) = 0.$$
Letting $t \to \infty$ in \eqref{eq:value_inequality_verification}, we obtain
\begin{equation}
\label{eq:value_inequality_verification_2}
U(\mu_0) \les \zeta + \lim_{t \to \infty} \dbE \Big[ \int_0^t \big( f(X(s), \dbE[X(s)|\cF_s^0], u(s)) - c_0 \big) ds \Big],
\end{equation}
where
$$\zeta:= \lim_{t \to \infty} \dbE \big[U(\mu_{t}^0)\big] = \int_{\dbR^n} \lan \Pi x, x \ran \wt{\mu}^0_{\infty}(dx) + 2 \lan p, \bar{\mu}_{\infty} \ran + \int_{\dbR^n} \lan P x, x \ran \mu_{\infty}(dx)$$
with $\bar{\mu}_{\infty} = \int_{\dbR^n} x \mu_{\infty}(dx)$. This yields, for all $u(\cd) \in \sU$, 
\begin{equation}
\label{eq:value_inequality_verification_3}
U(\mu_0) \les \zeta^* + \lim_{t \to \infty} \dbE \Big[ \int_0^t \big( f(X(s), \dbE[X(s)|\cF_s^0], u(s)) - c_0 \big) ds \Big],
\end{equation}
where $\zeta^*$ is defined in \eqref{eq:zeta_star}. Indeed,
\begin{equation*}
\begin{aligned}
\lim_{t \to \infty} \dbE\big[f \big(X(t), \dbE[X(t)|\cF_t^0], u(t) \big) \big] & = \lim_{T \to \infty} \frac{1}{T} \dbE \Big[ \int_0^T f \big(X(s), \dbE[X(s)|\cF_s^0], u(s) \big) ds \Big] \\
& \ges \lim_{T \to \infty} \frac{1}{T} \int_0^T \dbE \big[f \big(\bar X(t), \dbE \big[\bar X(t)|\cF_t^0 \big], \bar u(t) \big)\big]dt = c_0.
\end{aligned}
\end{equation*}
Thus, it is enough to consider two cases: (i) If 
$\lim_{t \to \infty} \dbE[f(X(t), \dbE[X(t)|\cF_t^0], u(t))] = c_0$,
then \eqref{eq:value_inequality_verification_2} holds with $\zeta$ in
place of $\zeta^*$, which yields \eqref{eq:value_inequality_verification_3}; (ii) if $\lim_{t \to \infty} \dbE[f(X(t), \dbE[X(t)|\cF_t^0], u(t))] > c_0$, then
$$\lim_{t \to \infty} \dbE \Big[ \int_0^t \big( f(X(s), \dbE[X(s)|\cF_s^0], u(s)) - c_0 \big) ds \Big] = \infty,$$
which concludes \eqref{eq:value_inequality_verification_3} with the right-hand side being positive infinity. Moreover, the inequality \eqref{eq:value_inequality_verification_3} holds as an equality if $u(\cd) = \bar{u}(\cd)$. Hence, we derive that
\begin{equation*}
U(\mu_0) - \zeta^* = \inf_{u(\cd) \in \sU} \lim_{T \to \infty} \dbE \Big[ \int_0^T \big( f(X(s), \dbE[X(s)|\cF_s^0], u(s)) - c_0 \big) ds \Big].
\end{equation*}
Then, by the definition of \textbf{Problem (EMFC)} in \eqref{eq:ergodic_control_problem}, we conclude the desired result in (iii). 
\end{proof}


\section{Turnpike property of the mean field control problem}
\label{s:turnpike_property}

In this section, we first present key estimates comparing the coefficient functions $P_T(\cdot)$, $\bar{\Pi}_T(\cd)$, $p_T(\cdot)$,
$\Theta_T^*(\cdot)$, $\bar \Th_T^*(\cd)$, and $\theta_T^*(\cdot)$, which arise in \textbf{Problem (MFC)$^T$} from Section \ref{s:mean_field_control_problem}, with the corresponding constant matrices and vectors $P$, $\bar{\Pi}$, $p$, $\Theta^*$, $\bar \Th^*$, and $\theta^*$ from \textbf{Problem (EMFC)} in Section \ref{s:ergodic_MF_control_problem}. Using these convergence results, we demonstrate the stochastic turnpike property between the optimal pairs of \textbf{Problem (MFC)$^T$} and \textbf{Problem (EMFC)}. Finally, we establish a turnpike property at the level of value functions.

\begin{proposition}
\label{p:convergence_Riccati_T_infinity}
Assume \textnormal{\textbf{(H1)-(H2)}} hold. Let $(P_T(\cd), \bar{\Pi}_T(\cd), p_T(\cd))$ be the solution to the system of equations \eqref{eq:Riccati_MC}, and let $(P, \bar{\Pi}, p)$ be the solution to the stationary system \eqref{eq:Riccati_ergodic}. Then, there exist positive constants $K$ and $\lambda$, independent of $t$ and $T$, such that
\begin{equation}
\label{eq:convergence_Riccati_T_infinity}
\big\|P_T(t) - P \big\| + \big\|\bar{\Pi}_T(t) - \bar{\Pi} \big\| + \big|p_T(t) - p \big| \les K e^{-\lambda (T-t)}, \quad \forall t \in [0, T].
\end{equation}
In particular, $P_T(\cd)$, $\bar{\Pi}_T(\cd)$, and $p_T(\cd)$ are uniformly bounded on $[0, T]$, i.e., there exists $K > 0$, independent of $T$, such that
$$\sup_{t \in [0, T]} \big(\|P_T(t)\| + \|\bar{\Pi}_T(t)\| + |p_T(t)|\big) \les K.$$
\end{proposition}
\begin{proof}
By an argument analogous to that of Theorem 4.1 in \cite{Sun-Wang-Yong-2022} (see also Theorem 5.6 in \cite{Sun-Yong-2024-Periodic}), we first obtain the estimate
$$\big\|P_T(t) - P \big\| \les K e^{-\lambda (T-t)}, \quad \forall t \in [0, T]$$
for some constants $K, \lambda > 0$ independent of $T$. In a similar manner, combining the same exponential stability idea with Theorem 4.1 in \cite{Sun-Yong-2024} yields
$$\big\|\bar{\Pi}_T(t) - \bar{\Pi} \big\| \les K e^{-\lambda (T-t)}, \quad \forall t \in [0, T].$$ 
Using the above two estimates, since $\Pi_T(t) = \bar{\Pi}_T(t) - P_T(t)$ and $\Pi= \bar{\Pi} - P$, we also have $\|\Pi_T(t) - \Pi\| \les K e^{-\lambda(T-t)}$ for all $t \in [0, T]$. Finally, invoking Theorem 4.1 of \cite{Jian-Jin-Song-Yong-2024} gives
$$\big|p_T(t) - p \big| \les K e^{-\lambda (T-t)}, \quad \forall t \in [0, T].$$
Moreover, because $P$ and $\bar{\Pi}$ are constant matrices and $p$ is a constant vector, the uniform boundedness of $P_T(\cd)$, $\bar{\Pi}_T(\cd)$, and $p_T(\cd)$ is immediate.
\end{proof}

The corresponding estimates for the coefficient functions associated with the optimal pairs $(X_T(\cd), u_T(\cd))$ and $(\bar{X}(\cd), \bar{u}(\cd))$ follow directly.

\begin{corollary}
\label{c:convergence_coefficients_control}
Let \textnormal{\textbf{(H1)-(H2)}} hold. There exist positive constants $K$ and $\lambda$, independent of $t$ and $T$, such that
\begin{equation}
\label{eq:convergence_coefficient_control}
\big\|\Theta_T^*(t) - \Theta^* \big\| + \big\|\bar{\Theta}_T^*(t) - \bar{\Theta}^* \big\| + \big|\theta_T^*(t) - \theta^* \big| \les K e^{-\lambda (T-t)}, \quad \forall t \in [0, T],
\end{equation}
where $(\Theta_T^*(\cd), \bar{\Theta}_T^*(\cd), \theta_T^*(\cd))$ and $(\Theta^*, \bar{\Theta}^*, \theta^*)$ are defined in \eqref{eq:theta_T_star} and \eqref{eq:theta_star}, respectively.
\end{corollary}
\begin{proof}
The claim follows by applying the same argument as in Theorem 4.1 of \cite{Bayraktar-Jian-2025}.
\end{proof}

We now state the first main result of this paper, which establishes a turnpike property characterizing the connection between the optimal pair of \textbf{Problem (MFC)$^T$} and that of \textbf{Problem (EMFC)}.

\begin{theorem}
\label{t:turnpike_property}
Suppose that \textnormal{\textbf{(H1)-(H2)}} hold. Let $(X_T(\cd), u_T(\cd))$ denote the optimal pair for \textnormal{\textbf{Problem (MFC)$^T$}} with initial state $\xi \in L^4_{\cF_0}$ satisfying $\cL(\xi) = \mu_0 \in \cP_4(\dbR^n)$, and let $(\bar X(\cd), \bar u(\cd))$ be the optimal pair for \textnormal{\textbf{Problem (EMFC)}} with an arbitrary initial state $\bar{\xi} \in L^4_{\cF_0}$ such that $\cL(\bar{\xi}) = \mu_0$, where $\bar{\xi}$ is independent of $W(\cd)$ and $W^0(\cd)$. Then, there exist positive constants $K$ and $\lambda$, independent of $T$, such that
\begin{equation}
\label{eq:strong_turnpike}
\dbE \big[|X_{T}(t) - \bar X (t)|^2
+ |u_{T} (t) - \bar u (t)|^2
\big] \les K \big(e^{-\lambda t} + e^{-\lambda(T - t)} \big), \quad \forall t \in [0, T].
\end{equation}
Moreover, if $\xi = \bar{\xi}$, we have
\begin{equation}
\label{eq:strong_turnpike_2}
\dbE \big[|X_{T}(t) - \bar X (t)|^2
+ |u_{T} (t) - \bar u (t)|^2
\big] \les K e^{-\lambda(T - t)}, \quad \forall t \in [0, T].
\end{equation}
\end{theorem}

\begin{proof}
For any $t \in [0, T]$, we denote
$\h X(t) = X_T(t) - \bar X(t)$. From the closed-loop system \eqref{eq:optimal_path_MFC} for $X_T(\cd)$ and the dynamics in \eqref{eq:optimal_path_ergodic} for $\bar{X}(\cd)$, we obtain
\begin{equation*}
\begin{aligned}
d \h X(t) &= \big\{(A + B \Th_T^*(t))(\h X(t) - \dbE[\h X(t) |\cF_t^0]) + B(\Th_T^*(t) - \Th^*) (\bar X(t) - \dbE[\bar X(t)|\cF_t^0]) \\
& \hspace{0.5in} + (A + \bar A + B \bar \Th_T^*(t)) \dbE[\h X(t)|\cF_t^0] + B(\bar \Th_T^*(t) - \bar \Th^*) \dbE[\bar X(t)|\cF_t^0] + B(\theta_T^*(t) - \theta^*) \big\} dt \\
& \hspace{0.2in} + \big\{(C + D \Th_T^*(t))(\h X(t) - \dbE[\h X(t)|\cF_t^0]) + D(\Th_T^*(t) - \Th^*) (\bar X(t) - \dbE[\bar X(t)|\cF_t^0]) \\
& \hspace{0.5in} + (C + \bar C + D \bar \Th_T^*(t)) \dbE[\h X(t)|\cF_t^0] + D(\bar \Th_T^*(t) - \bar \Th^*) \dbE[\bar X(t)|\cF_t^0] + D(\theta_T^*(t) - \theta^*) \big\} dW(t) \\
& \hspace{0.2in} + \big\{\Gamma \h X(t) + \bar{\Gamma} \dbE[\h X(t)|\cF_t^0] \big\} dW^0(t) 
\end{aligned} 
\end{equation*}
with the initial state
$\h X(0) = X_T(0) - \bar X(0) = \xi - \bar{\xi}$. Consequently, the conditional expectation $\dbE[\h X(t)|\cF_t^0]$ satisfies the following SDE
\begin{equation*}
\begin{aligned}
d \dbE[\h X(t)|\cF_t^0] &= \big\{(A + \bar A + B \bar \Th^*) \dbE[\h X(t)|\cF_t^0] + B(\bar \Th_T^*(t) - \bar \Th^*) \dbE[\h X(t)|\cF_t^0] \\
& \hspace{0.5in} + B(\bar \Th_T^*(t) - \bar \Th^*) \dbE[\bar X(t)|\cF_t^0] + B(\theta_T^*(t) - \theta^*)\big\} dt \\
& \hspace{0.3in} + (\Gamma + \bar{\Gamma}) \dbE[\h X(t)|\cF_t^0] dW^0(t)
\end{aligned}
\end{equation*}
with the initial condition
$\dbE[\h X(0)|\cF_0^0] = \dbE[\xi-\bar{\xi} |\cF_0^0] = \dbE[\xi-\bar{\xi}] = 0$. Since the system $[A + \bar{A} + B \bar{\Theta}^*, \Gamma + \bar{\Gamma}]$ is $L^2$-exponentially stable, Proposition 3.6 of \cite{Huang-Li-Yong-2015} and an argument similar to that of Lemma \ref{l:moment_boundedness} show that the Lyapunov equation
$$P_1 (A + \bar{A} + B \bar{\Theta}^*) +(A + \bar{A} + B \bar{\Theta}^*)^\top P_1 + (\Gamma + \bar{\Gamma})^\top P_1 (\Gamma + \bar{\Gamma}) + I_n = 0$$
admits a unique solution $P_1 \in \dbS^n_{++}$. Applying It\^o's formula, we have
\begin{equation*}
\begin{aligned}
& \frac{d}{dt} \dbE\big[ \big\lan P_1 \dbE \big[\h X(t)|\cF_t^0 \big], \dbE \big[\h X(t)|\cF_t^0 \big] \big\ran \big] \\
= \ & \dbE\Big[\big\lan \big(P_1 (A + \bar{A} + B \bar{\Theta}^*) +(A + \bar{A} + B \bar{\Theta}^*)^\top P_1 + (\Gamma + \bar{\Gamma})^\top P_1 (\Gamma + \bar{\Gamma})\big) \dbE \big[\h X(t)|\cF_t^0 \big], \dbE \big[\h X(t)|\cF_t^0 \big] \big\ran \\
& \hspace{0.5in} + \big\lan \big\{ P_1 B (\bar{\Theta}_T^*(t) - \bar{\Theta}^*) + (\bar{\Theta}_T^*(t) - \bar{\Theta}^*)^\top B^\top P_1 \big\} \dbE \big[\h X(t)|\cF_t^0 \big], \dbE \big[\h{ X}(t)|\cF_t^0 \big] \big\ran \\
& \hspace{0.5in} + \big\lan 2 P_1 B (\bar{\Theta}_T^*(t) - \bar{\Theta}^*) \dbE \big[\bar X(t)|\cF_t^0 \big] + 2 P_1 B(\theta_T^*(t) - \theta^*), \dbE \big[\h X(t)|\cF_t^0 \big] \big\ran \\
\les \ & \dbE \Big[ - \big|\dbE \big[\h X(t)|\cF_t^0 \big]\big|^2 + 2 \|P_1\| \|B\| \big\{\|\bar{\Theta}_T^*(t) - \bar{\Theta}^*\| \big|\dbE \big[\h X(t)|\cF_t^0 \big] \big|^2 + |\theta_T^*(t) -\theta^*| \big|\dbE \big[\h X(t)|\cF_t^0 \big] \big| \\
& \hspace{0.5in} + \|\bar{\Theta}_T^*(t) - \bar{\Theta}^*\| \big| \dbE \big[\bar{X}(t)|\cF_t^0 \big] \big| \big| \dbE \big[\h X(t)|\cF_t^0 \big] \big| \big\} \Big].
\end{aligned}
\end{equation*}
Then, using Young's inequality, the moment boundedness result in Lemma \ref{l:moment_boundedness}, and the estimates in \eqref{eq:convergence_coefficient_control}, we deduce that there exist constants $K> 0$ and $\lambda > 0$, independent of $t$ and $T$, such that
\begin{equation*}
\frac{d}{dt} \dbE\big[ \big\lan P_1 \dbE \big[\h X(t)|\cF_t^0 \big], \dbE \big[\h X(t)|\cF_t^0 \big] \big\ran \big] \les \Big(-\frac{1}{2} + K e^{-\lambda(T-t)} \Big) \dbE \big[\big|\dbE \big[\h X(t)|\cF_t^0 \big]\big|^2 \big] + K e^{-\lambda (T-t)}
\end{equation*}
for all $t \in [0, T]$. Let $\beta_1$ and $\beta_2$ be the largest and the smallest eigenvalues of $P_1$, respectively. Then, it is clear that
$$\beta_2 \dbE \big[\big|\dbE \big[\h X(t)|\cF_t^0 \big]\big|^2 \big] \les \dbE\big[ \big\lan P_1 \dbE \big[\h X(t)|\cF_t^0 \big], \dbE \big[\h X(t)|\cF_t^0 \big] \big\ran \big] \les \beta_1 \dbE \big[\big|\dbE \big[\h X(t)|\cF_t^0 \big]\big|^2 \big]$$
for all $t \in [0, T]$. Thus,
\begin{equation*}
\begin{aligned} 
& \frac{d}{dt} \dbE\big[ \big\lan P_1 \dbE \big[\h X(t)|\cF_t^0 \big], \dbE \big[\h X(t)|\cF_t^0 \big] \big\ran \big] \\
\les \ & \Big(-\frac{1}{2 \beta_1} + \frac{K}{\beta_2} e^{-\lambda(T-t)} \Big) \dbE\big[ \big\lan P_1 \dbE \big[\h X(t)|\cF_t^0 \big], \dbE \big[\h X(t)|\cF_t^0 \big] \big\ran \big] + K e^{-\lambda (T-t)},
\end{aligned}
\end{equation*}
which yields
\begin{equation*}
\dbE\big[ \big\lan P_1 \dbE \big[\h X(t)|\cF_t^0 \big], \dbE \big[\h X(t)|\cF_t^0 \big] \big\ran \big]
\les \int_0^t e^{ \int_s^t \big(-\frac{1}{2 \beta_1} + \frac{K}{\beta_2} e^{-\lambda(T-r)} \big) dr} K e^{-\lambda (T-s)} ds
\les K e^{-\lambda(T-t)}
\end{equation*}
for all $t \in [0, T]$ with possibly different constant $K$. Then, once again, since $P_1$ is positive definite, we conclude the following estimate:
\begin{equation}
\label{eq:moment_bounded_X_hat_mean}
\dbE \big[\big|\dbE \big[\h X(t)|\cF_t^0 \big]\big|^2 \big] \les K e^{-\lambda (T-t)}, \quad \forall t \in [0, T],
\end{equation}
for some positive constants $K$ and $\lambda$, independent of $T$.

Next, we show the turnpike estimate for $\dbE[|\h{X}(t)|^2]$. Let $P$ be the positive definite solution to the algebraic equation for $P$ in \eqref{eq:Riccati_ergodic}. By It\^o's formula,
\begin{equation*}
\begin{aligned}
& \frac{d}{dt} \dbE \big[\lan P \h X(t), \h X(t)\ran \big] \\
= \ & \dbE \big[\lan \{P(A + B \Th^*_T(t)) + (A + B \Th^*_T(t))^\top P + (C + D \Th^*_T(t))^\top P (C + D \Th^*_T(t)) + \Gamma^\top P \Gamma\} \h X(t), \h X(t) \ran  \\
& \hspace{0.5in} + 2 \lan \h X(t), H_2(t) \ran + \lan P H_1(t), H_1(t) \ran + \bar{\Gamma}^\top P \bar{\Gamma} \big|\dbE \big[\h X(t)|\cF_t^0 \big]\big|^2 \big],
\end{aligned}
\end{equation*}
where
\begin{equation*}
\begin{aligned}
H_1(t) &:= \big(\bar C + D \bar{\Th}_T^*(t) - D \Th_T^*(t) \big) \dbE \big[\h X(t)|\cF_t^0 \big] + D(\Th_T^*(t) - \Th^*) \big(\bar X(t) - \dbE \big[\bar{X}(t)|\cF_t^0 \big] \big) \\
& \hspace{0.5in} + D(\bar{\Th}_T^*(t) - \bar{\Th}^*) \dbE \big[\bar{X}(t)|\cF_t^0 \big] + D (\theta_T^*(t) - \theta^*),
\end{aligned}
\end{equation*}
and
\begin{equation*}
\begin{aligned}
H_2(t) &:= P\big(\bar A + B \bar{\Th}_T^*(t) - B \Th_T^*(t)\big) \dbE \big[\h X(t)|\cF_t^0 \big] + PB(\Th_T^*(t) - \Th^*) \big(\bar X(t) - \dbE \big[\bar{X}(t)|\cF_t^0 \big] \big) \\
& \hspace{0.5in} + PB(\bar{\Th}_T^*(t) - \bar{\Th}^*) \dbE \big[\bar{X}(t)|\cF_t^0 \big] + PB (\theta_T^*(t) - \theta^*) + (C + D \Th^*_T(t))^\top P H_1(t) \\
& \hspace{0.5in} + \Gamma^\top P \bar{\Gamma} \dbE \big[\h X(t)|\cF_t^0 \big].
\end{aligned}
\end{equation*}
Using the inequality \eqref{eq:moment_bounded_X_hat_mean} together with the results from Lemma \ref{l:moment_boundedness} and Corollary \ref{c:convergence_coefficients_control} again, we get the following estimate for $H_1(t)$ and $H_2(t)$:
\begin{equation*}
\begin{aligned}
\dbE \big[|H_1(t)|^2 + |H_2(t)|^2 \big] &\les K \big(\dbE \big[\big|\dbE \big[\h X(t)|\cF_t^0 \big]\big|^2 \big] + \|\Th_T^*(t) - \Th^*\|^2 + \|\bar{\Th}_T^*(t) - \bar{\Th}^* \|^2 + |\theta_T^*(t) - \theta^*|^2 \big) \\
& \les K e^{-\lambda (T-t)}, \quad \forall t \in [0, T]
\end{aligned}
\end{equation*}
for some $K > 0$ and $\lambda > 0$. Moreover, we observe that
\begin{equation*}
\begin{aligned}
& P(A+B\Th^*_T(t)) + (A+B\Th_T^*(t))^\top P
+(C+D\Th_T^*(t))^\top P (C+D\Th_T^*(t)) \\
= \ & P(A+B\Th^*) + (A+B\Th^*)^\top P + (C+D\Th^*)^\top P(C+D\Th^*) + H_3(t),
\end{aligned}
\end{equation*}
and
\begin{equation*}
\begin{aligned}
& P(A+B\Th^*) + (A+B\Th^*)^\top P + (C+D\Th^*)^\top P(C+D\Th^*) + \Gamma^\top P \Gamma \\
= \ & -(Q + S^\top \Th^* + (\Th^*)^\top S+(\Th^*)^\top R \Th^*) < 0,
\end{aligned}
\end{equation*}
where
\begin{equation*}
\begin{aligned}
H_3(t) &= PB(\Th_T^*(t) -\Th^*) +(B(\Th_T^*(t)-\Th^*))^\top P + (D(\Th_T^*(t) - \Th^*))^\top P (C+D\Th^*) \\
& \hspace{0.5in} +(C+D\Th_T^*(t))^\top PD(\Th_T^*(t)-\Th^*).
\end{aligned}
\end{equation*}
Using the estimates in Corollary \ref{c:convergence_coefficients_control}, we derive that there exist constants $K > 0$ and $\lambda > 0$ such that
\begin{equation*}
\|H_3(t)\| \les K e^{-\lambda (T - t)}, \quad \forall t \in [0, T].  
\end{equation*}
Let $\beta_3 > 0$ and $\beta_4 > 0$ denote the smallest and the largest eigenvalues of $P$, respectively, and let $\h{\beta}$ be the smallest eigenvalue of $Q + S^\top \Th^* + (\Th^*)^\top S+(\Th^*)^\top R \Th^*$. Then, it is straightforward that
\begin{equation}
\label{eq:eigen_inequality}
\beta_3 \dbE\big[|\h X(t)|^2\big] \les \dbE \big[\lan P \h X(t), \h X(t)\ran \big] \les \beta_4 \dbE\big[|\h X(t)|^2\big], \quad \forall t \in [0, T].
\end{equation}
Combining the above results and applying Young's inequality, we find that
\begin{equation*}
\begin{aligned}
\frac{d}{dt} \dbE \big[\lan P \h X(t), \h X(t)\ran \big] & \les \dbE \Big[\big(- \h{\beta} + Ke^{-\lambda(T-t)} \big) |\h X(t)|^2 + 2 |H_2(t)| |\h X(t)| + K|H_1(t)|^2 \\
& \hspace{0.5in} + K \big|\dbE \big[\h X(t)|\cF_t^0 \big]\big|^2 \Big] \\
& \les \Big(- \frac{\h{\beta}}{2} + Ke^{-\lambda(T-t)} \Big) \dbE \big[|\h X(t)|^2\big] + K e^{-\lambda(T-t)}
\end{aligned}
\end{equation*}
for all $t \in [0, T]$. By the inequality \eqref{eq:eigen_inequality},
$$\frac{d}{dt} \dbE \big[\lan P \h X(t), \h X(t)\ran \big] \les \Big(- \frac{\h{\beta}}{2 \beta_4} + \frac{K}{\beta_3} e^{-\lambda(T-t)} \Big) \dbE \big[\lan P \h X(t), \h X(t)\ran \big] + K e^{-\lambda(T-t)}.$$
Set $g(t) := - \frac{\h\beta}{2 \beta_4} + \frac{K}{\beta_3} e^{-\lambda(T-t)}$ for all $t \in [0, T]$. Then, for all $0 \les s \les t \les T$, the following estimate holds:
$$\exp \Big\{\int_s^t g(r) dr\Big\} = \exp \Big\{\frac{K}{\lambda \beta_3} e^{-\lambda T}(e^{\lambda t} - e^{\lambda s}) - \frac{\h{\beta}}{2 \beta_4}(t-s) \Big\} \les \exp \Big\{\frac{K}{\lambda \beta_3} - \frac{\h{\beta}}{2 \beta_4}(t-s) \Big\}.$$
Thus, we have
$$\int_0^t e^{\int_s^t g(r) dr} K e^{-\lambda (T-s)} ds \les K e^{-\lambda (T-t)}, \quad \forall t \in [0, T],$$
with possibly different constants $K, \lambda > 0$.
It follows that
\begin{equation*}
\begin{aligned}
\dbE \big[\lan P \h X(t), \h X(t)\ran \big] & \les \dbE \big[\lan P \h X(0), \h X(0)\ran \big] e^{\int_0^t g(r) dr} + \int_0^t e^{\int_s^t g(r) dr} K e^{-\lambda (T-s)} ds \\
& \les K \big(e^{-\lambda t} + e^{-\lambda(T-t)} \big)
\end{aligned}
\end{equation*}
for all $t \in [0, T]$. Moreover, if $\xi = \bar{\xi}$, then $\h{X}(0) = 0$, which implies that
$$\dbE \big[\lan P \h X(t), \h X(t)\ran \big] \les K e^{-\lambda(T-t)}, \quad \forall t \in [0, T].$$
Hence, since $P$ is positive definite, there exist positive constants $K$ and $\lambda$, independent of $T$, such that
\begin{equation*}
\dbE \big[|\h X(t)|^2 \big] \les K \big(e^{-\lambda t} + e^{-\lambda(T-t)} \big), \quad \forall t \in [0, T].
\end{equation*}
Next, the preceding estimate and the moment bound for $\bar X(\cd)$ in Lemma \ref{l:moment_boundedness} yield
$$\dbE \big[|X_T(t)|^2 \big] \les K, \quad \forall t \in [0, T].$$
Note that for all $t \in [0, T]$, the difference $u_T(t) - \bar u(t)$ satisfies the equality
\begin{equation*}
\begin{aligned}
u_T(t) - \bar u(t) &= \Th_T^*(t) \h X(t) + (\Th_T^*(t) - \Th^*) \bar X(t) + (\bar{\Th}_T^*(t) - \Th_T^*(t)) \dbE \big[\h X(t)|\cF_t^0 \big] \\
& \hspace{0.5in} + (\bar{\Th}_T^*(t) - \bar{\Th}^* + \Th^* - \Th_T^*(t)) \dbE \big[\bar{X}(t)|\cF_t^0 \big] + \theta_T^*(t) - \theta^*.
\end{aligned}
\end{equation*}
By applying the results in Lemma \ref{l:moment_boundedness} and Corollary \ref{c:convergence_coefficients_control} again, we conclude that the turnpike property for the optimal control holds. Specifically,
\begin{equation*}
\dbE \big[|u_T(t) - \bar u(t)|^2 \big] \les K \big(e^{-\lambda t} + e^{-\lambda (T-t)} \big), \quad \forall t \in [0, T].
\end{equation*}
In addition, if $\xi = \bar{\xi}$, the above estimate can be reduced to
\begin{equation*}
\dbE \big[|u_T(t) - \bar u(t)|^2 \big] \les K e^{-\lambda (T-t)}, \quad \forall t \in [0, T].
\end{equation*}
This completes the proof. 
\end{proof}

\begin{corollary}
\label{c:convergence_of_value_function_finite_ergodic}
Assume that \textnormal{\textbf{(H1)-(H2)}} hold. Let $J_{T}(\mu_0; \bar u(\cd))$ denote the cost functional \eqref{eq:cost_finite_N_infinity} evaluated along the optimal control
$\bar u(\cd)$ in \eqref{eq:optimal_path_ergodic} for \textnormal{\textbf{Problem (EMFC)}} over $[0, T]$, i.e.,
$$J_T(\mu_0; \bar u(\cd)) := \dbE \Big[\int_0^T f\big(\bar X(s), \dbE[\bar X(s)|\cF_s^0], \bar u(s) \big) d s\Big].$$
Then, for all $\mu_0 \in \cP_4(\dbR^n)$, the following estimate holds:
$$0 \les J_T(\mu_0; \bar u(\cd)) - U_T(\mu_0) = O(1).$$
Furthermore, the constant $c_0$ in \eqref{eq:c_0} coincides with the long-run average value: for all $\mu_0 \in \cP_4(\dbR^n)$,
$$c_0 = \lim_{T\to \infty} \frac{1}{T} U_{T}(\mu_0).$$
\end{corollary}
\begin{proof}
The proof follows the same argument as in Theorem 6.5 of \cite{Jian-Jin-Song-Yong-2024}.
\end{proof}

\section{Uniform-in-time convergence of the social optimization problem}
\label{s:convergence_of_SO_N}

In this section, we investigate the convergence of \textbf{Problem (SO-N)$^T$} in \eqref{eq:SO-N_finite} to \textbf{Problem (MFC)$^T$} in \eqref{eq:value_N_infinity}, and further to
\textbf{Problem (EMFC)} in \eqref{eq:ergodic_control_problem}. The essential step is to derive suitable quantitative estimates linking the solution of the system \eqref{eq:Riccati_N_agent_rescaling}, associated with the social optimization problem, to the solutions of the systems \eqref{eq:Riccati_MC} and \eqref{eq:Riccati_ergodic}, arising from the finite-horizon mean field control problem and its ergodic counterpart, respectively.

The following proposition provides such estimates, comparing the solutions of \eqref{eq:Riccati_N_agent_rescaling} and \eqref{eq:Riccati_MC}. Importantly, these estimates are uniform with respect to the time horizon. The detailed proof is given in Appendix \ref{s:proof_convergence_Riccati_N_infinity}.

\begin{proposition}
\label{p:convergence_Riccati_N_infinity}
Assume that \textnormal{\textbf{(H1)-(H2)}} hold. Let $(P_T^N(\cd), \Pi_T^N(\cd), p_T^N(\cd))$ be the solution to the system of ODEs \eqref{eq:Riccati_N_agent_rescaling}, and let $(P_T(\cd), \bar{\Pi}_T(\cd), p_T(\cd))$ be the solution to the system \eqref{eq:Riccati_MC}. Then, there exist constants $N^* > 0$ and $K > 0$, independent of $T$ and $N$, such that for all $N > N^*$, the following estimates hold:
\begin{equation}
\label{eq:convergence_Riccati_N_infinity}
\big\|P_T^N(t) - P_T(t) \big\| + \big\|\Pi_T^N(t) - \Pi_T(t) \big\| + \big|p_T^N(t) - p_T(t) \big| \les \frac{K}{N}, \quad \forall t \in [0, T],
\end{equation}
where $\Pi_T(t) = \bar{\Pi}_T(t) - P_T(t)$.
\end{proposition}

Proposition \ref{p:convergence_Riccati_N_infinity} yields the following estimates.

\begin{corollary}
\label{c:convergence_Theta_N_infinity}
Assume that \textnormal{\textbf{(H1)-(H2)}} hold. Then, there exist constants $N^* > 0$ and $K > 0$, independent of $T$ and $N$, such that for all $N > N^*$, the following estimates hold:
\begin{equation}
\label{eq:convergence_Theta_N_infinity}
\big\|\Theta_T^N(t) - \Theta^*_T(t) \big\| + \big\|\bar{\Theta}_T^N(t) - \bar{\Theta}^*_T(t) \big\| + \big|\theta_T^N(t) - \theta^*_T(t) \big| \les \frac{K}{N}, \quad \forall t \in [0, T],
\end{equation}
and
\begin{equation*}
\big\|\Theta_T^N(t) \big\| + \big\|\bar{\Theta}_T^N(t) \big\| + \big|\theta_T^N(t) \big| + \big\|\bar{\theta}_T^N(t) \big\| \les K, \quad \forall t \in [0, T],
\end{equation*}
where $(\Theta_T^N(\cd), \bar{\Theta}_T^N(\cd), \theta_T^N(\cd), \bar{\theta}_T^N(\cd))$ is defined in \eqref{eq:theta_T_N}, and $(\Theta^*_T(\cd), \bar{\Theta}^*_T(\cd), \theta^*_T(\cd))$ is given in \eqref{eq:theta_T_star}.
\end{corollary}
\begin{proof}
The desired results follow by an argument similar to that of Corollary \ref{c:convergence_coefficients_control}, together with the estimates \eqref{eq:uniform_bound_P_N}, \eqref{eq:uniform_bound_p_N}, and \eqref{eq:convergence_Riccati_N_infinity}.
\end{proof}

We now establish a convergence result comparing the optimal pair for \textbf{Problem (SO-N)$^T$} with that for \textbf{Problem (MFC)$^T$}. To proceed, we impose the following assumption on the initial state of each agent.

\ms

{\bf(H3)} \textit{$(\xi^1, \dots, \xi^N)$ are i.i.d. $\dbR^n$-valued random variables with common law $\mu_0 \in \cP_4(\dbR^n)$, and are independent of the collection of Brownian motions $\{W^i(t): i = 0, 1, \dots, N\}$.}

\begin{theorem}
\label{t:convergence_path_social_optimal_MFC}
Assume that \textnormal{\textbf{(H1)-(H2)-(H3)}} hold. For each $i \in \{1, \dots, N\}$, let $(X_T^{i, N}(\cd), u_T^{i, N}(\cd))$ denote the optimal pair for the $i$-th agent in \textnormal{\textbf{Problem (SO-N)$^T$}} with initial state $\xi^i$, and let $(X_T^i(\cd), u_T^i(\cd))$ be the corresponding realization of the optimal pair for \textnormal{\textbf{Problem (MFC)$^T$}}, driven by the same Brownian motion $W^i(\cd)$ and with the same initial state $\xi^i$. Then, there exist constants $N^* > 0$ and $K > 0$, independent of $T$ and $N$, such that for all $N > N^*$, 
\begin{equation}
\label{eq:convergence_social_optima_MFC}
\sup_{i \in \{1, \dots, N\}} \sup_{t \in [0, T]} \wt{\dbE} \big[|X_T^{i, N}(t) - X_T^i(t)|^2
+ |u_T^{i, N}(t) - u_T^i(t)|^2
\big] \les \frac{K}{N}.
\end{equation}
\end{theorem}
\begin{proof}
By Proposition \ref{p:solvability_SO_N} (i), the optimal state for agent $i$ satisfies
\begin{equation*}
\begin{aligned}
d X_T^{i, N}(t) & = \Big\{\Big(A + B \Theta_T^N(t) + \frac{1}{N} B \bar{\theta}_T^N(t) \Big) X_T^{i, N}(t) + \big(\bar{A} + B (\bar{\Theta}_T^N(t) - \Theta_T^N(t))  \big) X_T^{(N)}(t) \\
& \hspace{0.3in} + B \theta_T^N(t) + b \Big\} dt + \Big\{\Big(C + D \Theta_T^N(t) + \frac{1}{N} D \bar{\theta}_T^N(t) \Big) X_T^{i, N}(t) \\
& \hspace{0.3in} + \big(\bar{C} + D (\bar{\Theta}_T^N(t) - \Theta_T^N(t)) \big) X_T^{(N)}(t) + D \theta_T^N(t) + \sigma \Big\} dW^i(t) \\
& \hspace{0.3in} + \big(\Gamma X_T^{i, N}(t) + \bar{\Gamma} X_T^{(N)}(t) + \gamma \big) dW^0(t)
\end{aligned}
\end{equation*}
with $X_T^{i,N}(0) = \xi^i$. Thus, the weakly coupled state average $X_T^{(N)}(\cd)$ satisfies the following SDE:
\begin{equation*}
\begin{aligned}
d X_T^{(N)}(t) & = \Big\{\big(A + \bar{A} + B \bar{\Theta}_T^N(t) \big) X_T^{(N)}(t) + \frac{1}{N} B \bar{\theta}_T^N(t) X_T^{(N)}(t) + B \theta_T^N(t) + b \Big\} dt \\
& \hspace{0.3in}  + \big((\Gamma + \bar{\Gamma}) X_T^{(N)}(t) + \gamma \big) dW^0(t) + \frac{1}{N}\sum_{i=1}^N \Big\{\Big(C + D \Theta_T^N(t) + \frac{1}{N} D \bar{\theta}_T^N(t) \Big) X_T^{i, N}(t) \\
& \hspace{0.3in} + \big(\bar{C} + D (\bar{\Theta}_T^N(t) - \Theta_T^N(t)) \big) X_T^{(N)}(t) + D \theta_T^N(t) + \sigma \Big\} dW^i(t)
\end{aligned}
\end{equation*}
with the initial condition $X_T^{(N)}(0) = \frac{1}{N} \sum_{i=1}^{N} \xi^i$.

First, by an argument similar to that of Lemma \ref{l:moment_boundedness}, based on the corresponding Lyapunov equations for the $L^2$-stabilizable systems $[A + \bar{A} + B\bar{\Theta}^*, \Gamma + \bar{\Gamma}]$ and $[A+B\Theta^*, C+D\Theta^*, \Gamma]$, we establish the uniform-in-time estimates for $\wt{\dbE}[|X_T^{i, N}(t)|^2]$ and $\wt{\dbE}[|X_T^{(N)}(t)|^2]$, i.e., for all $N > N^{*}$,
\begin{equation}
\label{eq:estimate_moment_state_average_agent}
\sup_{i \in \{1, \dots, N\}} \wt{\dbE} \big[|X_T^{i, N}(t)|^2 \big] + \wt{\dbE} \big[|X_T^{(N)}(t)|^2 \big] \les K, \quad \forall t \in [0, T],
\end{equation}
where $K > 0$ is a constant that is independent of $T$ and $N$.

Next, we provide the estimate for
$\wt{\dbE}[|X_T^{(N)}(t) - \wt{\dbE}[X_T^i(t)|\cF_t^0]|^2]$. We denote $M_T^{i, N}(t) :=  X_T^{(N)}(t) - \wt{\dbE}[X_T^i(t)|\cF_t^0]$ for all $t \in [0, T]$. It is clear that $M_T^{i,N}(\cd)$ satisfies the SDE
\begin{equation*}
\begin{aligned}
d M_T^{i, N}(t) & = \big\{\big(A + \bar{A} + B \bar{\Theta}_T^*(t) \big) M_T^{i, N}(t) + \Sigma_{1}^N(t) \big\} dt \\
& \hspace{0.3in}  + \big(\Gamma + \bar{\Gamma} \big) M_T^{i, N}(t) dW^0(t) + \frac{1}{N} \sum_{j=1}^N \Sigma_2^{j, N}(t) dW^j(t),
\end{aligned}
\end{equation*}
where
$$\Sigma_{1}^{N}(t) := B \big(\bar{\Theta}_T^N(t) - \bar{\Theta}_T^*(t) \big) X_T^{(N)}(t) + \frac{1}{N} B \bar{\theta}_T^N(t) X_T^{(N)}(t) + B \big(\theta_T^N(t) - \theta_T^*(t)\big),$$
and for all $j \in \{1, \dots, N\}$,
\begin{equation*}
\begin{aligned}
\Sigma_{2}^{j, N}(t) &:= \Big(C + D \Theta_T^N(t) + \frac{1}{N} D \bar{\theta}_T^N(t) \Big) X_T^{j, N}(t) \\
& \hspace{0.3in} + \big(\bar{C} + D (\bar{\Theta}_T^N(t) - \Theta_T^N(t)) \big) X_T^{(N)}(t) + D \theta_T^N(t) + \sigma.
\end{aligned}
\end{equation*}
Let $P_1 \in \dbS^n_{++}$ be the unique stabilizing solution to the Lyapunov equation
$$P_1 (A + \bar{A} + B \bar{\Theta}^*) +(A + \bar{A} + B \bar{\Theta}^*)^\top P_1 + (\Gamma + \bar{\Gamma})^\top P_1 (\Gamma + \bar{\Gamma}) + I_n = 0.$$
Applying It\^o's formula, we obtain
\begin{equation*}
\begin{aligned}
& \frac{d}{dt} \wt{\dbE} \big[ \lan P_1 M_T^{i, N}(t), M_T^{i, N}(t) \ran \big] \\
= \ & \wt{\dbE} \Big[\big\lan \big\{P_1 (A + \bar{A} + B \bar{\Theta}_T^*(t)) +(A + \bar{A} + B \bar{\Theta}_T^*(t))^\top P_1 + (\Gamma + \bar{\Gamma})^\top P_1 (\Gamma + \bar{\Gamma}) \big\} M_T^{i, N}(t), M_T^{i, N}(t) \big\ran \\
& \hspace{0.3in} + 2 \lan P_1 M_T^{i, N}(t), \Sigma_1^N(t) \ran + \frac{1}{N^2} \sum_{j=1}^N \lan P_1 \Sigma_2^{j,N}(t), \Sigma_2^{j,N}(t) \ran \Big] \\
= \ & - \wt{\dbE} \big[|M_T^{i, N}(t)|^2\big] + \wt{\dbE} \Big[\big\lan \{P_1 B (\bar{\Theta}_T^*(t) - \bar{\Theta}^*) + (\bar{\Theta}_T^*(t) - \bar{\Theta}^*)^\top B^\top P_1\} M_T^{i, N}(t), M_T^{i, N}(t) \big\ran \\
& \hspace{0.3in} + 2 \lan P_1 M_T^{i, N}(t), \Sigma_1^N(t) \ran + \frac{1}{N^2} \sum_{j=1}^N \lan P_1 \Sigma_2^{j,N}(t), \Sigma_2^{j,N}(t) \ran \Big].
\end{aligned}
\end{equation*}
By the estimates \eqref{eq:convergence_Theta_N_infinity} and \eqref{eq:estimate_moment_state_average_agent}, we have
$$\wt{\dbE} \big[|\Sigma_1^N(t)|^2\big] \les \frac{K}{N^2}, \quad \text{and} \quad \wt{\dbE} \big[|\Sigma_2^{j,N}(t)|^2\big] \les K, \quad \forall t \in [0, T], \, j \in \{1, \dots, N\}.$$
By Young's inequality and the estimate \eqref{eq:convergence_coefficient_control}, and recalling that $\beta_1$ and $\beta_2$ are the largest and the smallest eigenvalues of $P_1$, respectively,
\begin{equation*}
\begin{aligned}
\frac{d}{dt} \wt{\dbE} \big[ \lan P_1 M_T^{i, N}(t), M_T^{i, N}(t) \ran \big] & \les \Big(-\frac{1}{2} + Ke^{-\lambda(T-t)} \Big) \wt{\dbE} \big[|M_T^{i, N}(t)|^2\big] + \frac{K}{N} \\
& \les \Big(-\frac{1}{2 \beta_1} + \frac{K}{\beta_2} e^{-\lambda(T-t)} \Big) \wt{\dbE} \big[ \lan P_1 M_T^{i, N}(t), M_T^{i, N}(t) \ran \big] + \frac{K}{N}.
\end{aligned}
\end{equation*}
Let $\bar{g}(t) := -\frac{1}{2 \beta_1} + \frac{K}{\beta_2} e^{-\lambda(T-t)}$ for all $t \in [0, T]$. By Gr\"onwall's inequality, we deduce that
\begin{equation*}
\begin{aligned}
\wt{\dbE} \big[ \lan P_1 M_T^{i, N}(t), M_T^{i, N}(t) \ran \big] & = e^{\int_0^t \bar{g}(r) dr} \wt{\dbE}\big[ \lan P_1 M_T^{i, N}(0), M_T^{i, N}(0) \ran \big] + \frac{K}{N} \int_0^t e^{\int_s^t \bar{g}(r) dr} ds \\
& \les K \wt{\dbE} \Big[\Big|\frac{1}{N} \sum_{j=1}^N \xi^j - \wt{\dbE}[\xi^i] \Big|^2\Big] + \frac{K}{N} \les \frac{K}{N}
\end{aligned}
\end{equation*}
for all $t \in [0, T]$ because $(\xi^1, \dots, \xi^N)$ is a sequence of i.i.d. random variables with common law in $\cP_4(\dbR^n)$. Since $P_1$ is positive definite, 
\begin{equation}
\label{eq:estimate_average_state}
\wt{\dbE} \Big[\big|X_T^{(N)}(t) - \wt{\dbE} \big[X_T^i(t) | \cF_t^0 \big] \big|^2 \Big] = \wt{\dbE} \big[|M_T^{i, N}(t)|^2\big] \les \frac{K}{N}, \quad \forall t \in [0, T].
\end{equation}

Next, we use the same idea to establish the estimate for $\wt{\dbE}[|X_T^{i, N}(t) - X_T^i(t)|^2]$. Note that
\begin{equation*}
\begin{aligned}
d \big(X_T^{i, N}(t) - X_T^i(t) \big) &= \big\{\big(A + B \Theta_T^*(t) \big) \big(X_T^{i, N}(t) - X_T^i(t) \big) + \Sigma_3^{i,N}(t) \big\} dt \\
&\hspace{0.3in} + \big\{\big(C + D \Theta_T^*(t) \big) \big(X_T^{i, N}(t) - X_T^i(t) \big) + \Sigma_4^{i,N}(t) \big\} dW^i(t) \\
&\hspace{0.3in} + \big\{\Gamma \big(X_T^{i, N}(t) - X_T^i(t) \big) + \bar{\Gamma} M_T^{i,N}(t) \big\} dW^0(t)
\end{aligned}
\end{equation*}
for $t \in [0, T]$ with the initial condition $X_T^{i, N}(0) - X_T^i(0) = \xi^i - \xi^i = 0$, where
\begin{equation*}
\begin{aligned}
\Sigma_3^{i,N}(t) &:= B \big(\Theta_T^N(t) - \Theta_T^*(t) \big) X_T^{i,N}(t) + \frac{1}{N} B \bar{\theta}_T^N(t) X_T^{i, N}(t) + \big(\bar{A} + B (\bar{\Theta}_T^*(t) - \Theta_T^*(t)) \big) M_T^{i,N}(t) \\
& \hspace{0.3in} + B \big[ \big(\bar{\Theta}_T^N(t) - \bar{\Theta}_T^*(t) \big) - \big(\Theta_T^N(t) - \Theta_T^*(t) \big) \big] X_T^{(N)}(t) + B \big(\theta_T^N(t) - \theta_T^*(t) \big),
\end{aligned}
\end{equation*}
and
\begin{equation*}
\begin{aligned}
\Sigma_4^{i,N}(t) &:= D \big(\Theta_T^N(t) - \Theta_T^*(t) \big) X_T^{i,N}(t) + \frac{1}{N} D \bar{\theta}_T^N(t) X_T^{i, N}(t) + \big(\bar{C} + D (\bar{\Theta}_T^*(t) - \Theta_T^*(t)) \big) M_T^{i,N}(t) \\
& \hspace{0.3in} + D \big[ \big(\bar{\Theta}_T^N(t) - \bar{\Theta}_T^*(t) \big) - \big(\Theta_T^N(t) - \Theta_T^*(t) \big) \big] X_T^{(N)}(t) + D \big(\theta_T^N(t) - \theta_T^*(t) \big).
\end{aligned}
\end{equation*}
From the estimates \eqref{eq:convergence_Theta_N_infinity}, \eqref{eq:estimate_moment_state_average_agent}, and \eqref{eq:estimate_average_state}, it is clear that
\begin{equation}
\label{eq:Estimate_Sigma_3_4}
\sup_{i \in \{1, \dots, N\}} \wt{\dbE} \big[|\Sigma_3^{i,N}(t)|^2 + |\Sigma_4^{i,N}(t)|^2 \big] \les \frac{K}{N}, \quad \forall t \in [0, T]
\end{equation}
for some $K > 0$. By an argument similar to that of Lemma \ref{l:moment_boundedness}, let $P_2 \in \dbS^n_{++}$ be the unique stabilizing solution to the Lyapunov equation 
$$P_2 (A+B\Theta^*) + (A+B\Theta^*)^\top P_2 + (C+D\Theta^*)^\top P_2 (C+D\Theta^*) + \Gamma^\top P_2 \Gamma + I_n = 0.$$
By It\^o's formula and Young's inequality,
\begin{equation*}
\begin{aligned}
& \frac{d}{dt} \wt{\dbE} \big[ \lan P_2 (X_T^{i, N}(t) - X_T^i(t)), X_T^{i, N}(t) - X_T^i(t) \ran \big] \\
= \ & \wt{\dbE} \big[\big\lan \big\{P_2 (A + B \Theta_T^*(t)) + (A + B \Theta_T^*(t))^\top P_2 + (C + D \Theta_T^*(t))^\top P_2 (C + D \Theta_T^*(t)) + \Gamma^\top P_2 \Gamma \big\} \\
& \hspace{0.3in} X_T^{i, N}(t) - X_T^i(t), X_T^{i, N}(t) - X_T^i(t) \big\ran + \lan P_2 \Sigma_4^{i,N}(t), \Sigma_4^{i,N}(t) \ran + \lan \bar{\Gamma}^\top P_2 \bar{\Gamma} M_T^{i,N}(t), M_T^{i,N}(t) \ran \\
& \hspace{0.3in} + 2 \big\lan X_T^{i, N}(t) - X_T^i(t), P_2 \Sigma_3^{i,N}(t) + \Gamma^\top P_2 \bar{\Gamma} M_T^{i,N}(t) + (C+D\Theta_T^*(t))^\top P_2 \Sigma_4^{i,N}(t) \big\ran \big] \\
\les \ & - \frac{1}{2} \wt{\dbE} \big[|X_T^{i, N}(t) - X_T^i(t)|^2\big] + \wt{\dbE} \big[\lan H_4(t) (X_T^{i, N}(t) - X_T^i(t)), X_T^{i, N}(t) - X_T^i(t) \ran \big] \\
& \hspace{0.3in} + 2 \wt{\dbE} \big[|P_2 \Sigma_3^{i,N}(t) + \Gamma^\top P_2 \bar{\Gamma} M_T^{i,N}(t) + (C+D\Theta_T^*(t))^\top P_2 \Sigma_4^{i,N}(t)|^2\big] \\
& \hspace{0.3in} + \wt{\dbE} \big[\lan P_2 \Sigma_4^{i,N}(t), \Sigma_4^{i,N}(t) \ran + \lan \bar{\Gamma}^\top P_2 \bar{\Gamma} M_T^{i,N}(t), M_T^{i,N}(t) \ran \big],
\end{aligned}
\end{equation*}
where
\begin{equation*}
\begin{aligned}
H_4(t) &= P_2 B (\Theta_T^*(t) - \Theta^*) + (\Theta_T^*(t) - \Theta^*)^\top B^\top P_2 + (C + D \Theta_T^*(t))^\top P_2 D (\Theta_T^*(t) - \Theta^*) \\
& \hspace{0.3in} + (\Theta_T^*(t) - \Theta^*)^\top D^\top P_2 (C + D \Theta^*).
\end{aligned}
\end{equation*}
Using the estimate \eqref{eq:convergence_coefficient_control}, we derive that $\|H_4(t)\| \les K e^{-\lambda(T-t)}$ for all $t \in [0, T]$, and thus by \eqref{eq:estimate_average_state} and \eqref{eq:Estimate_Sigma_3_4},
\begin{equation*}
\begin{aligned}
& \frac{d}{dt} \wt{\dbE} \big[ \lan P_2 (X_T^{i, N}(t) - X_T^i(t)), X_T^{i, N}(t) - X_T^i(t) \ran \big] \\
\les \ & \Big(- \frac{1}{2} + Ke^{-\lambda(T-t)} \Big) \wt{\dbE} \big[|X_T^{i, N}(t) - X_T^i(t)|^2\big] + \frac{K}{N}.
\end{aligned}
\end{equation*}
By an argument similar to the proof of \eqref{eq:estimate_average_state}, we conclude that
\begin{equation}
\label{eq:uniform_estimate_X}
\sup_{i \in \{1, \dots, N\}} \sup_{t \in [0, T]} \wt{\dbE} \big[|X_T^{i, N}(t) - X_T^i(t)|^2
\big] \les \frac{K}{N},
\end{equation}
where $K$ is a constant that is independent of $N$ and $T$. By the feedback form of $u_T^{i, N}(\cd)$ in Proposition \ref{p:solvability_SO_N} (i) and $u_T^i(\cd)$ in Proposition \ref{p:solvability_MFC_finite_horizon} (ii), we obtain
\begin{equation*}
\begin{aligned}
u_T^{i, N}(t) - u_T^i(t) &= \Big(\Theta_T^N(t) + \frac{1}{N} \bar{\theta}_T^N(t)\Big) X_T^{i, N}(t) + \big(\bar{\Theta}_T^N(t) - \Theta_T^N(t)\big) X_T^{(N)}(t) + \theta_T^N(t) \\
& \hspace{0.3in} - \big\{\Th_T^*(t) X_T^i(t) + \big(\bar \Th_T^*(t) - \Th_T^*(t) \big) \wt{\dbE}\big[X_T^i(t)|\cF_t^0\big] + \th_T^*(t) \big\} \\
& = \Big(\Theta_T^N(t) - \Theta_T^*(t) + \frac{1}{N} \bar{\theta}_T^N(t)\Big) X_T^{i, N}(t) + \Theta_T^*(t) \big(X_T^{i, N}(t) - X_T^i(t) \big) \\
& \hspace{0.3in} + \big(\bar{\Theta}_T^N(t) - \bar{\Theta}_T^*(t) - (\Theta_T^N(t) - \Theta_T^*(t))\big) X_T^{(N)}(t) + \theta_T^N(t) - \theta_T^*(t) \\
& \hspace{0.3in} + \big(\bar \Th_T^*(t) - \Th_T^*(t) \big) \big(X_T^{(N)}(t) -  \wt{\dbE}\big[X_T^i(t)|\cF_t^0\big] \big). 
\end{aligned}
\end{equation*}
Thus, the estimates \eqref{eq:convergence_Theta_N_infinity}, \eqref{eq:estimate_moment_state_average_agent}, \eqref{eq:estimate_average_state}, and \eqref{eq:uniform_estimate_X} imply that
$$\sup_{i \in \{1, \dots, N\}} \sup_{t \in [0, T]} \wt{\dbE} \big[|u_T^{i, N}(t) - u_T^i(t)|^2
\big] \les \frac{K}{N}.$$
This completes the proof of the desired result.
\end{proof}

\begin{corollary}
\label{c:convergence_both_N_and_T_optimal_pair}
Assume that \textnormal{\textbf{(H1)-(H2)-(H3)}} hold. For each $i \in \{1, \dots, N\}$, let $(X_T^{i, N}(\cd), u_T^{i, N}(\cd))$ denote the optimal pair for the $i$-th agent in \textnormal{\textbf{Problem (SO-N)$^T$}} with initial state $\xi^i$, and let $(\bar{X}^i(\cd), \bar{u}^i(\cd))$ be the corresponding realization of the optimal pair for \textnormal{\textbf{Problem (EMFC)}}, driven by the same Brownian motion $W^i(\cd)$ and
with the same initial state $\xi^i$. Then, there exist $N^* > 0$, and two constants $K, \lambda > 0$, independent of $T$ and $N$, such that if $N > N^*$, 
\begin{equation}
\label{eq:convergence_social_optima_to_EMFC}
\wt{\dbE} \big[|X_T^{i, N}(t) - \bar{X}^i(t)|^2
+ |u_T^{i, N}(t) - \bar{u}^i(t)|^2
\big] \les \frac{K}{N} + K e^{-\lambda (T-t)}, \quad \forall t \in [0, T],
\end{equation}
for all $i \in \{1, \dots, N\}$.
\end{corollary}
\begin{proof}
The estimate follows from the triangle inequality combined with the estimates \eqref{eq:strong_turnpike_2} in Theorem \ref{t:turnpike_property} and \eqref{eq:convergence_social_optima_MFC} in Theorem \ref{t:convergence_path_social_optimal_MFC}.
\end{proof}

In the following proposition, we show that, for all $T > 0$, the optimal control for \textbf{Problem (MFC)$^T$} is asymptotically optimal for \textbf{Problem (SO-N)$^T$} as $N \to \infty$.

\begin{proposition}
\label{p:convergence_value_social_optimal_MFC}
Assume that \textnormal{\textbf{(H1)-(H2)-(H3)}} hold. Let $\wt{u}_T^i(\cd)$ be the optimal control associated with \textnormal{\textbf{Problem (MFC)$^T$}} in feedback form, namely,
$$\wt{u}_T^i(t) = \Th_T^*(t) \wt{X}_T^i(t) + (\bar \Th_T^*(t) - \Th_T^*(t)) \wt{X}_T^{(N)}(t) + \th_T^*(t), \quad t \in [0,T],$$
where $\wt{X}_T^i(\cd)$ is the solution to the corresponding closed-loop system with $\wt{X}_T^i(0) = \xi^i$, $\wt{X}_T^{(N)}(t) = \frac{1}{N} \sum_{i=1}^{N} \wt{X}_T^i(t)$, and the tuple $(\Th_T^{*}(\cd), \bar \Th_{T}^{*} (\cd), \th_T^{*}(\cd))$ is defined in \eqref{eq:theta_T_star}. Let $J_{T}^i(\bm{\xi}; \wt{\bm{u}}_T(\cd))$ denote the cost functional \eqref{eq:cost_finite_N_agent} evaluated along the control $\wt{\bm{u}}_T(\cd):= (\wt{u}_T^1(\cd), \dots, \wt{u}_T^N(\cd))$ over the finite horizon $[0, T]$, i.e.,
$$J_T^i(\bm{\xi}; \wt{\bm{u}}_T(\cd)) := \wt{\dbE} \Big[\int_0^T f\big(\wt{X}_T^i(s), \wt{X}_T^{(N)}(s), \wt{u}_T^i(s) \big) d s\Big].$$
Then, we have the following estimate: there exist constants $N^* > 0$ and $K > 0$, independent of $T$ and $N$, such that for all $N > N^*$, 
\begin{equation*}
\Big|\frac{1}{N} \sum_{i=1}^N J_T^i(\bm{\xi}; \bm{u}_T(\cd)) - \frac{1}{N} \sum_{i=1}^N J_T^i(\bm{\xi}; \wt{\bm{u}}_T(\cd)) \Big| \les \frac{KT}{N}.
\end{equation*}
\end{proposition}

\begin{proof}
The proof follows a standard argument. First, using an approach similar to that of Theorem \ref{t:convergence_path_social_optimal_MFC}, we derive second-moment estimates for $\wt{X}_T^i(t)$ and $\wt{X}_T^{(N)}(t)$ as follows: for all $N > N^*$,
$$\sup_{i \in \{1, \dots, N\}} \sup_{t \in [0, T]} \wt{\dbE} \big[|\wt{X}_T^i(t)|^2 + |\wt{X}_T^{(N)}(t)|^2 \big] \les K$$
for some $K > 0$, independent of $T$ and $N$. Next, for all $i \in \{1, \dots, N\}$ and $t \in [0, T]$, we define
$$Y_T^{i, N} (t) := X_T^{i, N}(t) - \wt{X}_T^{i}(t), \quad Y_T^{(N)}(t) := \frac{1}{N} \sum_{i=1}^{N} Y_T^{i,N}(t).$$
Then, one can show that
$$\sup_{i \in \{1, \dots, N\}} \sup_{t \in [0, T]} \wt{\dbE} \big[|Y_T^{i, N} (t)|^2 + |Y_T^{(N)}(t)|^2\big] \les \frac{K}{N^2}$$
for some $K > 0$. Moreover, by the feedback representations of $u_T^{i,N}(t)$ and $\wt{u}_T^i(t)$, it follows that
$$\sup_{i \in \{1, \dots, N\}} \sup_{t \in [0, T]} \wt{\dbE} \big[ |u_T^{i,N}(t) - \wt{u}_T^i(t)|^2 \big] \les \frac{K}{N^2}.$$
Finally, since the running cost $f(x, \bar{x}, u)$ is quadratic in $x$, $\bar{x}$, and $u$, the preceding estimates and the second-moment bound for $X_T^{i,N}(t)$ and $X_T^{(N)}(t)$ in \eqref{eq:estimate_moment_state_average_agent} show that, for each $i \in \{1, \dots, N\}$,
\begin{equation*}
\begin{aligned}
& \wt{\dbE} \big[f\big(X_T^{i,N}(t), X_T^{(N)}(t), u_T^{i,N}(t)\big) - f\big(\wt{X}_T^i(t), \wt{X}_T^{(N)}(t), \wt{u}_T^i(t) \big)\big] \\
\les \ & K \big(\wt{\dbE} \big[|Y_T^{i, N} (t)|^2 + |Y_T^{(N)}(t)|^2 + |u_T^{i,N}(t) - \wt{u}_T^i(t)|^2 \big] \big)^{\frac{1}{2}} \\
\les \ & \frac{K}{N},
\end{aligned}
\end{equation*}
for all $t \in [0, T]$. It follows that, for all $i \in \{1, \dots, N\}$,
$$|J_T^i(\bm{\xi}; \bm{u}_T(\cd)) - J_T^i(\bm{\xi}; \wt{\bm{u}}_T(\cd))| \les \frac{KT}{N},$$
which concludes the desired result by averaging over $i$.
\end{proof}

Next, we establish an analogous estimate by applying the optimal control associated with \textbf{Problem (EMFC)} to \textbf{Problem (SO-N)$^T$}. This estimate explicitly captures the joint dependence on the population size and the time horizon, thereby linking the finite-horizon and ergodic regimes.

\begin{proposition}
\label{p:convergence_value_social_optimal_EMFC}
Assume that \textnormal{\textbf{(H1)-(H2)-(H3)}} hold. Let $\wt{u}^i(\cd)$ be the optimal control for \textnormal{\textbf{Problem (EMFC)}} in feedback form, i.e.,
$$\wt{u}^i(t) = \Th^* \wt{X}^i(t) + (\bar \Th^* - \Th^*) \wt{X}^{(N)}(t) + \th^*, \quad t \in [0,T],$$
where $\wt{X}^i(\cd)$ is the corresponding closed-loop state process with $\wt{X}^i(0) = \xi^i$, $\wt{X}^{(N)}(t) = \frac{1}{N} \sum_{i=1}^{N} \wt{X}^i(t)$, and the tuple $(\Th^{*}, \bar \Th^{*}, \th^{*})$ is given by \eqref{eq:theta_star}. Let $J_{T}^i(\bm{\xi}; \wt{\bm{u}}(\cd))$ denote the cost functional \eqref{eq:cost_finite_N_agent} evaluated along the control $\wt{\bm{u}}(\cd):= (\wt{u}^1(\cd), \dots, \wt{u}^N(\cd))$ over the finite horizon $[0, T]$, i.e.,
$$J_T^i(\bm{\xi}; \wt{\bm{u}}(\cd)) := \wt{\dbE} \Big[\int_0^T f\big(\wt{X}^i(s), \wt{X}^{(N)}(s), \wt{u}^i(s) \big) d s\Big].$$
Then, there exist constants $N^* > 0$ and $K > 0$, independent of $T$ and $N$, such that for all $N > N^*$, 
\begin{equation*}
\frac{1}{N} \sum_{i=1}^N  \frac{1}{T} \big|J_T^i(\bm{\xi}; \bm{u}_T(\cd)) - J_T^i(\bm{\xi}; \wt{\bm{u}}(\cd)) \big| \les \frac{K}{N} + \frac{K}{T}.
\end{equation*}
\end{proposition}

\begin{proof}
The proof follows an argument similar to that of Proposition \ref{p:convergence_value_social_optimal_MFC}.
First, by an argument similar to that of Lemma \ref{l:moment_boundedness}, based on the corresponding Lyapunov equations for the $L^2$-stabilizable systems $[A + \bar{A} + B\bar{\Theta}^*, \Gamma + \bar{\Gamma}]$ and $[A+B\Theta^*, C+D\Theta^*, \Gamma]$, we establish the uniform-in-time estimates for $\wt{\dbE}[|\wt{X}^{i}(t)|^2]$ and $\wt{\dbE}[|\wt{X}^{(N)}(t)|^2]$, i.e., for all $N > N^{*}$,
$$\sup_{i \in \{1, \dots, N\}} \sup_{t \in [0, T]} \wt{\dbE} \big[|\wt{X}^i(t)|^2 + |\wt{X}^{(N)}(t)|^2 \big] \les K,$$
where $K > 0$ is independent of $T$ and $N$. Next, for all $i \in \{1, \dots, N\}$ and $t \in [0, T]$, we define
$$Z_T^{i, N} (t) := X_T^{i, N}(t) - \wt{X}^{i}(t), \quad Z_T^{(N)}(t) := \frac{1}{N} \sum_{i=1}^{N} Z_T^{i,N}(t).$$
Then, one can show that
$$\sup_{i \in \{1, \dots, N\}} \wt{\dbE} \big[|Z_T^{i, N} (t)|^2 + |Z_T^{(N)}(t)|^2\big] \les \frac{K}{N^2} + K e^{-2\lambda(T-t)}, \quad \forall t \in [0, T],$$
for some positive constants $K$ and $\lambda$. By the feedback form of $u_T^{i,N}(t)$ and $\wt{u}^i(t)$, we have
$$\sup_{i \in \{1, \dots, N\}} \wt{\dbE} \big[ |u_T^{i,N}(t) - \wt{u}^i(t)|^2 \big] \les \frac{K}{N^2} + K e^{-2\lambda(T-t)}, \quad \forall t \in [0, T].$$
Since $f(x, \bar{x}, u)$ is quadratic in $x$, $\bar{x}$, and $u$, the preceding estimates and the second-moment bound for $X_T^{i,N}(t)$ and $X_T^{(N)}(t)$ in \eqref{eq:estimate_moment_state_average_agent} show that, for each $i \in \{1, \dots, N\}$,
\begin{equation*}
\begin{aligned}
& \wt{\dbE} \big[f\big(X_T^{i,N}(t), X_T^{(N)}(t), u_T^{i,N}(t)\big) - f\big(\wt{X}^i(t), \wt{X}^{(N)}(t), \wt{u}^i(t) \big)\big] \\
\les \ & K \big(\wt{\dbE} \big[|Z_T^{i, N} (t)|^2 + |Z_T^{(N)}(t)|^2 + |u_T^{i,N}(t) - \wt{u}^i(t)|^2 \big] \big)^{\frac{1}{2}} \\
\les \ & \frac{K}{N} + K e^{-\lambda(T-t)},
\end{aligned}
\end{equation*}
for all $t \in [0, T]$, which implies that
$$\frac{1}{T} |J_T^i(\bm{\xi}; \bm{u}_T(\cd)) - J_T^i(\bm{\xi}; \wt{\bm{u}}(\cd))| \les \frac{K}{N} + \frac{K}{T}.$$
Averaging over $i$ yields the desired estimate.
\end{proof}

\bibliographystyle{plain}
\bibliography{reference}

@book{Yong-Zhou-1999-stochastic,
  title={Stochastic controls: Hamiltonian systems and HJB equations},
  author={Yong, Jiongmin and Zhou, Xun Yu},
  volume={43},
  year={1999},
  publisher={Springer Science \& Business Media}
}

@article{Yong-2013,
  title={Linear-quadratic optimal control problems for mean-field stochastic differential equations},
  author={Yong, Jiongmin},
  journal={SIAM Journal on Control and Optimization},
  volume={51},
  number={4},
  pages={2809--2838},
  year={2013},
  publisher={SIAM}
}

@article{Sun-Yong-2024,
  title={Turnpike properties for mean-field linear-quadratic optimal control problems},
  author={Sun, Jingrui and Yong, Jiongmin},
  journal={SIAM Journal on Control and Optimization},
  volume={62},
  number={1},
  pages={752--775},
  year={2024},
  publisher={SIAM}
}

@article{Sun-Wang-Yong-2022,
  title={Turnpike properties for stochastic linear-quadratic optimal control problems},
  author={Sun, Jingrui and Wang, Hanxiao and Yong, Jiongmin},
  journal={Chinese Annals of Mathematics, Series B},
  volume={43},
  number={6},
  pages={999--1022},
  year={2022},
  publisher={Springer}
}

@article{Sun-Yong-2024-Periodic,
  title={Turnpike properties for stochastic linear-quadratic optimal control problems with periodic coefficients},
  author={Sun, Jingrui and Yong, Jiongmin},
  journal={Journal of Differential Equations},
  volume={400},
  pages={189--229},
  year={2024},
  publisher={Elsevier}
}

@article{Bayraktar-Jian-2025,
  title={Ergodicity and turnpike properties of linear-quadratic mean field control problems},
  author={Bayraktar, Erhan and Jian, Jiamin},
  journal={arXiv preprint arXiv:2502.08935},
  year={2025}
}

@article{Huang-Zhou-2019-Linear,
  title={Linear quadratic mean field games: Asymptotic solvability and relation to the fixed point approach},
  author={Huang, Minyi and Zhou, Mengjie},
  journal={IEEE Transactions on Automatic Control},
  volume={65},
  number={4},
  pages={1397--1412},
  year={2019},
  publisher={IEEE}
}

@article{Ma-Huang-2020-Linear,
  title={Linear quadratic mean field games with a major player: The multi-scale approach},
  author={Ma, Yan and Huang, Minyi},
  journal={Automatica},
  volume={113},
  pages={108774},
  year={2020},
  publisher={Elsevier}
}

@article{Huang-Yang-2021-Linear,
  title={Linear quadratic mean field social optimization: Asymptotic solvability and decentralized control},
  author={Huang, Minyi and Yang, Xuwei},
  journal={Applied Mathematics \& Optimization},
  volume={84},
  number={Suppl 2},
  pages={1969--2010},
  year={2021},
  publisher={Springer}
}

@article{Jian-Jin-Song-Yong-2024,
  title={Long-time behaviors of stochastic linear-quadratic optimal control problems},
  author={Jian, Jiamin and Jin, Sixian and Song, Qingshuo and Yong, Jiongmin},
  journal={Applied Mathematics \& Optimization},
  volume={93},
  number={2},
  pages={46},
  year={2026},
  publisher={Springer}
}

@article{Huang-Li-Yong-2015,
  title={A linear-quadratic optimal control problem for mean-field stochastic differential equations in infinite horizon},
  author={Huang, Jianhui and Li, Xun and Yong, Jiongmin},
  journal={Mathematical Control and Related Fields},
  volume={5},
  number={1},
  pages={97--139},
  year={2015},
  publisher={American Institute of Mathematical Sciences}
}

@article{Yong-2017,
  title={Linear-quadratic optimal control problems for mean-field stochastic differential equations—time-consistent solutions},
  author={Yong, Jiongmin},
  journal={Transactions of the American Mathematical Society},
  volume={369},
  number={8},
  pages={5467--5523},
  year={2017}
}

@book{Fleming-Soner-2006,
  title={Controlled {M}arkov processes and viscosity solutions},
  author={Fleming, Wendell H and Soner, H Mete},
  year={2006},
  publisher={Springer}
}

@article{Sun-2017,
  title={Mean-field stochastic linear quadratic optimal control problems: {O}pen-loop solvabilities},
  author={Sun, Jingrui},
  journal={ESAIM: Control, Optimisation and Calculus of Variations},
  volume={23},
  number={3},
  pages={1099--1127},
  year={2017}
}

@article{HTK-2017,
  title={Multiperiod network rate allocation with end-to-end delay constraints},
  author={Hajiesmaili, Mohammad Hassan and Talebi, Mohammad Sadegh and Khonsari, Ahmad},
  journal={IEEE Transactions on Control of Network Systems},
  volume={5},
  number={3},
  pages={1087--1097},
  year={2017},
  publisher={IEEE}
}

@book{Moulin-2004,
  title={Fair division and collective welfare},
  author={Moulin, Herv{\'e}},
  year={2004},
  publisher={MIT press}
}

@article{Huang-Caines-Malhame-2012,
  title={Social optima in mean field {LQG} control: centralized and decentralized strategies},
  author={Huang, Minyi and Caines, Peter E and Malham{\'e}, Roland P},
  journal={IEEE Transactions on Automatic Control},
  volume={57},
  number={7},
  pages={1736--1751},
  year={2012},
  publisher={IEEE}
}

@article{Chen-Huang-2019,
  title={Linear-quadratic mean field control: The invariant subspace method},
  author={Chen, Xiang and Huang, Minyi},
  journal={Automatica},
  volume={107},
  pages={582--586},
  year={2019},
  publisher={Elsevier}
}

@article{Wang-Zhang-2017,
  title={Social optima in mean field linear-quadratic-{G}aussian models with {M}arkov jump parameters},
  author={Wang, Bing-Chang and Zhang, Ji-Feng},
  journal={SIAM Journal on Control and Optimization},
  volume={55},
  number={1},
  pages={429--456},
  year={2017},
  publisher={SIAM}
}

@article{Wang-Zhang-2020,
  title={Indefinite linear quadratic mean field social control problems with multiplicative noise},
  author={Wang, Bing-Chang and Zhang, Huanshui},
  journal={IEEE Transactions on Automatic Control},
  volume={66},
  number={11},
  pages={5221--5236},
  year={2020},
  publisher={IEEE}
}

@article{Neumann-1945,
  title={A model of general economic equilibrium},
  author={Neumann, John von},
  journal={The Review of Economic Studies},
  volume={13},
  number={1},
  pages={1--9},
  year={1945},
  publisher={Wiley-Blackwell}
}

@book{Dorfman-Samuelson-Solow-2012,
  title={Linear programming and economic analysis},
  author={Dorfman, Robert and Samuelson, Paul A and Solow, Robert M},
  year={2012},
  publisher={Courier Corporation}
}

@article{Porretta-Zuazua-2013,
  title={Long time versus steady state optimal control},
  author={Porretta, Alessio and Zuazua, Enrique},
  journal={SIAM Journal on Control and Optimization},
  volume={51},
  number={6},
  pages={4242--4273},
  year={2013},
  publisher={SIAM}
}

@article{DGSW-2014,
  title={An exponential turnpike theorem for dissipative discrete time optimal control problems},
  author={Damm, Tobias and Gr\"une, Lars and Stieler, Marleen and Worthmann, Karl},
  journal={SIAM Journal on Control and Optimization},
  volume={52},
  number={3},
  pages={1935--1957},
  year={2014},
  publisher={SIAM}
}

@article{Trelat-Zuazua-2015,
  title={The turnpike property in finite-dimensional nonlinear optimal control},
  author={Tr{\'e}lat, Emmanuel and Zuazua, Enrique},
  journal={Journal of Differential Equations},
  volume={258},
  number={1},
  pages={81--114},
  year={2015},
  publisher={Elsevier}
}

@article{Gugat-Herty-Segala-2024,
  title={The turnpike property for mean-field optimal control problems},
  author={Gugat, Martin and Herty, Michael and Segala, Chiara},
  journal={European Journal of Applied Mathematics},
  volume={35},
  number={6},
  pages={733--747},
  year={2024},
  publisher={Cambridge University Press}
}

@article{Trelat-Zuazua-2025,
  title={Turnpike in optimal control and beyond: a survey},
  author={Tr{\'e}lat, Emmanuel and Zuazua, Enrique},
  journal={arXiv preprint arXiv:2503.20342},
  year={2025}
}

@article{Conforti-2023,
  title={Coupling by reflection for controlled diffusion processes: Turnpike property and large time behavior of {H}amilton--{J}acobi--{B}ellman equations},
  author={Conforti, Giovanni},
  journal={The Annals of Applied Probability},
  volume={33},
  number={6A},
  pages={4608--4644},
  year={2023},
  publisher={Institute of Mathematical Statistics}
}

@article{Mei-Wang-Yong-2025,
  title={Turnpike Property of Stochastic Linear-Quadratic Optimal Control Problems in Large Horizons with Regime Switching {I}: Homogeneous Cases},
  author={Mei, Hongwei and Wang, Rui and Yong, Jiongmin},
  journal={arXiv preprint arXiv:2506.09337},
  year={2025}
}

@article{Mei-Wang-Yong-2-2025,
  title={Turnpike property of a linear-quadratic optimal control problem in large horizons with regime switching {II}: nonhomogeneous cases},
  author={Mei, Hongwei and Wang, Rui and Yong, Jiongmin},
  journal={ESAIM: Control, Optimisation and Calculus of Variations},
  volume={32},
  pages={43},
  year={2026},
  publisher={EDP Sciences}
}

@article{Cohen-Jian-2025,
  title={Turnpike properties in linear quadratic {G}aussian {$N$}-player differential games},
  author={Cohen, Asaf and Jian, Jiamin},
  journal={ESAIM: Control, Optimisation and Calculus of Variations},
  volume={32},
  pages={47},
  year={2026},
  publisher={EDP Sciences}
}

@article{CLLP-2013,
  title={Long time average of mean field games with a nonlocal coupling},
  author={Cardaliaguet, Pierre and Lasry, Jean-Michel and Lions, Pierre-Louis and Porretta, Alessio},
  journal={SIAM Journal on Control and Optimization},
  volume={51},
  number={5},
  pages={3558--3591},
  year={2013},
  publisher={SIAM}
}

@article{Cardaliaguet-Porretta-2019,
  title={Long time behavior of the master equation in mean field game theory},
  author={Cardaliaguet, Pierre and Porretta, Alessio},
  journal={Analysis \& PDE},
  volume={12},
  number={6},
  pages={1397--1453},
  year={2019},
  publisher={Mathematical Sciences Publishers}
}

@article{Cirant-Porretta-2021,
  title={Long time behavior and turnpike solutions in mildly non-monotone mean field games},
  author={Cirant, Marco and Porretta, Alessio},
  journal={ESAIM: Control, Optimisation and Calculus of Variations},
  volume={27},
  pages={86},
  year={2021},
  publisher={EDP Sciences}
}

@article{Nourian-Caines-Malhame-Huang-2012,
  title={Nash, social and centralized solutions to consensus problems via mean field control theory},
  author={Nourian, Mojtaba and Caines, Peter E and Malham{\'e}, Roland P and Huang, Minyi},
  journal={IEEE Transactions on Automatic Control},
  volume={58},
  number={3},
  pages={639--653},
  year={2012},
  publisher={IEEE}
}

@book{Bensoussan-Frehse-Yam-2013,
  title={Mean field games and mean field type control theory},
  author={Bensoussan, Alain and Frehse, Jens and Yam, Phillip},
  volume={101},
  year={2013},
  publisher={Springer}
}

@article{Andersson-Djehiche-2011,
  title={A maximum principle for {SDE}s of mean-field type},
  author={Andersson, Daniel and Djehiche, Boualem},
  journal={Applied Mathematics \& Optimization},
  volume={63},
  number={3},
  pages={341--356},
  year={2011},
  publisher={Springer}
}

@article{Buckdahn-Li-Ma-2016,
  title={A stochastic maximum principle for general mean-field systems},
  author={Buckdahn, Rainer and Li, Juan and Ma, Jin},
  journal={Applied Mathematics \& Optimization},
  volume={74},
  number={3},
  pages={507--534},
  year={2016},
  publisher={Springer}
}

@article{Li-Sun-Yong-2016,
  title={Mean-field stochastic linear quadratic optimal control problems: closed-loop solvability},
  author={Li, Xun and Sun, Jingrui and Yong, Jiongmin},
  journal={Probability, Uncertainty and Quantitative Risk},
  volume={1},
  number={1},
  pages={2},
  year={2016},
  publisher={Springer}
}

@article{Xiong-Xu-2025,
  title={Mean-field stochastic linear quadratic control problem with random coefficients},
  author={Xiong, Jie and Xu, Wen},
  journal={SIAM Journal on Control and Optimization},
  volume={63},
  number={4},
  pages={3042--3060},
  year={2025},
  publisher={SIAM}
}

@article{Pham-Wei-2018,
  title={Bellman equation and viscosity solutions for mean-field stochastic control problem},
  author={Pham, Huy{\^e}n and Wei, Xiaoli},
  journal={ESAIM: Control, Optimisation and Calculus of Variations},
  volume={24},
  number={1},
  pages={437--461},
  year={2018},
  publisher={EDP Sciences}
}

@article{Li-Mou-Wu-Zhou-2025,
  title={Global Well-posedness of {H}amilton--{J}acobi Equations for Linear-Quadratic Mean Field Control Problems},
  author={Li, Mengzhen and Mou, Chenchen and Wu, Zhen and Zhou, Chao},
  journal={Journal of Dynamics and Differential Equations},
  pages={1--36},
  year={2025},
  publisher={Springer}
}

@article{Lacker-2017,
  title={Limit theory for controlled {M}cKean--{V}lasov dynamics},
  author={Lacker, Daniel},
  journal={SIAM Journal on Control and Optimization},
  volume={55},
  number={3},
  pages={1641--1672},
  year={2017},
  publisher={SIAM}
}

@article{Djete-Possamai-Tan-2022,
  title={{M}cKean--{V}lasov optimal control: limit theory and equivalence between different formulations},
  author={Djete, Mao Fabrice and Possama\"\i, Dylan and Tan, Xiaolu},
  journal={Mathematics of Operations Research},
  volume={47},
  number={4},
  pages={2891--2930},
  year={2022},
  publisher={INFORMS}
}

@article{Bayraktar-Cosso-Pham-2018,
  title={Randomized dynamic programming principle and {F}eynman--{K}ac representation for optimal control of {M}cKean--{V}lasov dynamics},
  author={Bayraktar, Erhan and Cosso, Andrea and Pham, Huy{\^e}n},
  journal={Transactions of the American Mathematical Society},
  volume={370},
  number={3},
  pages={2115--2160},
  year={2018}
}

@article{Pham-Wei-2017,
  title={Dynamic programming for optimal control of stochastic {M}cKean--{V}lasov dynamics},
  author={Pham, Huy{\^e}n and Wei, Xiaoli},
  journal={SIAM Journal on Control and Optimization},
  volume={55},
  number={2},
  pages={1069--1101},
  year={2017},
  publisher={SIAM}
}

@book {Carmona-Delarue-2018,
    AUTHOR = {Carmona, Ren\'e and Delarue, Fran\c{c}ois},
     TITLE = {Probabilistic theory of mean field games with applications.
              {II}},
    SERIES = {Probability Theory and Stochastic Modelling},
    VOLUME = {84},
      NOTE = {Mean field games with common noise and master equations},
 PUBLISHER = {Springer, Cham},
      YEAR = {2018},
     PAGES = {xxiv+697},
      ISBN = {978-3-319-56435-7; 978-3-319-56436-4},
   MRCLASS = {60-02 (35R60 49L20 60G55 60H10 60H30 91A13 91A15)},
  MRNUMBER = {3753660},
}

@article{Germain-Pham-Warin-2022,
  title={Rate of convergence for particle approximation of PDEs in {W}asserstein space},
  author={Germain, Maximilien and Pham, Huy{\^e}n and Warin, Xavier},
  journal={Journal of Applied Probability},
  volume={59},
  number={4},
  pages={992--1008},
  year={2022},
  publisher={Cambridge University Press}
}

@article{Bayraktar-Cecchin-Chakraborty-2023,
  title={Mean field control and finite agent approximation for regime-switching jump diffusions},
  author={Bayraktar, Erhan and Cecchin, Alekos and Chakraborty, Prakash},
  journal={Applied Mathematics \& Optimization},
  volume={88},
  number={2},
  pages={36},
  year={2023},
  publisher={Springer}
}

@article{Cardaliaguet-Souganidis-2023,
  title={Regularity of the value function and quantitative propagation of chaos for mean field control problems},
  author={Cardaliaguet, Pierre and Souganidis, Panagiotis E},
  journal={Nonlinear Differential Equations and Applications NoDEA},
  volume={30},
  number={2},
  pages={25},
  year={2023},
  publisher={Springer}
}

@article{Daudin-Delarue-Jackson-2024,
  title={On the optimal rate for the convergence problem in mean field control},
  author={Daudin, Samuel and Delarue, Fran{\c{c}}ois and Jackson, Joe},
  journal={Journal of Functional Analysis},
  volume={287},
  number={12},
  pages={110660},
  year={2024},
  publisher={Elsevier}
}

@article{BCEQTZ-2025,
  title={Viscosity solutions of fully second-order {HJB} equations in the {W}asserstein space},
  author={Bayraktar, Erhan and Cheung, Hang and Ekren, Ibrahim and Qiu, Jinniao and Tai, Ho Man and Zhang, Xin},
  journal={SIAM Journal on Control and Optimization},
  volume={64},
  number={4},
  pages={2250--2280},
  year={2026},
  publisher={SIAM}
}

@article{Bayraktar-Ekren-Zhang-2025,
  title={Comparison of viscosity solutions for a class of second-order {PDE}s on the {W}asserstein space},
  author={Bayraktar, Erhan and Ekren, Ibrahim and Zhang, Xin},
  journal={Communications in Partial Differential Equations},
  volume={50},
  number={4},
  pages={570--613},
  year={2025},
  publisher={Taylor \& Francis}
}

@article{Bayraktar-Ekren-He-Zhang-2025,
  title={Comparison for semi-continuous viscosity solutions for second order {PDE}s on the {W}asserstein space},
  author={Bayraktar, Erhan and Ekren, Ibrahim and He, Xihao and Zhang, Xin},
  journal={Journal of Differential Equations},
  volume={455},
  pages={113963},
  year={2026},
  publisher={Elsevier}
}

@article{Bayraktar-Ekren-Zhang-2025-convergence,
  title={Convergence rate of particle system for second-order {PDE}s on {W}asserstein space},
  author={Bayraktar, Erhan and Ekren, Ibrahim and Zhang, Xin},
  journal={SIAM Journal on Control and Optimization},
  volume={63},
  number={3},
  pages={1768--1782},
  year={2025},
  publisher={SIAM}
}

@article{Bayraktar-Zhang-2023,
  title={Solvability of infinite horizon {M}cKean--{V}lasov {FBSDE}s in mean field control problems and games},
  author={Bayraktar, Erhan and Zhang, Xin},
  journal={Applied Mathematics \& Optimization},
  volume={87},
  number={1},
  pages={13},
  year={2023},
  publisher={Springer}
}

@article{CMY25,
  title={On the long-time behavior of mean field game systems with a common noise},
  author={Cardaliaguet, Pierre and Maillet, Rapha{\"e}l and Yan, Wenbin},
  journal={arXiv preprint arXiv:2509.17443},
  year={2025}
}

\appendix

\section{Proofs of auxiliary results}
\label{s:appendx}

\subsection{Proof of Lemma \ref{l:form_BP_1}}
\label{s:proof_form_BP_1}

\begin{proof}
Let $I_{Nn} = \text{diag}[I_n, \dots, I_n]$. For $1 \les i, j \les N$, we denote by $\bm{\Upsilon}_{ij}$ the matrix obtained by exchanging the $i$-th and $j$-th block rows of $I_{Nn}$. It is clear that $\bm{\Upsilon}_{ij} = \bm{\Upsilon}_{ji}$ and $\bm{\Upsilon}_{ij} = \bm{\Upsilon}_{ij}^\top = \bm{\Upsilon}_{ij}^{-1}$ for all $1 \les i, j \les N$. 

Write $\BP_T(t) = (P_T^{ij}(t))_{1 \les i, j \les N}$, where each $P_T^{ij}(t)$ is an $n \times n$ matrix. For arbitrary $i \neq j$, define $\BP_T^{(ij)}(t) := \bm{\Upsilon}_{ij}^\top \BP_T(t) \bm{\Upsilon}_{ij}$ for all $t \in [0, T]$. Multiplying \eqref{eq:BP_1} on the left by $\bm{\Upsilon}_{ij}^\top$ and on the right by $\bm{\Upsilon}_{ij}$, we obtain
\begin{equation*}
\begin{aligned}
& \dot{\BP}_T^{(ij)}(t) - \big(\BP_T^{(ij)}(t) \BB + \cM_2(\BP_T^{(ij)}(t)) + \BS^\top \big) \big(\BR + 2 \cM_1(\BP_T^{(ij)}(t)) \big)^{-1} \\
& \hspace{0.5in} \big(\BP_T^{(ij)}(t) \BB + \cM_2(\BP_T^{(ij)}(t)) + \BS^\top \big)^\top  + \BP_T^{(ij)}(t) (\BA + \bar{\BA}) + (\BA + \bar{\BA})^\top \BP_T^{(ij)}(t) \\
& \hspace{0.5in} + 2 \cM_4(\BP_T^{(ij)}(t)) + \BQ + \bar{\BQ} = 0
\end{aligned}
\end{equation*}
with the terminal condition $\BP_T^{(ij)}(T) = 0 \in \dbS^{Nn}$, where we use the following facts for $\bm{\Psi} = \BA + \bar{\BA}, \BB, \BR, \BQ + \bar{\BQ}, \BS$, or $\BGamma + \bar{\BGamma}$:
$$\bm{\Upsilon}_{ij}^\top \bm{\Psi} \bm{\Upsilon}_{ij} = \bm{\Psi}$$
for all $i \neq j$, and
\begin{equation*}
\bm{\Upsilon}_{ij}^\top \BD^k \Be^k \bm{\Upsilon}_{ij} =  
\begin{cases}
\BD^k \Be^k, & \text{if } \, k \neq i, j, \\
\BD^j \Be^j, & \text{if } \, k = i, \\
\BD^i \Be^i, & \text{if } \, k = j, \\
\end{cases} \quad \bm{\Upsilon}_{ij}^\top (\BC^k + \bar{\BC}^k) \bm{\Upsilon}_{ij} =  
\begin{cases}
\BC^k + \bar{\BC}^k, & \text{if } \, k \neq i, j, \\
\BC^j + \bar{\BC}^j, & \text{if } \, k = i, \\
\BC^i + \bar{\BC}^i, & \text{if } \, k = j.
\end{cases} 
\end{equation*}
Thus, $\BP_T^{(ij)}(\cd)$ also satisfies the differential equation \eqref{eq:BP_1}. By the uniqueness of the solution to \eqref{eq:BP_1} established in Lemma \ref{l:unique_solvability_Riccati_N_agent}, we obtain $\bm{\Upsilon}_{ij}^\top \BP_T(\cd) \bm{\Upsilon}_{ij} = \BP_T(\cd)$ for all $i \neq j$. Hence, for all $t \in [0, T]$, the matrix $\BP_T(t) = (P_T^{ij}(t))_{1 \les i, j \les N}$ satisfies
\begin{equation*}
P_T^{ii}(t) = P_T^{jj}(t), \quad P_T^{ij}(t) = P_T^{ji}(t), \quad P_T^{ik}(t) = P_T^{jk}(t), \quad P_T^{ki}(t) = P_T^{kj}(t)  
\end{equation*}
for all $k \neq i, j$. This implies that the diagonal submatrices $\{P_T^{ii}(t): i = 1, \dots, N\}$ are identical and that all the off-diagonal submatrices $\{P_T^{ij}(t): i, j = 1, \dots, N, i\neq j\}$ are equal. Finally, since $\BP_T(t)$ is symmetric, we conclude that $P_T^{ij}(t) = (P_T^{ji}(t))^\top = P_T^{ji}(t)$ for all $i \neq j$. Set $P_T^{1,N}(\cd) := P_T^{ii}(\cd)$ for all $i = 1, \dots, N$, and $P_T^{2,N}(\cd) := P_T^{ij}(\cd)$ for all $i, j = 1, \dots, N$ with $i \neq j$. The desired characterization of $\BP_T(\cd)$ follows. 
\end{proof}

\subsection{Proof of Lemma \ref{l:form_BP_2}}
\label{s:proof_form_BP_2}

\begin{proof}
The proof is similar to that of Lemma \ref{l:form_BP_1}. For $t \in [0, T]$, write $\Bp_T(t) = (\wt{p}_T^{i}(t))_{1 \les i \les N}$, where each $\wt{p}_T^{i}(t)$ is a vector in $\dbR^n$. In addition, for arbitrary $i \neq j$, define $\Bp_T^{(ij)}(t) := \bm{\Upsilon}_{ij} \Bp_T(t)$ for all $t \in [0, T]$. Multiplying the ODE \eqref{eq:BP_2} on the left by $\bm{\Upsilon}_{ij}$ gives
\begin{equation*}
\begin{aligned}
& \dot{\Bp}_T^{(ij)}(t) - \big(\BP_T(t) \BB + \cM_2(\BP_T(t)) + \BS^\top \big) \big(\BR + 2 \cM_1(\BP_T(t)) \big)^{-1} \big((\Bp_T^{(ij)}(t))^\top \BB + \cM_3(\BP_T(t)) + \Br^\top \big)^\top \\
& \hspace{0.5in} + \BP_T(t) \Bb + (\BA + \bar{\BA})^\top \Bp_T^{(ij)}(t) + \big(\cM_5(\BP_T(t)) \big)^\top + \Bq = 0
\end{aligned}    
\end{equation*}
with the terminal condition $\Bp_T^{(ij)}(T) = 0 \in \dbR^{Nn}$, where we use Lemma \ref{l:form_BP_1} and the fact that $\bm{\Upsilon}_{ij} \bm{\Psi} = \bm{\Psi}$ for $\bm{\Psi} = \Br, \Bb, \Bq$, or $\bm{\gamma}$ and all $i \neq j$,
\begin{equation*}
\bm{\Upsilon}_{ij} (\Be^k)^\top (\BD^k)^\top \BP_T(t) \bm{\sigma}^k =  
\begin{cases}
(\Be^k)^\top (\BD^k)^\top \BP_T(t) \bm{\sigma}^k, & \text{if } \, k \neq i, j, \\
(\Be^j)^\top (\BD^j)^\top \BP_T(t) \bm{\sigma}^j, & \text{if } \, k = i, \\
(\Be^i)^\top (\BD^i)^\top \BP_T(t) \bm{\sigma}^i, & \text{if } \, k = j, \\
\end{cases}
\end{equation*}
and
\begin{equation*}
\bm{\Upsilon}_{ij} (\BC^k + \bar{\BC}^k)^\top \BP_T(t) \bm{\sigma}^k =  
\begin{cases}
(\BC^k + \bar{\BC}^k)^\top \BP_T(t) \bm{\sigma}^k, & \text{if } \, k \neq i, j, \\
(\BC^j + \bar{\BC}^j)^\top \BP_T(t) \bm{\sigma}^j, & \text{if } \, k = i, \\
(\BC^i + \bar{\BC}^i)^\top \BP_T(t) \bm{\sigma}^i, & \text{if } \, k = j.
\end{cases} 
\end{equation*}
This implies that $\Bp_T^{(ij)}(\cd)$ is also a solution to the ODE \eqref{eq:BP_2}. By the unique solvability of \eqref{eq:BP_2} in Lemma \ref{l:unique_solvability_Riccati_N_agent}, we obtain $\bm{\Upsilon}_{ij} \Bp_T(\cd) = \Bp_T(\cd)$ for all $i \neq j$, which gives $\wt{p}_T^i(\cd) = \wt{p}_T^j(\cd)$ for all $i \neq j$. Thus, for all $t \in [0, T]$, the subvectors $\{\wt{p}_T^i(t): i = 1, \dots, N\}$ are identical. The desired result follows by setting $p_T^{1,N}(\cd) := \wt{p}_T^i(\cd)$ for all $i = 1, \dots, N$.
\end{proof}

\subsection{Proof of Proposition \ref{p:solvability_MFC_finite_horizon}}
\label{s:proof_solvability_MFC_finite_horizon}

\begin{proof}
{\rm(i)} To avoid cumbersome notation throughout the proof, we remove the subscript $T$ from $P_T(\cd)$, $\bar{\Pi}_T(\cd)$, $p_T(\cd)$, and $\kappa_T(\cd)$. First, we prove the unique solvability of the first equation in \eqref{eq:Riccati_MC} in three steps. 

\textit{Step 1.} We show that the linear matrix-valued differential equation
\begin{equation}
\label{eq:ODE_P_transformed}
\begin{cases}
\vspace{4pt}
\displaystyle
\dot{P}(t) + P(t) \wt{A}(t) + \wt{A}(t)^\top P(t) + \wt{C}(t)^\top P(t) \wt{C}(t) + \Gamma^\top P(t) \Gamma + \wt{Q}(t) = 0, \\
P(T) = 0 \in \dbS^n,
\end{cases}
\end{equation}
admits a unique solution $P(\cd) \in C([0, T]; \dbS^n)$ if $\wt{A}(\cd)$, $\wt{C}(\cd)$, and $\wt{Q}(\cd)$ are measurable and uniformly bounded on $[0, T]$. Moreover, if $\wt{Q}(t) \in \dbS^n_{+}$ for all $t \in [0, T]$, then $P(\cd) \in C([0, T]; \dbS^n_{+})$. Since the equation \eqref{eq:ODE_P_transformed} is linear with bounded coefficients, it has a unique solution $P(\cd) \in C([0, T]; \dbS^n)$. Let $\Phi(\cd)$ be the solution to the following SDE:
\begin{equation*}
\begin{cases}
\vspace{4pt}
\displaystyle
d \Phi(t) = \wt{A}(t) \Phi(t) dt + \wt{C}(t) \Phi(t) dW(t) + \Gamma \Phi(t) d W^0(t), \\
\Phi(0) = I_n.
\end{cases}
\end{equation*}
By Theorem 6.14 in Chapter 1 of \cite{Yong-Zhou-1999-stochastic}, the above SDE admits a unique solution. Moreover, $\Phi(t)$ is invertible for all $t \in [0, T]$, $\dbP$-almost surely. Applying It\^o's formula, we obtain
\begin{equation*}
\begin{aligned}
d \big(\Phi(t)^\top P(t) \Phi(t) \big) &= \Phi(t)^\top \big(\wt{C}(t)^\top P(t) + P(t) \wt{C}(t) \big) \Phi(t) d W(t) \\
& \hspace{0.5in} + \Phi(t)^\top \big(\Gamma^\top P(t) + P(t) \Gamma \big) \Phi(t) d W^0(t) - \Phi(t)^\top \wt{Q}(t) \Phi(t) dt.
\end{aligned}
\end{equation*}
Consequently, we have
$$P(t) = \mathbb E \Big[\int_t^T \big(\Phi(t)^\top \big)^{-1} \Phi(s)^\top \wt{Q}(s) \Phi(s) \Phi(t)^{-1} ds \Big],$$
and $P(t) \in \dbS^n_{+}$ if $\wt{Q}(t) \in \dbS^n_{+}$.

\textit{Step 2.} Next, we claim that the first equation in \eqref{eq:Riccati_MC} has the same form as \eqref{eq:ODE_P_transformed}. Let
\begin{equation*}
\begin{aligned}
& \wt{A}(t) = A + B \Th_T^*(t), \quad \wt{C}(t) = C + D \Th_T^*(t), \\
& \wt{Q}(t) = \big(\Th_T^*(t) + R^{-1} S\big)^\top R \big(\Th_T^*(t) + R^{-1} S\big) + Q - S^\top R^{-1} S,
\end{aligned}
\end{equation*}
where $\Th_T^*(t)$ is defined in \eqref{eq:theta_T_star}. From the definition of $\wt{A}(t)$, $\wt{C}(t)$ and $\wt{Q}(t)$, we have
\begin{equation*}
- (R+D^\top P(t) D) \Th_T^*(t) = B^\top P(t) + D^\top P(t) C + S,
\end{equation*}
and
\begin{equation*}
\begin{aligned}
- \dot{P}(t) &= P(t) \big(\wt{A}(t) - B \Th_T^*(t) \big) + \big(\wt{A}(t) - B \Th_T^*(t) \big)^\top P(t) + \Gamma^\top P(t) \Gamma \\
& \hspace{0.3in} + \big(\wt{C}(t) - D \Th_T^*(t) \big)^\top P(t) \big(\wt{C}(t) - D \Th_T^*(t) \big) + Q - \Th_T^*(t)^\top (R+D^\top P(t) D) \Th_T^*(t) \\
&= P(t) \wt{A}(t) + \wt{A}(t)^\top P(t) + + \wt{C}(t)^\top P(t) \wt{C}(t) + \Gamma^\top P(t) \Gamma - \big(B^\top P(t) + D^\top P(t) \wt{C}(t) \big)^\top \Th_T^*(t) \\
& \hspace{0.3in} - \Th_T^*(t)^\top \big(B^\top P(t) + D^\top P(t) \wt{C}(t) \big) + Q - \Th_T^*(t)^\top R \Th_T^*(t).
\end{aligned}
\end{equation*}
By the definition of $\Th_T^*(t)$, we observe that
$$\big(B^\top P(t) + D^\top P(t) \wt{C}(t) \big)^\top \Th_T^*(t) = - \Th_T^*(t)^\top R \Th_T^*(t) - S^\top \Th_T^*(t),$$
which yields the following differential equation
$$- \dot{P}(t) = P(t) \wt{A}(t) + \wt{A}(t)^\top P(t) + \wt{C}(t)^\top P(t) \wt{C}(t) + \Gamma^\top P(t) \Gamma + \wt{Q}(t)$$
from the definition of $\wt{Q}(t)$.

\textit{Step 3.} We prove the unique solvability of the first equation in the system \eqref{eq:Riccati_MC} by constructing the following iterative scheme. Let $P_0(\cd)$ be the solution to
\begin{equation*}
\begin{cases}
\vspace{4pt}
\displaystyle
\dot{P}_0(t) + P_0(t) A + A^\top P_0(t) + C^\top P_0(t) C + \Gamma^\top P_0(t) \Gamma + Q = 0, \\
P_0(T) = 0.
\end{cases}
\end{equation*}
By the result in \textit{Step 1}, since $Q \in \dbS^n_{++}$, we obtain $P_0(\cd) \in C([0, T]; \dbS^n_{+})$. 
Next, inductively, for $i = 0, 1, 2, \dots$, we set
\begin{equation*}
\begin{aligned}
& \Th_i(t) =  - \cR(P_i(t))^{-1} \cS(P_i(t)), \quad \wt{A}_i(t) = A + B \Th_i(t), \quad \wt{C}_i(t) = C + D \Th_i(t), \\
& \wt{Q}_i(t) = \big(\Th_i(t) + R^{-1} S\big)^\top R \big(\Th_i(t) + R^{-1} S\big) + Q - S^\top R^{-1} S,
\end{aligned}
\end{equation*}
and let $P_{i+1}(t)$ be the solution to 
\begin{equation*}
\begin{cases}
\vspace{4pt}
\displaystyle
\dot{P}_{i+1}(t) + P_{i+1}(t) \wt{A}_i(t) + \wt{A}_i(t)^\top P_{i+1}(t) + \wt{C}_i(t)^\top P_{i+1}(t) \wt{C}_i(t) + \Gamma^\top P_{i+1}(t) \Gamma + \wt{Q}_i(t) = 0, \\
P_{i+1}(T) = 0.
\end{cases}
\end{equation*}
By \textit{Step 1}, $P_i(t) \in \dbS^n_{+}$ for all $t \in [0, T]$ and $i \ges 0$. Thus, there exists a constant $\alpha > 0$ such that $\cR(P_i(t)) = R + D^\top P_i(t) D \ges \alpha I_m$ for all $t \in [0, T]$ and $i \ges 0$. Next, we prove that the sequence $\{P_i(\cd)\}_{i=1}^{\infty}$ converges uniformly in $C([0, T]; \dbS^n_{+})$. Set
$$\Delta_i(t) := P_i(t) - P_{i+1}(t), \quad \text{and} \quad \Lambda_i(t) = \Th_{i-1}(t) - \Th_i(t)$$
for all $i \ges 1$ and $t \in [0, T]$. Subtracting the $(i+1)$-th equation from the $i$-th equation, we get
\begin{equation*}
\begin{aligned}
& \dot{\Delta}_i(t) + \Delta_i(t) \wt{A}_i(t) + \wt{A}_i(t)^\top \Delta_i(t) + \wt{C}_i(t)^\top \Delta_i(t) \wt{C}_i(t) + \Gamma^\top \Delta_i(t) \Gamma \\
& \hspace{0.5in} + P_i(t) \big(\wt{A}_{i-1}(t) - \wt{A}_i(t)\big) + \big(\wt{A}_{i-1}(t) - \wt{A}_i(t)\big)^\top P_i(t) + \wt{C}_{i-1}(t)^\top P_i(t) \wt{C}_{i-1}(t) \\
& \hspace{0.5in} - \wt{C}_i(t)^\top P_i(t) \wt{C}_i(t) + \wt{Q}_{i-1}(t) - \wt{Q}_i(t) = 0.
\end{aligned}
\end{equation*}
Note that
\begin{equation*}
\begin{aligned}
& \wt{A}_{i-1}(t) - \wt{A}_i(t) = B \Lambda_i(t), \quad \wt{C}_{i-1}(t) - \wt{C}_i(t) = D \Lambda_i(t), \\
& \wt{C}_{i-1}(t)^\top P_i(t) \wt{C}_{i-1}(t) - \wt{C}_i(t)^\top P_i(t) \wt{C}_i(t) 
= \Lambda_i(t)^\top D^\top P_i(t) D \Lambda_i(t) \\
& \hspace{0.5in} + \wt{C}_i(t)^\top P_i(t) D \Lambda_i(t) + \Lambda_i(t)^\top D^\top P_i(t) \wt{C}_i(t),
\end{aligned}
\end{equation*}
and
\begin{equation*}
\begin{aligned}
& \wt{Q}_{i-1}(t) - \wt{Q}_i(t) \\
= \ & \big(\Th_{i-1}(t) + R^{-1} S\big)^\top R \big(\Th_{i-1}(t) + R^{-1} S\big) - \big(\Th_i(t) + R^{-1} S\big)^\top R \big(\Th_i(t) + R^{-1} S\big) \\
= \ & \Lambda_i(t)^\top R \big(\Th_i(t) + R^{-1} S\big) + \big(\Th_i(t) + R^{-1} S\big)^\top R \Lambda_i(t) + \Lambda_i(t)^\top R \Lambda_i(t).
\end{aligned}
\end{equation*}
Thus, we derive that
\begin{equation}
\label{eq:ODE_Delta_i}
\begin{aligned}
& - \big[\dot{\Delta}_i(t) + \Delta_i(t) \wt{A}_i(t) + \wt{A}_i(t)^\top \Delta_i(t) + \wt{C}_i(t)^\top \Delta_i(t) \wt{C}_i(t) + \Gamma^\top \Delta_i(t) \Gamma \big] \\
= \ & \Lambda_i(t)^\top \big(R + D^\top P_i(t) D \big) \Lambda_i(t) + \Lambda_i(t)^\top \big(B^\top P_i(t) + D^\top P_i(t) \wt{C}_i(t) + R \Th_i(t) +S \big) \\
& \hspace{0.5in} + \big(B^\top P_i(t) + D^\top P_i(t) \wt{C}_i(t) + R \Th_i(t) +S \big)^\top \Lambda_i(t) \\
= \ & \Lambda_i(t)^\top \big(R + D^\top P_i(t) D \big) \Lambda_i(t)
\end{aligned}
\end{equation}
since
$$B^\top P_i(t) + D^\top P_i(t) \wt{C}_i(t) + R \Th_i(t) +S = B^\top P_i(t) + D^\top P_i(t) C + S + \big(R + D^\top P_i(t) D \big) \Th_i(t) = 0$$
by the definition of $\Th_i(t)$. As $\Delta_i(T) = 0$, another application of \textit{Step 1} gives $\Delta_i(t) \ges 0$, i.e.,
$$P_1(t) \ges P_i(t) \ges P_{i+1}(t) \ges 0, \quad \forall t \in [0, T], \, i \ges 1.$$
Hence, the sequence $\{P_i(\cd)\}_{i=1}^{\infty}$ is uniformly bounded. Consequently, there exists a constant $K > 0$ such that
\begin{equation*}
\|P_i(t)\|, \quad \|\cR(P_i(t))\| = \|R + D^\top P_i(t) D\|, \quad \|\Th_i(t)\|, \quad \|\wt{A}_i(t)\|, \quad \|\wt{C}_i(t)\|, \quad \|\wt{Q}_i(t)\|
\end{equation*}
are all bounded by $K$. We observe that
\begin{equation*}
\begin{aligned}
\Lambda_i(t) &= \big(R + D^\top P_i(t) D \big)^{-1} D^\top \Delta_{i-1}(t) D \big(R + D^\top P_{i-1}(t) D \big)^{-1} \big(B^\top P_i(t) + D^\top P_i(t) C + S\big) \\
& \hspace{0.5in} - \big(R + D^\top P_{i-1}(t) D \big)^{-1} \big(B^\top \Delta_{i-1}(t) + D^\top \Delta_{i-1}(t) C \big),
\end{aligned}
\end{equation*}
and
\begin{equation*}
\begin{aligned}
& \Th_{i-1}(t)^\top \cR(P_i(t)) \Th_{i-1}(t) - \Th_i(t)^\top \cR(P_i(t)) \Th_i(t) \\
= \ & \Lambda_i(t)^\top \cR(P_i(t)) \Lambda_i(t) + \Th_i(t)^\top \cR(P_i(t)) \Lambda_i(t) + \Lambda_i(t)^\top \cR(P_i(t)) \Th_i(t).
\end{aligned}
\end{equation*}
Then, one has
\begin{equation*}
\begin{aligned}
& \|\Lambda_i(t)^\top \cR(P_i(t)) \Lambda_i(t) \| \\
= \ & \|\Th_{i-1}(t)^\top \cR(P_i(t)) \Th_{i-1}(t) - \Th_i(t)^\top \cR(P_i(t)) \Th_i(t) - \Th_i(t)^\top \cR(P_i(t)) \Lambda_i(t) \\
& \hspace{0.5in} - \Lambda_i(t)^\top \cR(P_i(t)) \Th_i(t)\| \\
\les \ & K \big(\|\Th_{i-1}(t)\| + \|\Th_i(t)\| \big) \big(\|\Th_{i-1}(t) - \Th_i(t)\| \big) + K \|\Delta_{i-1}(t)\| \\
\les \ & K \|\Delta_{i-1}(t)\|
\end{aligned}
\end{equation*}
for all $t \in [0, T]$ and $i \ges 1$. From the differential equation \eqref{eq:ODE_Delta_i} and the terminal condition $\Delta_i(T) = 0 \in \dbS^n$, we obtain
\begin{equation*}
\begin{aligned}
\Delta_i(t) &= \int_t^T \big(\Delta_i(s) \wt{A}_i(s) + \wt{A}_i(s)^\top \Delta_i(s) + \wt{C}_i(s)^\top \Delta_i(s) \wt{C}_i(s) \\
& \hspace{0.5in} + \Gamma^\top \Delta_i(s) \Gamma + \Lambda_i(s)^\top \cR(P_i(s)) \Lambda_i(s) \big) ds,
\end{aligned}
\end{equation*}
which yields the following estimate:
\begin{equation*}
\|\Delta_i(t)\| \les \int_t^T K \big( \|\Delta_i(s)\| + \|\Delta_{i-1}(s)\| \big) ds, \quad \forall t \in [0, T], \, i \ges 1.
\end{equation*}
Let $\tau = T-t$ and $\wt{\Delta}_i(t) := \Delta_i(T-t)$ for all $t \in [0, T]$ and $i \ges 1$. After a change of variables, we have
\begin{equation*}
\|\wt{\Delta}_i(\tau)\| \les \int_0^\tau K \big( \|\wt{\Delta}_i(r)\| + \|\wt{\Delta}_{i-1}(r)\| \big) dr, \quad \forall \tau \in [0, T], \, i \ges 1.
\end{equation*}
Thus, by Gr\"onwall's inequality, 
\begin{equation*}
\|\wt{\Delta}_i(\tau)\| \les e^{K \tau} \int_0^\tau K \|\wt{\Delta}_{i-1}(r)\| dr.
\end{equation*}
By induction, we derive the following estimate:
\begin{equation*}
\|\wt{\Delta}_i(\tau)\| \les \max_{s \in [0, T]} \|\wt{\Delta}_0(s)\| \frac{(K\tau e^{K\tau})^i}{i !}, \quad \forall \tau \in [0, T], \, i \ges 1,
\end{equation*}
which implies that the sequence $\{P_i(\cd)\}_{i=1}^{\infty}$ is uniformly Cauchy. Therefore, the sequence $\{P_i(\cd)\}_{i=1}^{\infty}$ converges uniformly in $C([0, T]; \dbS^n_{+})$. Denote by $P(\cd)$ the limit of $\{P_i(\cd)\}_{i=1}^{\infty}$. Then, by \textit{Step 2}, $P(\cd)$ is the unique solution to the first equation in \eqref{eq:Riccati_MC}. Moreover, $\cR(P(t)) \ges \alpha I_m$ for all $t \in [0, T]$.

Next, given $P(\cd) \in C([0, T]; \dbS^n_{+})$, we prove the unique solvability of the second equation in the system \eqref{eq:Riccati_MC}. We first rewrite this equation as follows
\begin{equation*}
\begin{aligned}
0 &= \dot{\bar{\Pi}}(t) + \bar{\Pi}(t) \big[\h{A} - B\big(R^{-1} S + (R + D^\top P(t) D)^{-1} D^\top P(t) (\h{C} - DR^{-1} S) \big) \big] \\
& \hspace{0.5in} + \big[\h{A} - B\big(R^{-1} S + (R + D^\top P(t) D)^{-1} D^\top P(t) (\h{C} - DR^{-1} S) \big) \big]^\top \bar{\Pi}(t) \\
& \hspace{0.5in} + \h{\Gamma}^\top \bar{\Pi}(t) \h{\Gamma} - \bar{\Pi}(t) B (R + D^\top P(t) D)^{-1} B^\top \bar{\Pi}(t) + \h{Q} - SR^{-1} S \\
& \hspace{0.5in} + (\h{C} - DR^{-1} S)^\top \big[P(t) - P(t) D (R + D^\top P(t) D)^{-1} D^\top P(t)\big] (\h{C} - DR^{-1} S),
\end{aligned}
\end{equation*}
which is of the same form as the differential equation satisfied by $P(\cd)$. Note that 
$\cR(P(t)) = R + D^\top P(t) D \in \dbS^m_{++}$, $\h{Q} - SR^{-1} S \in \dbS^n_{++}$
and
\begin{equation*}
\begin{aligned}
& P(t) - P(t) D (R + D^\top P(t) D)^{-1} D^\top P(t) \\
= \ & P(t)^{\frac{1}{2}} \big[I_n - P(t)^{\frac{1}{2}} D R^{-\frac{1}{2}} \big(I_n + R^{-\frac{1}{2}} D^\top P(t)^{\frac{1}{2}} P(t)^{\frac{1}{2}} D  R^{-\frac{1}{2}} \big)^{-1}  R^{-\frac{1}{2}} D^\top P(t)^{\frac{1}{2}} \big] P(t)^{\frac{1}{2}} \\
:= \ & P(t)^{\frac{1}{2}} \big[I_n - \wt{P}(t) \big(I_n + \wt{P}(t)^\top \wt{P}(t) \big)^{-1} \wt{P}(t)^\top \big] P(t)^{\frac{1}{2}} \\
= \ & P(t)^{\frac{1}{2}} \big(I_n + \wt{P}(t) \wt{P}(t)^\top \big)^{-1} P(t)^{\frac{1}{2}} \\
= \ & P(t)^{\frac{1}{2}} \big(I_n + P(t)^{\frac{1}{2}} D R^{-1} D^\top P(t)^{\frac{1}{2}} \big)^{-1} P(t)^{\frac{1}{2}} \ges 0,
\end{aligned}
\end{equation*}
where we use the fact that
$$I_n - \wt{P}(t) \big(I_n + \wt{P}(t)^\top \wt{P}(t) \big)^{-1} \wt{P}(t)^\top = \big(I_n + \wt{P}(t) \wt{P}(t)^\top \big)^{-1}.$$
Then, by the previous result, we establish the unique solvability of the second equation in \eqref{eq:Riccati_MC}, and we obtain $\bar{\Pi}(\cd) \in C([0, T]; \dbS^n_{+})$.

Given $P(\cd), \bar{\Pi}(\cd) \in C([0, T]; \dbS^n_{+})$, the equation for $p(\cd)$ is a linear ODE, so existence and uniqueness follow from standard results in the theory of linear differential equations. Finally, once $P(\cd), \bar{\Pi}(\cd) \in C([0, T]; \dbS^n_{+})$ and $p(\cd) \in C([0, T]; \dbR^n)$ are given, the last equation in \eqref{eq:Riccati_MC} yields an explicit expression for $\kappa(\cd)$ by integration.
\ms 

{\rm (ii)} The optimality of $u_T(\cd)$ in \eqref{eq:optimal_control_MFC} can be proved by applying a completion-of-squares technique to the cost of \textbf{Problem (MFC)$^T$}; see, for example, Theorem 4.1 in \cite{Yong-2013} or Theorem 5.2 in \cite{Sun-2017} for linear-quadratic mean field control problems. Substituting $u_T(\cd)$ into the dynamics \eqref{eq:state_N_infinity} yields the optimal path governed by the closed-loop system \eqref{eq:optimal_path_MFC}.
\ms

{\rm(iii)} The result follows directly from (i), (ii), and the fact that $U_T(\mu_0) := U_T(0, \mu_0)$, where $U_T(\cd, \cd)$ has the form \eqref{eq:value_ansatz_MFC} and solves the infinite-dimensional HJ equation \eqref{eq:master_equation_MC}.
\end{proof}

\subsection{Proof of Lemma \ref{l:stability_homo_system}}
\label{s:proof_stability_homo_system}

\begin{proof}
First, given $u(\cd) \in \sU[0, \infty)$, a standard argument using the contraction mapping theorem shows that the homogeneous system $[A, \bar A, C, \bar C, \Gamma, \bar{\Gamma}; B, D]$ admits a unique solution $X^\text{Hom}(\cd) \in \sX_{\text{loc}}[0, \infty)$. Taking the conditional expectation of \eqref{eq:state_MFC_homo}, we obtain 
\begin{equation*}
\begin{cases}
\vspace{4pt}
\displaystyle
d \dbE \big[X^\text{Hom}(t)|\cF_t^0 \big] = \big\{(A + \bar A) \dbE \big[X^\text{Hom}(t)|\cF_t^0 \big] + B \dbE \big[u(t)|\cF_t^0 \big] \big\} dt \\
\hspace{1.4in} + (\Gamma + \bar{\Gamma}) \dbE \big[X^\text{Hom}(t)|\cF_t^0 \big] d W^0(t), \quad t \ges 0 \\
\dbE \big[X^\text{Hom}(0)|\cF_0^0 \big] = \dbE[\xi].
\end{cases}
\end{equation*}
Under Assumption \textnormal{\textbf{(H2)}}, the system $[A + \bar{A}, \Gamma + \bar{\Gamma}; B]$ is $L^2$-stabilizable. Hence, there exists a matrix $\bar{\Theta} \in \dbR^{m \times n}$ such that, for any initial state $\xi \in L^4_{\cF_0}$, the solution to
\begin{equation*}
\begin{cases}
\vspace{4pt}
\displaystyle
d \dbE \big[X^\text{Hom}(t)|\cF_t^0 \big] = (A + \bar A + B\bar{\Theta}) \dbE \big[X^\text{Hom}(t)|\cF_t^0 \big] dt + (\Gamma + \bar{\Gamma}) \dbE \big[X^\text{Hom}(t)|\cF_t^0 \big] d W^0(t), \quad t \ges 0 \\
\dbE \big[X^\text{Hom}(0)|\cF_0^0 \big] = \dbE[\xi]
\end{cases}
\end{equation*}
satisfies
$$\dbE \Big[\int_0^\infty \big|\dbE \big[X^\text{Hom}(t)|\cF_t^0 \big] \big|^2 dt \Big] < \infty.$$
Similarly, since the controlled system $[A, C, \Gamma; B, D]$ is $L^2$-stabilizable under Assumption \textnormal{\textbf{(H2)}}, there exists a matrix $\Th \in \dbR^{m \times n}$ such that, for any initial state $\xi \in L^4_{\cF_0}$, the system denoted by $[A + B \Theta, C+ D\Theta, \Gamma]$
\begin{equation*}
\begin{cases}
\vspace{4pt}
\displaystyle
d X(t) = (A + B\Th) X(t) dt + (C+ D \Th)X(t) dW(t) + \Gamma X(t) dW^0(t), \quad t \ges 0 \\
X(0) = \xi,
\end{cases}
\end{equation*}
is $L^2$-globally integrable. By an argument similar to that in Proposition 3.5 of \cite{Huang-Li-Yong-2015}, the following Lyapunov equation, with $I_n \in \dbS^{n}_{++}$,
$$P(A + B \Theta) + (A + B \Theta)^\top P + (C + D \Theta)^\top P (C + D \Theta) + \Gamma^\top P \Gamma + I_n = 0$$
admits a solution $P \in \dbS^n_{++}$. Let the control $u(\cd)$ in \eqref{eq:state_MFC_homo} take the form
\begin{equation*}
u(t) = \Theta \big(X^\text{Hom}(t) - \dbE \big[X^\text{Hom}(t)|\cF_t^0 \big] \big) + \bar \Theta \dbE \big[X^\text{Hom}(t)|\cF_t^0 \big], \quad t \ges 0,
\end{equation*}
then we arrive at the following homogeneous closed-loop system 
\begin{equation*}
\begin{cases}
\vspace{4pt}
\displaystyle
d X^\text{Hom}(t) = \big\{(A+B\Theta) X^\text{Hom}(t) + \left[\bar A + B(\bar \Theta - \Theta)\right] \dbE \big[X^\text{Hom}(t)|\cF_t^0 \big] \big\} dt \\
\vspace{4pt}
\displaystyle
\hspace{0.9in} + \big\{(C+D\Theta) X^\text{Hom}(t) + \left[\bar C + D(\bar \Theta - \Theta)\right] \dbE \big[X^\text{Hom}(t)|\cF_t^0 \big] \big\} dW(t), \\
\vspace{4pt}
\displaystyle
\hspace{0.9in} + \big\{\Gamma X^\text{Hom}(t) + \bar{\Gamma} \dbE \big[X^\text{Hom}(t)|\cF_t^0 \big] \big\} d W^0(t), \quad t \ges 0, \\
X^\text{Hom}(0) = \xi.
\end{cases}
\end{equation*}
Next, we show that $X^\text{Hom}(\cd)$ and $u(\cd)$ are both $L^2$-globally integrable on $[0, \infty)$. Set
$$\check{X}^\text{Hom}(t) = X^\text{Hom}(t) - \dbE \big[X^\text{Hom}(t)|\cF_t^0 \big], \quad \forall t \ges 0.$$
Then, we have
\begin{equation*}
\begin{cases}
\vspace{4pt}
\displaystyle
d \check{X}^\text{Hom}(t) = (A+B\Theta) \check{X}^\text{Hom}(t) dt + \Gamma \check{X}^\text{Hom}(t) d W^0(t) \\
\vspace{4pt}
\displaystyle
\hspace{0.9in} + \big\{(C+D\Theta) \check{X}^\text{Hom}(t) + (C+ \bar C + D \bar \Theta) \dbE \big[X^\text{Hom}(t)|\cF_t^0 \big] \big\} dW(t),  \quad t \ges 0, \\
\check{X}^\text{Hom}(0) = \xi - \dbE[\xi].
\end{cases}
\end{equation*}
Applying It\^o's formula to $\lan P \check{X}^\text{Hom}(t), \check{X}^\text{Hom}(t)\ran$, we derive that
\begin{equation*}
\begin{aligned}
& \dbE \big[\lan P \check{X}^\text{Hom}(t), \check{X}^\text{Hom}(t)\ran \big] = \dbE[\lan P(\xi - \dbE[\xi]), \xi - \dbE[\xi] \ran] \\
& + \dbE \Big[\int_0^t \big\lan \big\{P(A + B \Theta) + (A + B \Theta)^\top P + (C + D \Theta)^\top P (C + D \Theta) + \Gamma^\top P \Gamma \big\} \check{X}^\text{Hom}(s), \check{X}^\text{Hom}(s) \big\ran ds \Big] \\
& + \dbE \Big[\int_0^t \big\lan \dbE \big[X^\text{Hom}(s)|\cF_s^0 \big], (C+ \bar C + D \bar \Theta)^\top P (C+ \bar C + D \bar \Theta) \dbE \big[X^\text{Hom}(s)|\cF_s^0 \big] \big\ran ds \Big]
\end{aligned}
\end{equation*}
since 
$$\dbE\big[\big\lan (C+D\Theta) \check{X}^\text{Hom}(t), P(C+ \bar C + D \bar \Theta) \dbE \big[X^\text{Hom}(t)|\cF_t^0 \big] \big\ran \big] = 0, \quad \forall t \ges 0.$$
It follows that
\begin{equation*}
\begin{aligned}
& \dbE \Big[\int_0^t \big|\check{X}^\text{Hom}(s) \big|^2 ds \Big] \\
= \ & \dbE[\lan P(\xi - \dbE[\xi]), \xi - \dbE[\xi] \ran] - \dbE \big[\lan P \check{X}^\text{Hom}(t), \check{X}^\text{Hom}(t)\ran \big] \\
& + \dbE \Big[\int_0^t \big\lan \dbE \big[X^\text{Hom}(s)|\cF_s^0 \big], (C+ \bar C + D \bar \Theta)^\top P (C+ \bar C + D \bar \Theta) \dbE \big[X^\text{Hom}(s)|\cF_s^0 \big] \big\ran ds \Big].
\end{aligned}
\end{equation*}
Since $P$ is positive definite, we obtain the following estimate:
\begin{equation*}
\dbE \Big[\int_0^t \big|\check{X}^\text{Hom}(s) \big|^2 ds \Big] \les \dbE[\lan P(\xi - \dbE[\xi]), \xi - \dbE[\xi] \ran] + K \dbE \Big[\int_0^t \big|\dbE \big[X^\text{Hom}(s)|\cF_s^0 \big]\big|^2 ds \Big], \quad \forall t \ges 0
\end{equation*}
for some constant $K > 0$. Letting $t \to \infty$ and using the facts that $\xi \in L^4_{\cF_0}$ and $\dbE \big[X^\text{Hom}(t)|\cF_t^0 \big]$ is $L^2$-globally integrable on $[0, \infty)$, we obtain
\begin{equation*}
\dbE \Big[\int_0^{\infty} \big|\check{X}^\text{Hom}(t) \big|^2 dt \Big] < \infty,
\end{equation*}
which then implies that $\dbE[\int_0^{\infty} |X^\text{Hom}(t)|^2 dt] < \infty$. By the linearity of $u(t)$ with respect to $X^\text{Hom}(t)$ and $\dbE \big[X^\text{Hom}(t)|\cF_t^0 \big]$ for all $t \ges 0$, we conclude that 
\begin{equation*}
\dbE \Big[\int_0^{\infty} \big(\big|X^\text{Hom}(t) \big|^2 + |u(t)|^2 \big) dt \Big] < \infty.
\end{equation*}
From Definition \ref{d:homo_system_stabilizable}, it is clear that the system $[A, \bar A, C, \bar C, \Gamma, \bar{\Gamma}; B, D]$ is MF-$L^2$-stabilizable, and thus the set $\sU_{ad}[0, \infty)$ is nonempty.
\end{proof}

\subsection{Proof of Lemma \ref{l:moment_boundedness}}
\label{s:proof_moment_boundedness}

\begin{proof}
First, by taking the conditional expectation of \eqref{eq:optimal_path_ergodic} with respect to $\cF_t^0$, we have
\begin{equation}
\label{eq:ergodic_mean_state}
\begin{cases}
\displaystyle
\vspace{4pt}
d \dbE \big[\bar X(t)|\cF_t^0 \big] = \big\{(A + \bar A + B \bar \Th^*) \dbE \big[\bar X(t)|\cF_t^0 \big] + \wt{b} \big\} dt + \big\{(\Gamma + \bar{\Gamma}) \dbE \big[\bar X(t)|\cF_t^0 \big] + \gamma \big\} dW^0(t), \\
\dbE \big[\bar{X}(0)|\cF_0^0 \big] = \dbE[\xi].
\end{cases}
\end{equation}
By Proposition \ref{p:solvability_Bellman_equation_ergodic} (i), $\bar{\Theta}^*$ is a stabilizer of \eqref{eq:homo_sde_1}, so the system $[A + \bar{A} + B \bar{\Theta}^*, \Gamma + \bar{\Gamma}]$ is $L^2$-exponentially stable. Thus, by Proposition 3.6 in \cite{Huang-Li-Yong-2015}, the Lyapunov equation
\begin{equation}
\label{eq:Lyapunov_equation_1}
P_1 (A + \bar{A} + B \bar{\Theta}^*) +(A + \bar{A} + B \bar{\Theta}^*)^\top P_1 + (\Gamma + \bar{\Gamma})^\top P_1 (\Gamma + \bar{\Gamma}) + I_n = 0
\end{equation}
admits a unique solution $P_1 \in \dbS^n_{++}$. Applying It\^o's formula and Young's inequality, we obtain
\begin{equation*}
\begin{aligned}
& \frac{d}{dt} \dbE\big[ \big\lan P_1 \dbE \big[\bar X(t)|\cF_t^0 \big], \dbE \big[\bar X(t)|\cF_t^0 \big] \big\ran \big] \\
= \ & \dbE\Big[\big\lan \big(P_1 (A + \bar{A} + B \bar{\Theta}^*) +(A + \bar{A} + B \bar{\Theta}^*)^\top P_1 + (\Gamma + \bar{\Gamma})^\top P_1 (\Gamma + \bar{\Gamma})\big) \dbE \big[\bar X(t)|\cF_t^0 \big], \dbE \big[\bar X(t)|\cF_t^0 \big] \big\ran \\
& \hspace{0.5in} + 2 \big\lan \dbE \big[\bar X(t)|\cF_t^0 \big], P_1 \wt{b} + (\Gamma + \bar{\Gamma})^\top P_1 \gamma \big\ran + \lan P_1 \gamma, \gamma \ran \Big] \\
\les \ & \dbE \Big[ - \big|\dbE \big[\bar X(t)|\cF_t^0 \big]\big|^2 + 2 \big|P_1 \wt{b} + (\Gamma + \bar{\Gamma})^\top P_1 \gamma \big| \big|\dbE \big[\bar X(t)|\cF_t^0 \big] \big| + \gamma^\top P_1 \gamma \Big] \\
\les \ & \dbE \Big[ - \frac{1}{2} \big|\dbE \big[\bar X(t)|\cF_t^0 \big]\big|^2 + 2 \big|P_1 \wt{b} + (\Gamma + \bar{\Gamma})^\top P_1 \gamma \big|^2 + \gamma^\top P_1 \gamma \Big].
\end{aligned}
\end{equation*}
Let $\beta_1$ be the largest eigenvalue of $P_1$ and set $\delta_1 = 1/(2 \beta_1)$. It follows that
\begin{equation*}
\begin{aligned}
& \frac{d}{dt} \dbE\big[ \big\lan P_1 \dbE \big[\bar X(t)|\cF_t^0 \big], \dbE \big[\bar X(t)|\cF_t^0 \big] \big\ran \big] \\
\les \ & \dbE \big[ - \delta_1 \big\lan P_1 \dbE \big[\bar X(t)|\cF_t^0 \big], \dbE \big[\bar X(t)|\cF_t^0 \big] \big\ran + 2 \big|P_1 \wt{b} + (\Gamma + \bar{\Gamma})^\top P_1 \gamma \big|^2 + \gamma^\top P_1 \gamma \big],
\end{aligned}
\end{equation*}
which implies 
$$\dbE\big[ \big\lan P_1 \dbE \big[\bar X(t)|\cF_t^0 \big], \dbE \big[\bar X(t)|\cF_t^0 \big] \big\ran \big] \les e^{-\delta_1 t} \lan P_1 \dbE[\xi], \dbE[\xi] \ran + K \int_0^t e^{-\delta_1 (t-s)} ds \les K$$
for all $t \ges 0$. Since $P_1$ is positive definite, we conclude that
$$\dbE \Big[\big|\dbE \big[\bar X(t)|\cF_t^0 \big]\big|^2 \Big] \les K, \quad \forall t \ges 0.$$
Then, by standard moment estimates for SDEs (see, e.g., estimate (D.5) in \cite{Fleming-Soner-2006}), there exists a constant $K > 0$ such that
\begin{equation*}
\begin{aligned}
\dbE\Big[\sup_{t \in [0, T]} |\bar X(t)|^4 \Big] &\les K \dbE \big[|\xi|^4 \big] + KT e^{KT} \dbE \Big[\int_0^T \Big\{|\xi|^4 + \left(\|\wt{A}_2\| + |\wt{b}|\right)^4 + \left(\|\wt{C}_2\| + |\wt{\sigma}|\right)^4 \Big\} dt \Big],
\end{aligned}
\end{equation*}
which gives that for all $T > 0$,
$$\dbE \Big[\sup_{t \in [0, T]} |\bar{X}(t)|^4 \Big] < \infty.$$

Next, we show that the second moment of $\bar{X}(\cdot)$ is uniformly bounded on $[0,\infty)$. By Proposition \ref{p:solvability_Bellman_equation_ergodic} (i) again, $\Theta^*$ is a stabilizer of the SDE \eqref{eq:homo_sde_2}. Thus, the homogeneous system $[\wt{A}_1, \wt{C}_1, \Gamma]$ is $L^2$-exponentially stable, which implies that the following Lyapunov equation 
\begin{equation}
\label{eq:Lyapunov_equation_2}
P_2 \wt{A}_1 + \wt{A}_1^\top P_2 + \wt{C}_1^\top P_2 \wt{C}_1 + \Gamma^\top P_2 \Gamma + I_n = 0
\end{equation}
admits a unique solution $P_2 \in \dbS^n_{++}$. Then, applying It\^o's formula, we derive
\begin{equation*}
\begin{aligned}
& \frac{d}{dt} \dbE\left[\lan P_2\bar{X}(t), \bar{X}(t) \right] \\
= \ & \dbE\Big[ \big\lan \big(P_2\wt{A}_1 + \wt{A}_1^\top P_2 + \wt{C}_1^\top P_2 \wt{C}_1 + \Gamma^\top P_2 \Gamma \big) \bar{X}(t), \bar{X}(t) \big\ran + \big\lan \big( P_2 \wt{A}_2 + \wt{A}_2^\top P_2 + \wt{C}_2^\top P_2 \wt{C}_2 + \wt{C}_2^\top P_2 \wt{C}_1 \\
& \hspace{0.5in}  + \wt{C}_1^\top P_2 \wt{C}_2 + \bar{\Gamma}^\top P_2 \Gamma + \Gamma^\top P_2 \bar{\Gamma} + \bar{\Gamma}^\top P_2 \bar{\Gamma} \big) \dbE \big[\bar X(t)|\cF_t^0 \big], \dbE \big[\bar X(t)|\cF_t^0 \big] \big\ran \\
& \hspace{0.5in} + 2 \big\lan P_2 \wt{b} + \wt{C}_1^\top P_2 \wt{\sigma} + \wt{C}_2^\top P_2 \wt{\sigma} + \Gamma^\top P_2 \gamma + \bar{\Gamma}^\top P_2 \gamma, \bar{X}(t) \big\ran + \lan P_2\wt{\sigma}, \wt{\sigma} \ran + \lan P_2 \gamma, \gamma \ran \Big]
\end{aligned}
\end{equation*}
by using the fact that
$$\dbE\big[ \big\lan \bar{X}(t), \dbE \big[\bar X(t)|\cF_t^0 \big] \big\ran \big] = \dbE\big[ \big\lan \dbE \big[\bar X(t)|\cF_t^0 \big], \dbE \big[\bar X(t)|\cF_t^0 \big] \big\ran \big].$$
Using the above Lyapunov equation and Young's inequality, and recalling from the previous step that $\mathbb E[|\mathbb E[\bar X(t) | \cF_t^0]|^{2}]$ is uniformly bounded on $[0, \infty)$, we deduce that for some constant $K > 0$, independent of $t$, 
$$\frac{d}{dt} \dbE\left[\lan P_2 \bar{X}(t), \bar{X}(t) \right] \les \dbE \big[ - |\bar{X}(t)|^2 + K + K|\bar{X}(t)|\big] \les -\frac{1}{2}  \dbE \big[|\bar{X}(t)|^2\big] + K.$$
Let $\beta$ be the largest eigenvalue of $P_2$. Then, by the same argument as above, we conclude
\begin{equation*}
\begin{aligned}
\dbE\left[\lan P_2 \bar{X}(t), \bar{X}(t) \right] \les K e^{-\frac{1}{2\beta} t} \dbE \big[\lan P_2 \xi, \xi \ran \big] + K \int_0^t e^{-\frac{1}{2\beta}(t-s)} ds \les K, \quad \forall t \ges 0,
\end{aligned}
\end{equation*}
which implies the desired result that
$$\dbE \big[|\bar X(t)|^2 \big] \les K, \quad \forall t \ges 0$$
as $P_2$ is positive definite.
\end{proof}

\subsection{Proof of Lemma \ref{l:invariant-distribution}}
\label{s:proof_invariant-distribution}

\begin{proof}
Let $Z(t) := (\bar{X}(t), \dbE[\bar{X}(t)|\cF_t^0])$ for all $t \ges 0$. By \eqref{eq:optimal_path_ergodic} and \eqref{eq:ergodic_mean_state}, $Z(\cd)$ satisfies the following $2n$-dimensional linear SDE
\begin{equation}
\label{eq:SDE_Z_t}
\begin{aligned}
dZ(t) & = \left\{\begin{bmatrix}
\wt{A}_1 & \wt{A}_2 \\ 0 & \wt{A}_1 + \wt{A}_2
\end{bmatrix} Z(t) + \begin{bmatrix}
\wt{b} \\ \wt{b}
\end{bmatrix}\right\} dt + \left\{\begin{bmatrix}
\wt{C}_1 & \wt{C}_2 \\ 0 & 0
\end{bmatrix} Z(t) + \begin{bmatrix}
\wt{\sigma} \\ 0
\end{bmatrix}\right\} dW(t) \\
& \hspace{0.3in} + \left\{\begin{bmatrix}
\Gamma & \bar{\Gamma} \\ 0 & \Gamma + \bar{\Gamma}
\end{bmatrix} Z(t) + \begin{bmatrix}
\gamma \\ \gamma
\end{bmatrix}\right\} dW^0(t)
\end{aligned}
\end{equation}
with the initial condition $Z(0) = (\xi, \dbE[\xi])$. Since the coefficients of the above SDE are affine, globally Lipschitz, and independent of time, $Z(\cd)$ is a time-homogeneous Markov process on $\dbR^{2n}$. Denote its transition semigroup by $\{\Phi_t\}_{t \ges 0}$, where 
$$\Phi_t \varphi(z) := \dbE[\varphi(Z^z(t))]$$
for every bounded Borel measurable function $\varphi: \dbR^{2n} \to \dbR$, and $Z^z(\cd)$ denotes the solution of the SDE \eqref{eq:SDE_Z_t} with initial condition $Z^z(0) = z$. In particular, the Markov property yields
$$\dbE[\varphi(Z(s+t))|\cF_s] = \Phi_t \varphi(Z(s)), \quad \forall s, t \ges 0.$$

Let $Z_1(t) := (\bar{X}_1(t), \dbE[\bar{X}_1(t)|\cF_t^0])$ and $Z_2(t) := (\bar{X}_2(t), \dbE[\bar{X}_2(t)|\cF_t^0])$  be two solutions of the SDE \eqref{eq:SDE_Z_t} driven by the same Brownian motions $W(\cd)$ and $W^0(\cd)$, but with possibly different initial conditions. We set
$$\Delta \bar{X}(t) = \bar{X}_1(t) - \bar{X}_2(t), \quad \text{and} \quad \Delta Y(t) = \Delta \bar{X}(t) - \dbE[\Delta \bar{X}(t)|\cF_t^0], \quad \forall t \ges 0.$$
Then, it is clear that
$$d \dbE[\Delta \bar{X}(t)|\cF_t^0] = (A + \bar A + B \bar \Th^*) \dbE[\Delta \bar{X}(t)|\cF_t^0] dt + (\Gamma + \bar{\Gamma}) \dbE[\Delta \bar{X}(t)|\cF_t^0]  dW^0(t)$$
and
$$d \Delta Y(t) = \wt{A}_1 \Delta Y(t) + \big(\wt{C}_1 \Delta Y(t) + (\wt{C}_1 + \wt{C}_2) \dbE[\Delta \bar{X}(t)|\cF_t^0]   \big) dW(t) + \Gamma \Delta Y(t) dW^0(t).$$
By Proposition \ref{p:solvability_Bellman_equation_ergodic} (i), the systems $[A+\bar{A}+B\bar{\Theta}^*, \Gamma + \bar{\Gamma}]$ and $[\wt{A}_1, \wt{C}_1, \Gamma]$ are $L^2$-exponentially stable. By an argument similar to that of Lemma \ref{l:moment_boundedness}, the Lyapunov equations \eqref{eq:Lyapunov_equation_1} and \eqref{eq:Lyapunov_equation_2} admit unique stabilizing solutions $P_1 \in \dbS^n_{++}$ and $P_2 \in \dbS^n_{++}$, respectively. Applying It\^o's formula yields 
$$\frac{d}{dt} \dbE\big[\big\lan P_1 \dbE[\Delta \bar{X}(t)|\cF_t^0], \dbE[\Delta \bar{X}(t)|\cF_t^0] \big\ran \big] \les - \dbE \big[|\dbE[\Delta \bar{X}(t)|\cF_t^0]|^2 \big].$$
It follows that there exist positive constants $K$ and $\lambda$ such that
$$\dbE \big[|\dbE[\Delta \bar{X}(t)|\cF_t^0]|^2 \big] \les K e^{-\lambda t} \dbE \big[|\dbE[\Delta \bar{X}(0)|\cF_0^0]|^2 \big], \quad \forall t \ges 0.$$
Similarly, applying It\^o's formula to $\lan P_2 \Delta Y(t), \Delta Y(t) \ran$, we obtain
\begin{equation*}
\begin{aligned}
\frac{d}{dt} \dbE \big[\lan P_2 \Delta Y(t), \Delta Y(t) \ran \big] &= - \dbE \big[|\Delta Y(t)|^2\big] + 2 \dbE \big[\big\lan P_2 \wt{C}_1 \Delta Y(t), (\wt{C}_1 + \wt{C}_2) \dbE[\Delta \bar{X}(t)|\cF_t^0] \big\ran \big] \\
& \hspace{0.3in} + \dbE \big[\big\lan P_2 (\wt{C}_1 + \wt{C}_2) \dbE[\Delta \bar{X}(t)|\cF_t^0], (\wt{C}_1 + \wt{C}_2) \dbE[\Delta \bar{X}(t)|\cF_t^0] \big\ran \big].
\end{aligned}
\end{equation*}
By Young's inequality,
$$\frac{d}{dt} \dbE \big[\lan P_2 \Delta Y(t), \Delta Y(t) \ran \big] \les -\frac{1}{2} \dbE \big[|\Delta Y(t)|^2 \big] + K \dbE \big[|\dbE[\Delta \bar{X}(t)|\cF_t^0]|^2 \big].$$
Combining this inequality with the exponential estimate for $\dbE \big[|\dbE[\Delta \bar{X}(t)|\cF_t^0]|^2 \big]$, and applying Gr\"onwall's inequality, we obtain
$$\dbE \big[|\Delta Y(t)|^2\big] \les K e^{-\lambda t} \dbE \big[ |\Delta Y(0)|^2 + |\dbE[\Delta \bar{X}(0)|\cF_0^0]|^2 \big], \quad \forall t \ges 0.$$
Since $\Delta \bar{X}(t) = \Delta Y(t) + \dbE[\Delta \bar{X}(t)|\cF_t^0]$, it follows that
$$\dbE \big[|Z_1(t) - Z_2(t)|^2\big] \les K e^{-\lambda t} \dbE \big[ |Z_1(0) - Z_2(0)|^2 \big], \quad \forall t \ges 0.$$
Let $\Phi_t^*\rho$ denote the distribution at time $t$ of the augmented system \eqref{eq:SDE_Z_t} whose initial distribution is $\rho \in \cP_2(\dbR^{2n})$. The preceding synchronous coupling estimate implies that, for all $\rho_1, \rho_2 \in \cP_2(\dbR^{2n})$,
$$\cW_2^2(\Phi_t^*\rho_1, \Phi_t^*\rho_2) \les Ke^{-\lambda t} \cW_2^2(\rho_1, \rho_2).$$
We define $\rho_t := \cL(Z(t)) = \cL(\bar{X}(t), \dbE[\bar{X}(t)|\cF_t^0])$. By the Markov property and the time homogeneity of $Z(\cd)$, we have $\rho_{s+t} = \Phi_t^* \rho_s$ for all $s, t \ges 0$. Therefore,
$$\cW_2^2(\rho_{s+t}, \rho_t) = \cW_2^2(\Phi_t^*\rho_s, \Phi_t^*\rho_0) \les K e^{-\lambda t} \cW_2^2(\rho_s, \rho_0), \quad \forall s, t \ges 0.$$
By Lemma \ref{l:moment_boundedness}, we have $\sup_{t \ges 0} \dbE[|Z(t)|^2] \les K$, and thus $\sup_{s \ges 0} \cW^2(\rho_s, \rho_0) \les K$. Consequently,
$$\cW_2^2(\rho_{s+t}, \rho_t) \les K e^{-\lambda t}, \quad \forall s, t \ges 0,$$
where $K$ and $\lambda$ are independent of $s$ and $t$. Thus, $\{\rho_t\}_{t \ges 0}$ is Cauchy in $(\cP_2(\dbR^{2n}), \cW_2)$. By the completeness of this space, there exists $\rho^*_{\infty} \in \cP_2(\dbR^{2n})$ such that
$$\lim_{t \to \infty} \cW_2(\rho_t, \rho^*_{\infty}) = 0.$$
Let $\pi_1, \pi_2: \dbR^{2n} \to \dbR^n$ be the coordinate projections, i.e., $\pi_1(x, y) = x$ and $\pi_2(x, y) = y$ for all $(x, y) \in \dbR^{2n}$, and define
$$\mu^*_{\infty} := (\pi_1)_{\#}\rho^*_{\infty}, \quad \text{and} \quad \wt{\mu}^{0,*}_{\infty} := (\pi_2)_{\#} \rho^{*}_{\infty}$$
as the pushforward measures. Since $\pi_1$ and $\pi_2$ are both $1$-Lipschitz, 
$$\cW_2 \big(\cL(\bar{X}(t)), \mu^{*}_{\infty} \big) \to 0, \quad \text{and} \quad \cW_2 \big(\cL(\dbE[\bar X(t)|\cF_t^0]), \wt{\mu}^{0,*}_{\infty} \big) \to 0, \quad \hbox{as } t \to \infty.$$
This completes the proof.
\end{proof}

\subsection{Proof of Proposition \ref{p:convergence_Riccati_N_infinity}}
\label{s:proof_convergence_Riccati_N_infinity}

\begin{proof}
The proof proceeds in the following steps.

\textit{Step 1.} We first show that $P_T^{N}(\cd)$ and $\Pi_T^{N}(\cd)$ are uniformly bounded on $[0, T]$ when $N$ is sufficiently large. Consider the homogeneous social optimization problem, i.e., \textbf{Problem (SO-N)$^T$} with $b = \sigma = \gamma = q = 0$ and $r = 0$. In this case, by uniqueness of the solution to \eqref{eq:Riccati_N_agent_rescaling}, we have $p_T^N(t) = 0$ and $\kappa_T^N(t) = 0$ for all $t \in [0, T]$. By Lemma \ref{l:stability_homo_system}, under Assumption \textbf{(H2)}, the system $[A, \bar A, C, \bar C, \Gamma, \bar{\Gamma}; B, D]$ is MF-$L^2$-stabilizable. Let $(\Theta, \bar{\Theta})$ be a stabilizer of the system \eqref{eq:state_MFC_homo}. For the $i$-th agent, we consider the control 
$$u^{i, N}(t) = \Theta \big(X^{i, N}(t) - X^{(N)}(t) \big) + \bar{\Theta} X^{(N)}(t).$$
Then, the dynamics of the $i$-th agent is governed by 
\begin{equation*}
\begin{cases}
\vspace{4pt}
\displaystyle
d X^{i, N}(t) = \big\{(A+B\Theta) X^{i, N}(t) + \left[\bar A + B(\bar \Theta - \Theta)\right] X^{(N)}(t) \big\} dt \\
\vspace{4pt}
\displaystyle
\hspace{0.8in} + \big\{(C+D\Theta) X^{i, N}(t) + \left[\bar C + D(\bar \Theta - \Theta)\right] X^{(N)}(t) \big\} dW^i(t), \\
\vspace{4pt}
\displaystyle
\hspace{0.8in} + \big\{\Gamma X^{i, N}(t) + \bar{\Gamma} X^{(N)}(t) \big\} d W^0(t), \quad t \ges 0, \\
X^{i, N}(0) = \xi^i \in L^4_{\wt{\cF}_0}.
\end{cases}
\end{equation*}
Thus, the weakly coupled state average $X^{(N)}(\cd)$ satisfies
\begin{equation*}
\begin{cases}
\vspace{4pt}
\displaystyle
d X^{(N)}(t) = (A+\bar{A}+B\bar{\Theta}) X^{(N)}(t) dt + (\Gamma + \bar{\Gamma}) X^{(N)}(t) dW^0(t) \\
\vspace{4pt}
\displaystyle
\hspace{0.8in} + \frac{1}{N} \sum_{j=1}^N \big\{(C+D\Theta) X^{j, N}(t) + \left[\bar C + D(\bar \Theta - \Theta)\right] X^{(N)}(t) \big\} dW^j(t), \quad t \ges 0, \\
\displaystyle
X^{(N)}(0) = \frac{1}{N} \sum_{j=1}^N \xi^j.
\end{cases}
\end{equation*}
Since the system $[A + \bar{A} + B \bar{\Theta}, \Gamma + \bar{\Gamma}]$ is $L^2$-exponentially stable, by Proposition 3.6 in \cite{Huang-Li-Yong-2015}, the Lyapunov equation
$$P_3 (A + \bar{A} + B \bar{\Theta}) +(A + \bar{A} + B \bar{\Theta})^\top P_3 + (\Gamma + \bar{\Gamma})^\top P_3 (\Gamma + \bar{\Gamma}) + I_n = 0$$
admits a unique solution $P_3 \in \dbS^n_{++}$. Applying It\^o's formula, we obtain
\begin{equation*}
\begin{aligned}
& \frac{d}{dt} \dbE \big[\lan P_3 X^{(N)}(t), X^{(N)}(t) \ran \big] \\
= \ & - \dbE \big[|X^{(N)}(t)|^2 \big] + \frac{1}{N} \dbE \big[\lan P_3 (C+\bar{C}+D\bar{\Theta}) X^{(N)}(t), (C+\bar{C}+D\bar{\Theta}) X^{(N)}(t) \ran \big] \\
& \hspace{0.3in} + \frac{1}{N^2} \sum_{j=1}^N \dbE \big[\lan P_3 (C+D\Theta) (X^{j, N}(t) - X^{(N)}(t)), (C+D\Theta) (X^{j, N}(t) - X^{(N)}(t)) \ran \big] \\
\les \ & \Big(-1 + \frac{k_1}{N} \Big) \dbE \big[|X^{(N)}(t)|^2 \big] + \frac{k_2}{N^2} \sum_{j=1}^N \dbE \big[|X^{j, N}(t) - X^{(N)}(t)|^2 \big]
\end{aligned}
\end{equation*}
for some $k_1, k_2 > 0$. Observe that $X^{j,N}(t) - X^{(N)}(t)$ satisfies
\begin{equation*}
\begin{aligned}
d \big(X^{j,N}(t) - X^{(N)}(t) \big) &= (A+B \Theta) \big(X^{j,N}(t) - X^{(N)}(t) \big) dt + \Gamma \big(X^{j,N}(t) - X^{(N)}(t) \big) d W^0(t) \\
& \hspace{0.3in} + \big\{(C+D\Theta) \big(X^{j,N}(t) - X^{(N)}(t) \big) + (C + \bar C + D\bar{\Theta}) X^{(N)}(t) \big\} dW^j(t) \\
& \hspace{0.3in} - \frac{1}{N} \sum_{k=1}^N \big\{(C+D\Theta) X^{k, N}(t) + (\bar C + D\bar{\Theta} - D\Theta) X^{(N)}(t) \big\} dW^k(t).
\end{aligned}
\end{equation*}
Since $[A+B\Theta, C+D\Theta, \Gamma]$ is $L^2$-exponentially stable, there exists $P_4 \in \dbS^{n}_{++}$ such that
$$P_4(A+B\Theta) + (A+B\Theta)^\top P_4 + (C+D\Theta)^\top P_4 (C+D\Theta) + \Gamma^\top P_4 \Gamma + I_n = 0.$$
By It\^o's formula again, we derive that
\begin{equation*}
\begin{aligned}
& \frac{d}{dt} \frac{1}{N} \sum_{j=1}^N \dbE \big[\lan P_4 (X^{j,N}(t) - X^{(N)}(t)), X^{j,N}(t) - X^{(N)}(t) \ran \big] \\
= \ & - \frac{1}{N} \sum_{j=1}^N \dbE \big[|X^{j,N}(t) - X^{(N)}(t)|^2 \big] + \Big( 1 - \frac{1}{N} \Big) \dbE \big[\lan P_4 (C+\bar{C}+D\bar{\Theta}) X^{(N)}(t), (C+\bar{C}+D\bar{\Theta}) X^{(N)}(t) \ran \big] \\
& \hspace{0.3in} - \frac{1}{N^2} \sum_{j=1}^N \dbE \big[\lan P_4 (C+D\Theta) (X^{j, N}(t) - X^{(N)}(t)), (C+D\Theta) (X^{j, N}(t) - X^{(N)}(t)) \ran \big] \\
\les \ & - \frac{1}{N} \sum_{j=1}^N \dbE \big[|X^{j,N}(t) - X^{(N)}(t)|^2 \big] + k_3 \dbE \big[|X^{(N)}(t)|^2 \big]
\end{aligned}
\end{equation*}
for some $k_3 > 0$. Let $k_4 = 2(k_3 + 1)$, and let
$$H_1^N(t) = k_4 \dbE \big[\lan P_3 X^{(N)}(t), X^{(N)}(t) \ran \big] + \frac{1}{N} \sum_{j=1}^N \dbE \big[\lan P_4 (X^{j,N}(t) - X^{(N)}(t)), X^{j,N}(t) - X^{(N)}(t) \ran \big]$$
for all $t \in [0, T]$. Then, it is clear that
\begin{equation*}
\frac{d}{dt} H_1^N(t) \les \Big(-1 + \frac{k_2 k_4}{N} \Big) \frac{1}{N} \sum_{j=1}^N \dbE \big[|X^{j,N}(t) - X^{(N)}(t)|^2 \big] + \Big(-2 + \frac{k_1 k_4}{N} \Big) \dbE \big[|X^{(N)}(t)|^2 \big].
\end{equation*}
We choose $N_0$ sufficiently large such that $\frac{k_2k_4}{N_0} \les \frac{1}{2}$ and $\frac{k_1 k_4}{N_0} \les 1$, and let $N > N_0$. Then, integrating the above differential inequality, we have
\begin{equation*}
\int_0^T \Big(\frac{1}{2N} \sum_{j=1}^N \dbE \big[|X^{j,N}(t) - X^{(N)}(t)|^2 \big] + \dbE \big[|X^{(N)}(t)|^2 \big] \Big) dt \les H_1^N(0).
\end{equation*}
Since $P_3$ and $P_4$ are both positive definite, we deduce
$$H_1^N(0) \les k_5 \Big(\frac{1}{N} \sum_{j=1}^N \dbE \big[|X^{j,N}(0) - X^{(N)}(0)|^2 \big] + \dbE \big[|X^{(N)}(0)|^2 \big] \Big) = \frac{k_5}{N} \sum_{j=1}^N \dbE[|\xi^j|^2],$$
which implies that
$$\int_0^T \Big(\frac{1}{N} \sum_{j=1}^N \dbE \big[|X^{j,N}(t) - X^{(N)}(t)|^2 \big] + \dbE \big[|X^{(N)}(t)|^2 \big] \Big) dt \les 2 H_1^N(0) \les \frac{2 k_5}{N} \sum_{j=1}^N \dbE[|\xi^j|^2].$$
Hence, by the definition of $u^{i, N}(\cd)$, there exists $K > 0$ such that
$$\int_0^T \frac{1}{N} \sum_{j=1}^N \dbE \big[|X^{j,N}(t)|^2 + |u^{j,N}(t)|^2 \big] dt \les \frac{K}{N} \sum_{j=1}^N \dbE[|\xi^j|^2].$$
Consequently, for $N > N_0$, the value function $V^{\textnormal{soc}}_T(\cdot)$ satisfies
$$0 \les V^{\textnormal{soc}}_T(\Bx) \les \frac{K}{N} \sum_{j=1}^N |x^j|^2.$$
Using the explicit form of $V^{\textnormal{soc}}_T(\cdot)$ in \eqref{eq:value_function_SO-N} and taking $x^1 = x$ and $x^j = 0$ for all $j \neq 1$, we obtain
$$\frac{1}{N} x^\top P_T^N(0) x \les \frac{K}{N} |x|^2,$$
which yields $0 \les P_T^N(0) \les K I_n$. Replacing the initial time by an arbitrary $t$, we also have $0 \les P_T^N(t) \les K I_n$ for all $t \in [0, T]$. Thus, $\cR(P_T^N(t)) = R + D^\top P_T^N(t) D > 0$, and $\cR(P_T^N(t))^{-1}$ is well-defined for all $t \in [0, T]$. Next, taking $x^j = x$ for all $j = 1, \dots, N$,
$$0 \les x^\top \Big(P_T^N(0) + \Big(1 - \frac{1}{N} \Big) \Pi_T^N(0) \Big) x \les K |x|^2.$$
Therefore, we conclude that
\begin{equation}
\label{eq:uniform_bound_P_N}
\sup_{N > N_0} \sup_{t \in [0, T]} \big(\|P_T^N(t)\| + \|\Pi_T^N(t)\| \big) \les K,
\end{equation}
where $N_0$ and $K$ are independent of $T$ and $N$.

\textit{Step 2.} Next, we derive a uniform-in-time estimate for $P_T^N(\cd) - P_T(\cd)$. For any $P_1, P_2 \in \dbS^n_{+}$, the explicit form of $\Psi_1(\cdot)$ in \eqref{eq:functions_Psi_1} gives
\begin{equation*}
\begin{aligned}
\Psi_1(P_1) - \Psi_1(P_2) & = (P_1 - P_2) (A+B \Theta(P_2)) + (A+B \Theta(P_2))^\top (P_1 - P_2) \\
& \hspace{0.3in} + (C+D \Theta(P_2))^\top (P_1 - P_2) (C+D \Theta(P_2)) + \Gamma^\top (P_1 - P_2) \Gamma \\
& \hspace{0.3in} - (\Theta(P_1) - \Theta(P_2))^\top \cR(P_1) (\Theta(P_1) - \Theta(P_2)).
\end{aligned}
\end{equation*}
Let $\Theta^*$ be given in \eqref{eq:theta_star}. We define the linear operator $\cO_1$ on $\dbS^n$ as follows:
$$\cO_1(P) = P(A+B\Theta^*) + (A+B\Theta^*)^\top P + (C+D\Theta^*)^\top P (C+D\Theta^*) + \Gamma^\top P \Gamma.$$
Similarly, for each $t \in [0, T]$, let $\Theta_T^*(t)$ be given in \eqref{eq:theta_T_star}, and we define the operator $\cO_1(t)$ by
$$\cO_1(t)(P) = P(A+B\Theta_T^*(t)) + (A+B\Theta_T^*(t))^\top P + (C+D\Theta_T^*(t))^\top P (C+D\Theta_T^*(t)) + \Gamma^\top P \Gamma.$$
Since the system $[A+B\Theta^*, C+D\Theta^*, \Gamma]$ is $L^2$-exponentially stable, there exist $K_1 > 0$ and $\lambda_1 > 0$ such that
$$\|e^{\cO_1 t}\|_{op} \les K_1 e^{-\lambda_1 t}, \quad \forall t \ges 0,$$
where $\|\cdot\|_{op}$ denotes the operator norm on $\dbS^n$. Moreover, we define $\Delta^{N, P}_T(t) = P_T^N(t) - P_T(t)$ for all $t \in [0, T]$. Then, $\Delta^{N,P}_T(\cdot)$ satisfies the following differential equation:
\begin{equation*}
\begin{aligned}
& \dot{\Delta}^{N,P}_T(t) + \cO_1 (\Delta^{N,P}_T(t)) + \cO_1(t) (\Delta^{N,P}_T(t)) - \cO_1 (\Delta^{N,P}_T(t)) \\
& \hspace{0.3in} - (\Theta_T^N(t) - \Theta_T^*(t))^\top \cR(P_T^N(t)) (\Theta_T^N(t) - \Theta_T^*(t)) + g_1 \big(N, P_T^N(t), \Pi_T^N(t) \big) = 0
\end{aligned}
\end{equation*}
with the terminal condition $\Delta^{N, P}_T(T) = 0$, where $\Theta_T^N(\cd)$ is defined in \eqref{eq:theta_T_N}. The estimate \eqref{eq:convergence_Riccati_T_infinity} in Proposition \ref{p:convergence_Riccati_T_infinity} gives
$$\|\cO_1(t) (\Delta^{N,P}_T(t)) - \cO_1 (\Delta^{N,P}_T(t))\| \les K_2 e^{-\lambda_2 (T-t)} \|\Delta^{N,P}_T(t)\|, \quad \forall t \in [0, T]$$
for $K_2, \lambda_2 > 0$. Similarly, using the estimate \eqref{eq:uniform_bound_P_N}, for $N > N_0$ we obtain
$$\|(\Theta_T^N(t) - \Theta_T^*(t))^\top \cR(P_T^N(t)) (\Theta_T^N(t) - \Theta_T^*(t))\| \les K_3 \|\Delta^{N,P}_T(t)\|^2, \quad \forall t \in [0, T]$$
for $K_3 > 0$, and there exists $K_4 > 0$ such that
$$\big\|g_1 \big(N, P_T^N(t), \Pi_T^N(t) \big) \big\| \les \frac{K_4}{N}, \quad \forall t \in [0, T].$$
Thus, we derive that, for $N > N_0$ and $t \in [0, T]$,
$$\|\Delta^{N,P}_T(t)\| \les \int_t^T K_1 e^{-\lambda_1(s - t)} \Big(K_2 e^{-\lambda_2 (T-s)} \|\Delta^{N,P}_T(s)\| + K_3 \|\Delta^{N,P}_T(s)\|^2 + \frac{K_4}{N}\Big) ds.$$
Define $\wt{\Delta}^{N,P}_T(t) = \Delta^{N,P}_T(T-t)$ for all $t \in [0, T]$. Then, 
$$\|\wt{\Delta}^{N,P}_T(t)\| \les \int_0^t K_1 e^{-\lambda_1(t - r)} \Big(K_2 e^{-\lambda_2 r} \|\wt{\Delta}^{N,P}_T(r)\| + K_3 \|\wt{\Delta}^{N,P}_T(r)\|^2 + \frac{K_4}{N}\Big) dr.$$
Multiplying both sides by $e^{\lambda_1 t}$ and letting $h_1^N(t) = \int_0^t K_1 e^{\lambda_1 r}(K_3 \|\wt{\Delta}^{N,P}_T(r)\|^2 + \frac{K_4}{N}) dr$, we obtain
$$e^{\lambda_1 t} \|\wt{\Delta}^{N,P}_T(t)\| \les \int_0^t K_1 K_2 e^{-\lambda_2 r} e^{\lambda_1 r} \|\wt{\Delta}^{N,P}_T(r)\| dr + h_1^N(t).$$
Gr\"onwall's inequality implies that 
\begin{equation*}
\begin{aligned}
e^{\lambda_1 t} \|\wt{\Delta}^{N,P}_T(t)\| \les h_1^N(t) + \int_0^t K_1 K_2 e^{-\lambda_2 r} h_1^N(r) e^{\int_r^t K_1 K_2 e^{-\lambda_2 u} du} dr \les K_5 h_1^N(t)
\end{aligned}
\end{equation*}
for some $K_5 > 0$. Thus, there exist $K_6, K_7 > 0$ such that
$$\|\wt{\Delta}^{N,P}_T(t)\| \les \int_0^t e^{-\lambda_1 (t-r)} \Big(K_6 \|\wt{\Delta}^{N,P}_T(r)\|^2 + \frac{K_7}{N} \Big) dr := h_2^N(t).$$
Observe that $h_2^N(0) = 0$, and
$$\dot{h}_2^N(t) = - \lambda_1 h_2^N(t) + K_6 \|\wt{\Delta}^{N,P}_T(t)\|^2 + \frac{K_7}{N} \les - \lambda_1 h_2^N(t) + K_6 (h_2^N(t))^2 + \frac{K_7}{N}.$$
Let $h(\cdot)$ be the solution to the scalar Riccati equation $\dot{h}(t) = - \lambda_1 h(t) + K_6 h^2(t) + \frac{K_7}{N}$ with $h(0) = 0$. By the comparison principle, we have $h_2^N(t) \les h(t)$ for all $t \in [0, T]$. Let $N \ges N_1 := \max\{N_0, \frac{4K_6 K_7}{\lambda_1^2}\}$. Then, the equation $K_6 x^2 - \lambda_1 x + \frac{K_7}{N} = 0$ admits two positive roots $x_1$ and $x_2$, and its smaller root satisfies
$$x_1 = \frac{\lambda_1 - \sqrt{\lambda_1^2 -\frac{4 K_6 K_7}{N}}}{2K_6} \les \frac{2K_7}{\lambda_1 N}.$$
Thus, we obtain $h(t) \les \frac{2K_7}{\lambda_1 N}$ for all $t \in [0, T]$ since $h(0) = 0$ and $\dot{h}(t) = K_6 (h(t) - x_1)(h(t) - x_2)$. Hence, for all $N \ges N_1$,
$$\|\wt{\Delta}^{N,P}_T(t)\| \les h_2^N(t) \les h(t) \les \frac{K}{N}, \quad \forall t \in [0, T]$$
for some $K > 0$, which is independent of $N$. It is equivalent to
\begin{equation}
\label{eq:uniform_difference_P_T}
\|P_T^N(t) - P_T(t)\| \les \frac{K}{N}, \quad \forall t \in [0, T].
\end{equation}

\textit{Step 3.} We prove the uniform-in-time estimate for $\bar{\Pi}_T^N(\cd) - \bar{\Pi}_T(\cd)$, where $\bar{\Pi}_T^N(\cd) = P_T^N(\cd) + \Pi_T^N(\cd)$. For any $P \in \dbS^n_{+}$ and $\bar{\Pi}_1, \bar{\Pi}_2 \in \dbS^n$, by the definition of $\bar{\Psi}_2(\cdot)$ in \eqref{eq:functions_bar_Psi_2}, 
\begin{equation*}
\begin{aligned}
\bar{\Psi}_2(P, \bar{\Pi}_1) - \bar{\Psi}_2(P, \bar{\Pi}_2) &= (\bar{\Pi}_1 - \bar{\Pi}_2)(A + \bar{A} + B \bar{\Theta}(P, \bar{\Pi}_2)) + A + \bar{A} + B \bar{\Theta}(P, \bar{\Pi}_2))^\top (\bar{\Pi}_1 - \bar{\Pi}_2) \\
& \hspace{0.3in} + (\Gamma + \bar{\Gamma})^\top (\bar{\Pi}_1 - \bar{\Pi}_2) (\Gamma + \bar{\Gamma}) \\
& \hspace{0.3in}  - (\bar{\Theta}(P, \bar{\Pi}_1) - \bar{\Theta}(P, \bar{\Pi}_2))^\top \cR(P) (\bar{\Theta}(P, \bar{\Pi}_1) - \bar{\Theta}(P, \bar{\Pi}_2)).
\end{aligned}
\end{equation*}
Let $\bar{\Theta}^*$ be given in \eqref{eq:theta_star}. We define the linear operator $\cO_2$ on $\dbS^n$ as follows:
$$\cO_2(P) = P(A+\bar{A}+B\bar{\Theta}^*) + (A+\bar{A}+B\bar{\Theta}^*)^\top P + (\Gamma + \bar{\Gamma})^\top P (\Gamma + \bar{\Gamma}).$$
Similarly, for each $t \in [0, T]$, let $\bar{\Theta}_T^*(t)$ be given in \eqref{eq:theta_T_star}, and we define the operator $\cO_2(t)$ by
$$\cO_2(t)(P) = P(A+\bar{A}+B\bar{\Theta}_T^*(t)) + (A+\bar{A}+B\bar{\Theta}_T^*(t))^\top P + (\Gamma + \bar{\Gamma})^\top P (\Gamma + \bar{\Gamma}).$$
Since the system $[A+\bar{A}+B\bar{\Theta}^*, \Gamma+ \bar{\Gamma}]$ is $L^2$-exponentially stable, there exist $K_8 > 0$ and $\lambda_3 > 0$ such that
$$\|e^{\cO_2 t}\|_{op} \les K_8 e^{-\lambda_3 t}, \quad \forall t \ges 0.$$
Next, we define $\Delta^{N, \bar{\Pi}}_T(t) = \bar{\Pi}_T^N(t) - \bar{\Pi}_T(t)$ for all $t \in [0, T]$. Then, $\Delta^{N,\bar{\Pi}}_T(\cdot)$ satisfies the following differential equation:
\begin{equation*}
\begin{aligned}
& \dot{\Delta}^{N,\bar{\Pi}}_T(t) + \cO_2 (\Delta^{N,\bar{\Pi}}_T(t)) + \Delta^{N,\bar{\Pi}}_T(t) B(\bar{\Theta}_T^*(t) - \bar{\Theta}^*) + (\bar{\Theta}_T^*(t) - \bar{\Theta}^*)^\top B^\top \Delta^{N,\bar{\Pi}}_T(t) \\
& \hspace{0.3in} - (\bar{\Theta}(P_T(t), \bar{\Pi}_T^N(t)) - \bar{\Theta}_T^*(t))^\top \cR(P_T(t)) (\bar{\Theta}(P_T(t), \bar{\Pi}_T^N(t)) - \bar{\Theta}_T^*(t)) \\
& \hspace{0.3in} + \bar{\Psi}_2(P_T^N(t), \bar{\Pi}_T^N(t)) - \bar{\Psi}_2(P_T(t), \bar{\Pi}_T^N(t)) + g_1 \big(N, P_T^N(t), \Pi_T^N(t) \big) + g_2 \big(N, P_T^N(t), \Pi_T^N(t) \big) = 0
\end{aligned}
\end{equation*}
with the terminal condition $\Delta_T^{N, \bar{\Pi}}(T) = 0$, where $\bar{\Theta}_T^N(\cd)$ is given in \eqref{eq:theta_T_N}. By the estimates \eqref{eq:uniform_bound_P_N} and \eqref{eq:uniform_difference_P_T}, we have
$$\big\|g_1 \big(N, P_T^N(t), \Pi_T^N(t) \big) + g_2 \big(N, P_T^N(t), \Pi_T^N(t) \big) \big\| \les \frac{K_9}{N}, \quad \forall t \in [0, T],$$
and 
$$\|\bar{\Psi}_2(P_T^N(t), \bar{\Pi}_T^N(t)) - \bar{\Psi}_2(P_T(t), \bar{\Pi}_T^N(t))\| \les K \|P_T^N(t) - P_T(t)\| \les \frac{K_{10}}{N}, \quad \forall t \in [0, T].$$
Moreover, there exists $K_{11} > 0$ such that
$$\|(\bar{\Theta}(P_T(t), \bar{\Pi}_T^N(t)) - \bar{\Theta}_T^*(t))^\top \cR(P_T(t)) (\bar{\Theta}(P_T(t), \bar{\Pi}_T^N(t)) - \bar{\Theta}_T^*(t))\| \les K_{11} \|\Delta^{N,\bar{\Pi}}_T(t) \|^2$$
for all $t \in [0, T]$. Then, by using the estimate \eqref{eq:convergence_coefficient_control} in Corollary \ref{c:convergence_coefficients_control}, we obtain
$$\|\Delta^{N,\bar{\Pi}}_T(t) B(\bar{\Theta}_T^*(t) - \bar{\Theta}^*) + (\bar{\Theta}_T^*(t) - \bar{\Theta}^*)^\top B^\top \Delta^{N,\bar{\Pi}}_T(t)\| \les K_{12} e^{-\lambda_4 (T-t)} \|\Delta^{N,\bar{\Pi}}_T(t)\|,$$
for some constants $K_{12} > 0$ and $\lambda_4 > 0$, and thus
$$\|\Delta^{N,\bar{\Pi}}_T(t)\| \les \int_t^T K_8 e^{-\lambda_3(s - t)} \Big(K_{12} e^{-\lambda_4 (T-s)} \|\Delta^{N,\bar{\Pi}}_T(s)\| + K_{11} \|\Delta^{N,\bar{\Pi}}_T(s)\|^2 + \frac{K_9 + K_{10}}{N}\Big) ds.$$
Using an argument similar to that in \textit{Step 2}, we conclude that
\begin{equation}
\label{eq:uniform_difference_Pi_T}
\|\Delta^{N,\bar{\Pi}}_T(t)\| \les \frac{K}{N}, \quad \forall t \in [0, T]
\end{equation}
for some $K > 0$, independent of $N$. By \eqref{eq:uniform_difference_P_T}, it follows that
$$\|\Pi_T^N(t) - \Pi_T(t)\| \les \|\bar{\Pi}_T^N(t) - \bar{\Pi}_T(t)\| + \|P_T^N(t) - P_T(t)\| \les \frac{K}{N}, \quad \forall t \in [0, T].$$

\textit{Step 4.} Finally, we prove a uniform-in-time estimate for $p_T^N(\cdot) - p_T(\cd)$. Recall that $p_T^N(\cd)$ is the solution to 
$$\dot{p}_T^N(t) + \Psi_3 \big(P_T^N(t), \Pi_T^N(t), p_T^N(t) \big) + g_3 \big(N, P_T^N(t), \Pi_T^N(t), p_T^N(t) \big) = 0$$
with the terminal condition $p_T^N(T) = 0$, where $\Psi_3$ and $g_3$ are defined in \eqref{eq:functions_Psi_3} and \eqref{eq:functions_g_3}, respectively. It is equivalent to 
\begin{equation*}
\begin{aligned}
& \dot{p}_T^N(t) + (A + \bar{A} + B \bar{\Theta}^*)^\top p_T^N(t) + \big(\bar{\Theta}_T^N(t) - \bar{\Theta}^* \big)^\top B^\top p_T^N(t) + \frac{1}{N} \Pi_T^N(t)B \cR(P_T^N(t))^{-1} B^\top p_T^N(t) \\
& \hspace{0.3in} + \big(C + \bar{C} + D\bar{\Theta}_T^N(t) \big)^\top P_T^N(t) \sigma + \bar{\Theta}_T^N(t)^\top r + \bar{\Pi}_T^N(t) b + \h{\Gamma}^\top \bar{\Pi}_T^N(t) \gamma + q \\
& \hspace{0.3in} + \frac{1}{N} \Pi_T^N(t)B\cR(P_T^N(t))^{-1} \big(D^\top P_T^N(t) \sigma + r \big) - \frac{1}{N} \big(\Pi_T^N(t) b + \h{\Gamma} \Pi_T^N(t) \gamma \big) = 0
\end{aligned}
\end{equation*}
with $p_T^N(T) = 0$, where $\bar{\Theta}_T^N(\cd)$ is defined in \eqref{eq:theta_T_N}. Since the system $[A+\bar{A}+B\bar{\Theta}^*, \Gamma+ \bar{\Gamma}]$ is $L^2$-exponentially stable, the matrix $A + \bar{A} + B \bar{\Theta}^*$ is stable and thus there exist positive constants $K_{13}$ and $\lambda_5$ such that
$$\|e^{(A + \bar{A} + B \bar{\Theta}^*)^\top  t}\| \les K_{13} e^{-\lambda_5 t}, \quad \forall t \ges 0.$$
Using the estimates \eqref{eq:convergence_Riccati_T_infinity}, \eqref{eq:uniform_bound_P_N}, \eqref{eq:uniform_difference_P_T}, and \eqref{eq:uniform_difference_Pi_T}, together with the integrating factor method, we obtain
\begin{equation*}
|p_T^N(t)| \les \int_t^T K e^{-\lambda_5 (s-t)} \Big\{\Big(\frac{1}{N} + e^{-\lambda(T-s)}\Big) |p_T^N(s)| + 1 \Big\} ds
\end{equation*}
for some positive constants $K$ and $\lambda$. Applying backward Gr\"onwall's inequality, we deduce the uniform boundedness of $p_T^N$: for all $N > N_1$,
\begin{equation}
\label{eq:uniform_bound_p_N}
|p_T^N(t)| \les K, \quad \forall t \in [0, T].
\end{equation}
Define $\Delta^{N, p}_T(t) = p_T^N(t) - p_T(t)$ for all $t \in [0, T]$. Then, $\Delta^{N,p}_T(\cdot)$ satisfies the following differential equation:
\begin{equation*}
\begin{aligned}
& \dot{\Delta}^{N,p}_T(t) + (A+\bar{A}+B\bar{\Theta}^*)^\top \Delta^{N, p}_T(t) + (\bar{\Theta}_T^*(t) - \bar{\Theta}^*)^\top B^\top \Delta^{N, p}_T(t) \\
& \hspace{0.3in} + \Psi_3(P_T^N(t), \Pi_T^N(t), p_T^N(t)) - \Psi_3(P_T(t), \Pi_T(t), p_T^N(t)) + g_3 \big(N, P_T^N(t), \Pi_T^N(t), p_T^N(t) \big) = 0
\end{aligned}
\end{equation*}
with the terminal condition $\Delta^{N, p}_T(T) = 0$. By the estimates \eqref{eq:uniform_bound_P_N}, \eqref{eq:uniform_difference_P_T}, \eqref{eq:uniform_difference_Pi_T}, and \eqref{eq:uniform_bound_p_N}, there exists $K_{14} > 0$ such that
$$|g_3 \big(N, P_T^N(t), \Pi_T^N(t), p_T^N(t) \big)| \les \frac{K_{14}}{N}, \quad \forall t \in [0, T]$$
and 
$$|\Psi_3(P_T^N(t), \Pi_T^N(t), p_T^N(t)) - \Psi_3(P_T(t), \Pi_T(t), p_T^N(t))| \les \frac{K_{14}}{N}, \quad \forall t \in [0, T].$$
Hence, using the integrating factor method and the estimate \eqref{eq:convergence_coefficient_control} in Corollary \ref{c:convergence_coefficients_control} again, we obtain
$$|\Delta^{N, p}_T(t)| \les \int_t^T K_{13} e^{-\lambda_5 (s-t)} \Big(K e^{-\lambda(T-s)}|\Delta^{N, p}_T(s)| + \frac{2K_{14}}{N} \Big) ds.$$
The backward Gr\"onwall inequality implies that
\begin{equation*}
|\Delta^{N,p}_T(t)| \les \frac{K}{N}, \quad \forall t \in [0, T]
\end{equation*}
for some $K > 0$, independent of $N$. This completes the proof.
\end{proof}

\end{document}